%% file: main.tex
\pdfoutput=1
\documentclass[aos,authoryear,noinfoline]{imsart}
\usepackage{amssymb, amsmath,bm,amsthm}
\usepackage{mathrsfs,url}
\usepackage{graphicx}
\usepackage[comma,sort&compress]{natbib}
\usepackage{tabularx}
\usepackage{multirow}
\usepackage{color}
\usepackage{enumitem}
\usepackage{longtable}

\newcommand{\citeasnoun}{\citet}

\newcommand{\net}{\Lambda_n}
\newcommand{\nett}{\Lambda}
\newcommand{\nettt}{\Lambda_{n}}

\newcommand{\etah}{\ensuremath{\hat{\eta}}}
\newcommand{\taut}{\ensuremath{\tilde{\tau}}}
\newcommand{\ba}{\ensuremath{m_{\alpha}}}

\newcommand{\thr}{\ensuremath{\lambda_{n}}}

\newcommand{\alphab}{\ensuremath{{\bar{\alpha}}}}
\newcommand{\qhatgb}{\ensuremath{\hat{\bm{q}}^{\mathrm{G}}}}
\newcommand{\qhatg}{\ensuremath{\hat{q}^{\mathrm{G}}}}
\newcommand{\Ahat}{\mathrm{\widehat{ARE}}}
\newcommand{\A}{\mathrm{ARE}}
\newcommand{\TT}{\ensuremath{{\widetilde{\mathrm{T}}}}}
\newcommand{\AT}{\widetilde{\mathrm{ARE}}}
\newcommand{\AD}{\mathrm{\widehat{ARE}^D}}
\newcommand{\ADD}{\mathrm{D}}
\newcommand{\AG}{\mathrm{ARE^G}}
\newcommand{\AGhat}{\mathrm{\widehat{ARE}^G}}

\newcommand{\tauh}{\ensuremath{\hat{{\tau}}}}
\newcommand{\qhatb}{\ensuremath{\hat{\bm{q}}}}
\newcommand{\qhat}{\ensuremath{\hat{{q}}}}

\newcommand{\X}{\ensuremath{{\bm{X}}}}
\newcommand{\Xm}{\ensuremath{{\bar{\X}}}}
\newcommand{\Tm}{\ensuremath{{\bar{\thetab}}}}
\newcommand{\Z}{\ensuremath{{\bm{Z}}}}
\newcommand{\Y}{\ensuremath{{\bm{Y}}}}
\newcommand{\x}{\ensuremath{{\bm{x}}}}
\newcommand{\thetab}{\ensuremath{{\bm{\theta}}}}

\newcommand{\ex}{\ensuremath{{\mathbb{E}}}}
\newcommand{\E}{\ensuremath{{\mathbb{E}}}}
\newcommand{\Phit}{\ensuremath{{\tilde{\Phi}}}}
\newcommand{\Real}{\mathbb{R}}
\newcommand{\Lt}{\ensuremath{{{\tilde{L}}}}} 
\newcommand{\lt}{\ensuremath{{{\tilde{l}}}}} 
\newcommand{\rt}{\ensuremath{{{\tilde{r}}}}} 
\newcommand{\Rt}{\ensuremath{{{\tilde{R}}}}} 
\newcommand{\bt}{\ensuremath{\tilde{b}}}
\newcommand{\thai}{\ensuremath{\theta_i(\eta,\tau)}}
\newcommand{\lamdai}{\ensuremath{\Lambda_i(\eta,\tau)}}
\newcommand{\tha}{\ensuremath{\theta_{\alpha}}}
\newcommand{\Ua}{\ensuremath{U_{\alpha}}}
\newcommand{\Va}{\ensuremath{V_{\alpha}}}
\newcommand{\ca}{\ensuremath{c_{\alpha}}}
\newcommand{\da}{\ensuremath{d_{\alpha}}}
\newcommand{\That}{\ensuremath{\hat{T}}}
\newcommand{\Ta}{\ensuremath{\That_{\alpha,n}}}

\newcommand{\Sm}{\ensuremath{\tilde{S}}}
\newcommand{\var}{\ensuremath{\mathrm{Var}}}
\newcommand{\bias}{\ensuremath{\mathrm{Bias}}}

\newcommand{\Kn}{\ensuremath{K_n}}
\newcommand{\s}{\ensuremath{\sigma}}

\newcommand{\bigo}[1]{\ensuremath{\mathcal{O}(#1)}}

\newcommand\smallo[1]{
    	{
    		\ensuremath{\mathop{}\mathopen{}{o}\mathopen{}\left(#1\right)}
    	}
    }

\DeclareMathOperator*{\argmin}{arg\,min}

\newcommand{\ben}{\begin{enumerate}}
\newcommand{\een}{\end{enumerate}}

\newcommand{\be}{\begin{eqnarray}}
\newcommand{\ee}{\end{eqnarray}}
\newcommand{\bex}{\begin{eqnarray*}}
\newcommand{\eex}{\end{eqnarray*}}
\newcommand{\beq}{\begin{equation}}

\newcommand{\eeq}{\end{equation}}
\providecommand{\abs}[1]{\lvert #1 \rvert}

\startlocaldefs
\numberwithin{equation}{section}
\theoremstyle{plain}
\newtheorem{thm}{Theorem}[section]
\newtheorem{prop}{Proposition}[section]
\newtheorem{lem}{Lemma}[section]
\newtheorem{cor}{Corollary}[section]
\endlocaldefs

\newcommand{\pr}[1]{{\color{black}{#1}}}

\renewcommand{\Pr}{\ensuremath{\mathbb{{P}}}}
\renewcommand{\sigma}{{\mathrm{\nu}}}
\newcommand{\indicator}[1]{\ensuremath{{\mathbb{I}}_{#1}}}

\begin{document}
\bibliographystyle{imsart-nameyear}

\begin{frontmatter}

\title{Efficient Empirical Bayes prediction under check	loss using Asymptotic Risk Estimates}


\runtitle{EB Prediction under check loss}

\begin{aug}
\author{\fnms{Gourab} \snm{Mukherjee}\thanksref{m1}\ead[label=e1]{gourab@usc.edu}},
\author{\fnms{Lawrence D.} \snm{Brown}\thanksref{m2}\ead[label=e2]{lbrown@wharton.upenn.edu}}
\and \\
\author{\fnms{Paat} \snm{Rusmevichientong}\thanksref{m1}\ead[label=e3]{rusmevic@usc.edu}}
\affiliation{University of Southern California\thanksmark{m1} and University of Pennsylvania\thanksmark{m2}}
\end{aug}

%

\vspace{0.2in}
\begin{abstract}
	We develop a novel Empirical Bayes methodology for prediction under check loss in high-dimensional Gaussian models. The check loss is a piecewise linear loss function having differential weights  for measuring the	amount	of	underestimation	or	overestimation.  Prediction under it differs in fundamental aspects
	from estimation or prediction under weighted-quadratic losses. Because of the nature of this loss, our inferential target is a pre-chosen quantile of the predictive distribution rather than the mean of the predictive distribution. We develop a new method for constructing uniformly efficient asymptotic risk estimates which are then minimized to produce effective linear shrinkage predictive rules. In calculating the magnitude and direction of shrinkage, our proposed predictive rules incorporate the asymmetric nature of the loss function and are shown to be asymptotically optimal. Using numerical experiments we compare the performance of our method with traditional Empirical Bayes procedures and obtain encouraging results. 
\end{abstract}

\begin{keyword}
\kwd{Shrinkage estimators}
\kwd{Empirical Bayes prediction}
\kwd{Asymptotic optimality}  
\kwd{Uniformly efficient risk estimates} 
\kwd{Oracle inequality} 
\kwd{Pin-ball loss}
\kwd{Piecewise linear loss}
\kwd{Hermite polynomials}
\end{keyword}

\end{frontmatter}

\section{Introduction}\label{sec-1}

\input{sec-1}
\section{Proof of Theorem \ref{origin.loss.are} and Explanation of the $\A$ Method}\label{sec-2}
\input{sec-2}
\section{Simulation Experiments}\label{section:experiments}
\input{sec-3}

\section{Explanations and Proofs for Estimators in $\mathcal{S}$}\label{sec-data-driven} 
\input{sec-4}

\section{Explanations and Proofs for Estimators in $\mathcal{S}^G$}\label{sec-grand-mean} 
\input{sec-5}

\section{Discussion}
\input{discussion}

\begin{center}
\large{\textbf{\textrm{APPENDIX}}}
\end{center}

Appendices ~\ref{append-origin}, \ref{append-data-driven} and \ref{append-grand-mean} associated with Sections~\ref{sec-2}, \ref{sec-data-driven} and \ref{sec-grand-mean} are provided here. Appendix~\ref{real-data} exhibits online retail data based numerical experiments which shows encouraging performance of our proposed methodology when applied to the multivariate newsvendor problem. A glossary of all the notations as well as a list of all basic results used in the paper are also presented here.  

\smallskip
\appendix
\input{appendix}

\bibliography{mybib,append}

\end{document}

%% file: sec-1.tex
We consider Empirical Bayes (EB) prediction under check loss (see Chapter~11.2.3 of \citealp{press2009subjective} and \citealp{koenker1978regression}) in high-dimensional Gaussian models. 
The check loss (sometimes also referred to as tick loss) is linear in	the	amount	of	underestimation	or	overestimation and the	weights	for	these two linear segments differ. The asymmetric check loss function often arises in modern business problems as well as in medical and scientific research \citep{koenker2005quantile}.  Statistical prediction analysis under asymmetric loss functions in fixed dimensional models have been considered by  \citet{zellner1968}, \citet{AitchisonDunsmore:1976},	\citet{zellner1986bayesian} and \citet{blattberg1992estimation}. Here, we consider the multivariate	prediction problem under an agglomerative co-ordinatewise check loss as the dimensionality of the underlying Gaussian location model increases. 
In	common	with	many	other 	multivariate	problems	we	find	that	empirical	Bayes (shrinkage)	can	provide	better	performance	than	simple	coordinate-wise	 rules;	see \citet{james1961estimation}, \citet{zhang2003}, and \citet{greenshtein2009asymptotic} for	some background.	However, prediction	under the	loss	function	here	 differs	in	fundamental	aspects	from estimation	 or	prediction	under the	weighted	quadratic	losses	considered	in	most	 of	the	previous	literature.	This	necessitates	different	strategies	for	creation	of	effective	empirical	Bayes predictors.	
\par
We begin by considering a	Gaussian	hierarchical	Bayes	structure, with	unknown	hyperparameters. We	develop	an	estimate	of	the hyperparameters 	that	is	adapted	to	the	shape	of	the	concerned loss function.	This	estimate	of	the	hyperparameters	is then	substituted	in	the	Bayes formula to produce an EB predictive rule. This yields a co-ordinatewise prediction	that  we	prove is	overall	asymptotically	optimal	as	the	dimension of the problem	grows	increasingly	large.	The	hyperparameters are estimated by minimizing asymptotically efficient risk estimates. Due to the asymmetric nature of the loss function, direct construction of unbiased risk estimates which is usually done under weighted quadratic losses is difficult here. We develop a new asymptotically unbiased risk estimation method which involves	an  appropriate	use	of	Hermite	polynomial	expansions	for	the	relevant	stochastic	functions.	\citet{cai2011} 	used	 such	an	expansion 	for	a	different,	though	somewhat  related,	problem	involving	estimation	of	the	$L_1$ norm	of	an	unknown	mean	vector.		In	other 	respects	our	derivation	logically	resembles	 that	of	\citet{xie2012,xie2015}	who	constructed	empirical	Bayes	estimators	 built	from	an	unbiased	estimate	of	risk. 	However	their	problem	involved	estimation	 under	quadratic	loss,	 and	the	mathematical	formulae	they	used	 provide	exactly	unbiased	estimates  of	risk,	and	are	quite	 different	from	those we	develop.
\par
The	remainder	of	Section~\ref{sec-1}	describes	our	basic 	setup	and gives formal	statements	of	our	main  asymptotic  results.	Section~\ref{sec-2}	provides	further	details.	It	explains	the	general	mathematical	structure  of	our	asymptotic	risk	estimation	 methodology	and	sketches	the	proof	 techniques	used	 to	prove	the	main	 theorems	
about	it.	Sections~\ref{sec-data-driven}	and	\ref{sec-grand-mean}	contain	further	narrative	to	explain	the	proofs	of	the	main results,	but	technical	details	are	deferred	to	the  Appendices.	Section~\ref{section:experiments}	reports	on	some	simulations.	These	clarify	the  nature	of	our	estimator	and	provide	some	confidence	that	it	performs	well	even	 when	the	dimension of the model is	not	extremely	large.

\subsection{Basic Setup}
We adopt the statistical prediction analysis framework of \citeasnoun{AitchisonDunsmore:1976} and \citet{Geisser-book}. We consider a one-step, $n$ dimensional Gaussian predictive model where 
for each $i=1,\ldots,n$,  the observed past  $X_i$ and the unobserved future $Y_i$ are distributed according to a normal distribution with an unknown mean $\theta_i$; that is,
\begin{align}\label{pred.model}
X_i&=\theta_i + \sqrt{\sigma_{p,i}} \cdot \epsilon_{1,i} \quad\text{ for } i=1,2,\ldots,n\\
Y_i&=\theta_i +\sqrt{\sigma_{f,i}} \cdot \epsilon_{2,i} \quad \text{ for } i=1,2,\ldots,n~,
\end{align}
where the noise terms $\{\epsilon_{j,i}: j=1,2; i=1,\ldots,n\}$ are i.i.d. from \pr{a} standard normal distribution, and the past and future variances $\sigma_{p,i}$, $\sigma_{f,i}$ are known for all~$i$.  Note that, in multivariate notation $\X \sim N(\thetab, \bm{\Sigma}_p)$ and $\Y \sim N(\thetab, \bm{\Sigma}_f)$ 
where $\bm{\Sigma}_p$ and $\bm{\Sigma}_f$ are $n$ dimensional diagonal matrices whose $i^{\text{th}}$ entries are $\sigma_{p,i}$ and $\sigma_{f,i}$,~respectively. If the mean $\theta_i$ were known, then the observed past $X_i$ and future $Y_i$ would be independent of each~other.  

Our objective is to compute  
$\qhatb =\{\qhat_i(\X): 1\leq i \leq n\}$ based on the past data $\X$ such that $\qhatb$ optimally predicts $\Y$. 
As a convention, we use bold font to denote vectors \pr{and matrices}, while regular font denotes scalars.  For ease of exposition, we will use $~\widehat{}~$ to denote data-dependent estimates, and we will sometimes write $\qhatb$ or its univariate version $\qhat_i$ without an explicit reference to $\X$.
\par
When we predict the future $Y_i$ by $\qhat_i$, the loss corresponding to the ${i}^{\text{th}}$ coordinate is $b_i \cdot  (Y_i-\qhat_i)^+  + h_i \cdot  (\qhat_i-Y_i)^+$.  
This loss is related to the pin-ball loss function \citep{steinwart2011}, which is widely used in statistics and machine learning for estimating conditional quantiles. For each $\X = \x$, the associated predictive loss is given by
\begin{align}\label{eq:loss-temp.1}
l_i(\theta_i,\qhat_i(\x))~=~\ex_{Y_i \sim N(\theta_i \,,\,\sigma_{f,i})}  \left[ b_i   (Y_i-\qhat_i(\x))^+  + h_i  (\qhat_i(\x)-Y_i)^+ \right]~,
\end{align}
where the expectation is taken over the distribution of the future $Y_i$ only. We use the notation $N(\mu,\sigma)$ to denote a normal random variable with mean $\mu$ and variance $\sigma$. 
Since $Y_i$ is normally distributed with mean~$\theta_i$, it follows from Lemma~\ref{loss.properties} that
\begin{equation}\label{eq:loss-function}
\begin{split} 
l_i(\pr{\theta_i},\qhat_i) & = \sqrt{\sigma_{f,i}}\,(b_i+h_i) \, G(\; (\qhat_i-\theta_i) / \sqrt{\sigma_{f,i}} \, , \, b_i/(b_i+h_i) \;), \text{ where }\\
G(w,\beta) & = \phi(w) + w \Phi(w) - \beta w  \quad \text{ for } w \in \Real, \beta \in [0,1]\,,
\end{split}
\end{equation}
where $\phi(\cdot)$ and $\Phi(\cdot)$ are the standard normal PDF and CDF, respectively.
Thus, given \X, \pr{the} cumulative loss associated with the $n$ dimensional vector $\hat{\bm{q}}$ is
$$
L_n(\thetab,\qhatb) = \frac{1}{n}\sum_{i=1}^n l_i(\theta_i,\pr{\qhat_i})~.
$$
\textit{An Example: The Newsvendor problem.} As a motivation for the check loss, consider the inventory management problem of a vendor who sells a large number of products. Consider a one-period setting, where based on the observed demand $\X$ in the previous period,  the vendor must determine the stocking quantity $\qhat_i$ of each product in the next period. 
He has to balance the tradeoffs between stocking too much and incurring high inventory cost versus stocking too little and suffering lost sales.
If each unit of inventory incurs a holding cost $h_i > 0$, and each unit of \pr{lost} sale incurs a cost of $b_i  > 0$, the vendor's cost function is given by \eqref{eq:loss-temp.1}. Usually, the lost sales cost is much higher than the inventory cost leading to a highly asymmetric loss function. This problem of determining optimal stocking levels is a classical problem in the literature on inventory management \citep{arrow1951optimal,karlin_scarf_58,rudin2015big,levi2011data} and is referred to as the \textit{multivariate newsvendor problem}. In Appendix~\ref{real-data}, using a data-informed illustration on newsvendor problem we study estimation under \eqref{eq:loss-function}.
\par
\textit{Hierarchical Modeling and Predictive Risk.} 
We want to minimize the expected loss $\E_\X \left[ L_n(\thetab,\qhatb) \right]$ over the class of estimators~$\qhatb$ for {\bf all} values of $\thetab$. If $\bm{\theta}$ \pr{were known}, then by Lemma~\ref{loss.properties}, the optimal prediction for each dimension $i$ is given by $\theta_i + \sqrt{\sigma_{f,i}}\,\Phi^{-1}(b_i / (b_i + h_i))$. In absence of such knowledge,  we consider hierarchical modeling and the related Empirical Bayes (EB) approach \citep{robbins1964empirical,zhang2003}. This is a popular statistical method for combining information and conducting simultaneous inference on multiple parameters that are connected by the structure of the problem \citep{good1980some,efron1973combining,efron1973stein}.
\par
We consider the conjugate hierarchical model and put a prior distribution $\pi_{\eta,\tau}$ on each $\theta_i$, under which $\theta_1, \theta_2,\ldots,\theta_n$ are  i.i.d. from $N(\eta,\tau)$ distribution. Here, $\eta$ and $\tau$ are  the {\em unknown} location and scale hyperparameters, respectively. The {\em predictive risk} associated with our estimator~$\qhatb$ is defined~by
$$
R_n(\thetab,\qhatb)=\ex_{\X \sim N(\thetab, \bm{\Sigma}_p)} \left[ L_n(\thetab,\qhatb) \right]~,
$$
where the expectation is taken over $\X$. Note that the expectation over $\Y$ is already included in $L$ via the definition of the loss $\ell_i$.  
Because of the nature of the check loss function, our inferential target here is a pre-chosen quantile of the predictive distribution rather than the mean of the predictive distribution which is usually the object of interest in prediction under quadratic loss. By Lemma~\ref{linear.estimates}, the Bayes estimate -- the unique minimizer of the integrated Bayes risk 
$B_n(\eta,\tau)=\int  R_n(\thetab,\qhatb) \pi_{\eta,\tau}(\thetab)\, d\thetab$ -- is given for $i = 1, \ldots, n$ by
\begin{equation}
\label{eq:linear.est}
\qhat_i^{{\sf Bayes}}(\eta,\tau) =  \alpha_i(\tau) X_i +  (1-\alpha_i(\tau)) \eta +  \sqrt{ \sigma_{f,i}+ \alpha_i(\tau)\sigma_{p,i} } \,\Phi^{-1}( b_i/(b_i+h_i)),
\end{equation}
where, \pr{for all $i$, $\alpha_i(\tau) =\tau/(\tau+\sigma_{p,i})$ denotes the shrinkage factor of coordinate $i$.}

Standard parametric Empirical Bayes methods \citep{efron1973stein,stein1962confidence,morris1983parametric,lindley1962discussion} suggest using the marginal distribution of $\X$ to estimate the unknown hyperparameters. In this paper, inspired by Stein's Unbiased Risk Estimation (SURE) approach of constructing shrinkage estimators \citep{Stein1981estimation}, we consider an alternative estimation method.  Afterwards, in Section~\ref{subsec:main.results}, we show that  our method outperforms standard parametric EB methods which are based on the popular maximum likelihood and method of moments approaches.   
\par
\textit{Class of Shrinkage Estimators:} The Bayes estimates defined in \eqref{eq:linear.est} are based on the conjugate Gaussian prior and constitute a class of linear estimators \citep{Johnstone-book}. When the hyperparameters are estimated from  data, they form a class of adaptive linear estimators. Note that these estimates themselves are not linear but are derived from  linear estimators by the estimation of tuning parameters, which, in this case, correspond to the shrinkage factor $\alpha_i(\tau)$ and the direction of shrinkage $\eta$.  Motivated by the form of the Bayes estimate in \eqref{eq:linear.est}, we  study the estimation problem in the following three specific classes of shrinkage estimators:
\begin{itemize}[leftmargin=*]
\item{\bf Shrinkage governed by Origin-centric priors:} Here, $\eta=0$ and $\tau$ is estimated based on the past data $\X$.  Shrinkage here is governed by mean-zero priors.  This class of estimators is denoted by~$\mathcal{S}^0 =  \left\{ \qhatb(\tau)  ~|~  \tau \in [0,\infty] \right\}$, where for each $\tau$,
$\qhatb(\tau)  = \{ \qhat_i( \tau) : i =1,\ldots, n\}$, and for all $i$,
$$
	\qhat_i( \tau) = \alpha_i(\tau) X_i  +  \sqrt{ \sigma_{f,i}+ \alpha_i(\tau)\sigma_{p,i} } \,\Phi^{-1}\left( b_i / (b_i + h_i) \right)\,.
$$
We can generalize $\mathcal{S}^0$ by considering shrinkage based on priors with an a priori chosen location $\eta_0$. The corresponding class of shrinkage estimators
$\mathcal{S}^A(\eta_0)=\{\qhatb(\eta_0,\tau)| \tau \in [0,\infty]\}$, where $\eta_0$ is a prefixed location, consists of
$$
\qhat_i( \eta_0,\tau) = \alpha_i(\tau) X_i  +  (1-\alpha_i(\tau)) \eta_0 +  \sqrt{ \sigma_{f,i}+ \alpha_i(\tau)\sigma_{p,i} } \,\Phi^{-1}\left( b_i / (b_i + h_i)  \right)\,. 
$$
As these estimators are location equivariant \citep{lehmann1998theory} the estimation problem in $\mathcal{S}^A(\eta_0)$  for any fixed $\eta_0$ reduces to an estimation problem in $\mathcal{S}^0$. Hence, we do not discuss shrinkage classes based on a priori centric priors as separate cases.\\[-2ex]
\item{\bf Shrinkage governed by Grand Mean centric priors:} In this case, $\eta=\bar{X}_n:=n^{-1} \sum_{i=1}^n X_i$, and $\tau$ is estimated based on the past data. Shrinkage here  is governed by priors centering near the grand mean of the past $\X$. This class of estimators is denoted by $\mathcal{S}^G = \left\{ \qhatb^G(\tau) ~|~  \tau \in [0,\infty]  \right\}$, where for all $\tau \in [0,\infty]$ and $i = 1, \ldots, n$,
$$
\qhat_i^G( \tau) = \alpha_i(\tau) X_i  +  (1-\alpha_i(\tau)) \bar{X}_n +  \sqrt{ \sigma_{f,i}+ \alpha_i(\tau)\sigma_{p,i} } \,\Phi^{-1}\left( b_i / (b_i + h_i)  \right)\,. 
$$
\item{\bf Shrinkage towards a general Data-Driven location:} In the final case, we consider the general class of shrinkage estimators where both $\eta$ and $\tau$ are simultaneously estimated.  We shrink towards a data-driven location \pr{while} simultaneously optimizing the shrinkage factor; this class is denoted by $\mathcal{S} = \left\{ \qhatb(\eta,\tau) ~|~  \eta \in \Real, \tau \in [0,\infty]  \right\}$, where 
$$
\qhat_i(\eta, \tau) = \alpha_i(\tau) X_i  +  (1-\alpha_i(\tau)) \eta +  \sqrt{ \sigma_{f,i}+ \alpha_i(\tau)\sigma_{p,i} } \,\Phi^{-1}\left( b_i / (b_i + h_i)  \right)\,. 
$$
\end{itemize}

\subsection{Main Results}\label{subsec:main.results}
For ease of understanding, we first describe the results for the class  $\mathcal{S}^0$ where the direction of shrinkage is governed by mean-zero priors so that $\eta = 0$.  The results for the other cases are stated afterwards; see Section~\ref{further.results}. By \pr{definition}, estimators in $\mathcal{S}^0$ are of the form: for $i=1, \ldots, n$,
\begin{align}\label{eq:linear.est.o}
\qhat_i(\tau) = \alpha_i(\tau) X_i  +  \sqrt{ \sigma_{f,i}+ \alpha_i(\tau)\sigma_{p,i} } \,\Phi^{-1}\left( b_i/(b_i+h_i) \right)~,
\end{align} 
where $\alpha_i(\tau) = \tau / (\tau + \sigma_{p,i})$ is the shrinkage factor, and the tuning parameter $\tau$ varies from $[0,\infty]$. 
We next describe  the reasonable and mild conditions that we impose on the problem structure. These assumptions mainly facilitate the rigorousness of the theoretical proofs and can be further relaxed for practical use. 

\noindent \underline{{\bf Assumptions}}

\textit{A1. Bounded weights of the loss function.} 
To avoid degeneracies in our loss function, which can be handled easily but require separate case by case inspections, we impose the following condition on the  weights of the coordinatewise losses:
$$
	0 < \inf_i b_i/(b_i+h_i) \leq \sup_i  b_i/(b_i+h_i) < 1
	\quad \text{ and } \quad \sup_i (b_i + h_i) < \infty~.
$$
%

\textit{A2. Bounded parametric space.} We assume that average magnitude of $\thetab$ is bounded:
\begin{align}\label{assump.A3}
\limsup_{n \to \infty} \frac{1}{n} \sum_{i=1}^n |\theta_i| < \infty.
\end{align}
Note that, both A1 and A2 are benign structural assumptions necessary for clear statements of the proofs. 

\textit{A3. Sufficient historical data.} 
We assume the following upper bound on the ratio of the past to future variances for all the coordinates:
\begin{align}\label{assumption.A2}
\sup_i ~\sigma_{p,i}/\sigma_{f,i} <  1/(4 e)~.
\end{align}
To understand the implication of this assumption, consider a multi-sample prediction problem where we observe $m$ i.i.d. past data vectors from a $n$-dimensional Gaussian location model.
Using sufficiency argument in the Gaussian setup, we can reduce this multi-sample past data problem to a vector problem by averaging across the  $m$ observations. The  variance of the averaged model is proportional to $m^{-1}$, and in this case, we will have $\sigma_{p,i}/\sigma_{f,i}=m^{-1}$ for each~$i$. Therefore, if we have a lot of independent historical records, the above condition will be satisfied. As such \eqref{assumption.A2} holds if we have $11$ or more independent and identically distributed past records.
Conditions of this form are not new in the predictive literature,  as the ratio of the past to future variability controls the role of estimation accuracy in predictive models  \citep{mukherjee2012exact,George06}.  
Simulation experiments (see Section~\ref{section:experiments}) suggest that the constant in \eqref{assumption.A2} can be reduced but some condition of this form is needed.
Also, to avoid degeneracies in general calculations with the loss function (which can be easily dealt by separate case analysis), we impose very mild assumptions on the variances:
$\sup_i \sigma_{f,i} < \infty$ and $\inf_i \sigma_{p,i} > 0$.


\textit{Our Proposed Shrinkage Estimate:} The predictive risk of estimators $\qhatb(\tau)$ of the form \eqref{eq:linear.est.o} is given by $R_n(\thetab,\qhatb(\tau))=\ex_{\thetab} \left[ L_n(\thetab,\qhatb(\tau)) \right]$, where the expectation is taken over $\X \sim N(\thetab, \bm{\Sigma}_p)$.
Following \citet{Stein1981estimation}, the idea of minimizing unbiased estimates of risk to obtain efficient estimates of tuning parameters has a considerable history in statistics \citep{Kurt2000,george2012tribute,stigler1990,efron1973stein}. 
However, as shown in Equation \eqref{eq:loss-function}, 
our loss function $l(\cdot,\cdot)$ is {\em not} quadratic, so  a direct construction of unbiased risk estimates is difficult.  Instead, we approximate the risk function $ \tau \mapsto R_n(\thetab,\qhatb(\tau))$ by an {\it Asymptotic Risk Estimator} ($\A$) function $\tau \mapsto {\Ahat}_n(\tau)$, which may be {\em biased}, but it approximates the true risk function {\em uniformly  well for all $\tau$}, particularly in large dimensions. Note that $\Ahat_n(\tau)$ depends {\em only} on the observed $\X$ and $\tau$ and is not dependent on $\thetab$. The estimation procedure is fairly complicated and is built on a Hermite polynomial expansion of the risk. It is described in the next subsection (See \eqref{def.A}). Afterward, we show that our risk estimation method not only adapts to the data but also does a better job in adapting to the shape of the loss \pr{function when compared with} the widely used Empirical Bayes MLE (EBML) or method of moments (EBMM) estimates.   The main results of this paper are built on the following theorem. 

\begin{thm}[Uniform Point-wise Approximation of the Risk]\label{origin.loss.are}
Under Assumptions A1 and A3, for all $\thetab$ satisfying Assumption~A2 and for all estimates $\qhatb(\tau) \in \mathcal{S}^0$, we have
$$ \lim_{n \to \infty} a_n^8 \bigg \{\sup_{\tau \in [0,\infty ]} \ex \big( \Ahat_n(\tau)-R_n(\thetab,\qhatb(\tau))\big)^2 \bigg\}~=~ 0~, \text{ where } a_n=\log \log n$$ 
and the expectation is taken over the random variable $\X \sim N(\thetab, \Sigma_p)$.
\end{thm}
The above theorem shows that our proposed ARE method approximate the true risk in terms of mean square error uniformly well at each hyperparameter value. Next, we have a useful property of our class of predictors $\mathcal{S}^0$ which we will use along with Theorem~\ref{origin.risk.loss}. It shows that for each member of $\mathcal{S}^0$ the loss function $L_n(\thetab,\qhatb(\tau))$ uniformly concentrates around its expected value, which is the risk $R_n(\thetab, \qhatb(\tau))$.
\begin{thm}[Uniform Concentration of the Loss around the Risk] \label{origin.risk.loss}
	Under Assumption A1, for any $\thetab$ obeying Assumption~A2,
	$$ 
	\lim_{n \to \infty}    \ex \bigg [ \sup_{\tau \in [0,\infty]} \big |  R_n(\thetab, \qhatb(\tau)) -  L_n(\thetab, \qhatb(\tau)) \big| \bigg] = 0~.
	$$
\end{thm}
The above two theorems have different flavors; Theorem~\ref{origin.risk.loss} displays uniform $L_1$ convergence where as Theorem~\ref{origin.loss.are} shows convergence of the expected squared deviation at the rate of $a_n^8$ uniformly over all possible $\tau$ values. Proving the uniform $L_1$ convergence version of Theorem~\ref{origin.loss.are} as is usually done in establishing optimality results for estimation under quadratic loss, is difficult here due to the complicated nature of the ARE estimator. Also,
the rate of convergence $a_n^8$ (which is used to tackle the discretization step mentioned afterwards) is not optimal and can be made better. However, it is enough for proving the optimality of our proposed method which is our main interest.
\par
Combining the above two theorems, we see the average distance between $\Ahat$ and the actual loss is asymptotically uniformly negligible and so, we expect that minimizing $\Ahat$ would lead to an estimate with competitive performance. We propose an estimate of the tuning parameter $\tau$ for the class of shrinkage estimates $\mathcal{S}^0$ as follows:
\begin{equation}
\tauh^\A_n =\argmin_{\tau \in \net} \Ahat_n(\tau)~. \label{eq:SURE_origin} \tag{{\sf ARE Estimate}}
\end{equation}
where the minimization is done over a discrete sub-set $\net$ of $[0,\infty]$. Ideally, we would have liked to optimize the criterion over the entire domain $[0,\infty]$ of $\tau$. The discretization is done for computational reasons as we minimize $\Ahat$ by exhaustively evaluating it across the discrete set $\Lambda_n$  which only depends on $n$ and is independent of $\x$. Details about the construction of $\Lambda_n$ is provided in Section~\ref{append-origin} of the Appendix. $\Lambda_n$ contains countably infinite points as $n \to \infty$.  We subsequently show that the precision of our estimates is not hampered by such discretization of the domain. To facilitate our discussion of the risk properties of our \ref{eq:SURE_origin}, we next introduce
the oracle loss ({\sf OR}) hyperparameter
$$
\tau_{n}^{{\sf OR}} =\argmin_{\tau \in [0,\infty]} L_n(\thetab,\qhatb(\tau))~.
$$
Note that $\tau_n^{{\sf OR}}$ depends not only on $\x$ but also on the unknown $\thetab$. Therefore, it is not an estimator. Rather, it serves as the theoretical benchmark of estimation accuracy because no estimator in $\mathcal{S}^0$ can have smaller
risk than $\qhatb\left( \tau_{n}^{{\sf OR}} \right)$. Unlike the ARE estimate, $\tau_n^{{\sf OR}}$ involves minimizing the true loss over the entire domain of $\tau$. \pr{Note that} $\qhatb^{\sf Bayes} \in \mathcal{S}_0$, \pr{and thus}, even if the correct hyperparameter $\tau$ were known,  the estimator $\qhatb\left( \tau_{n}^{{\sf OR}} \right)$ is as good as the Bayes estimator.
The following theorem shows that our proposed estimator is asymptotically nearly as good as the oracle loss~estimator.
\begin{thm}[Oracle Optimality in Predictive Loss] \label{origin.are.oracle.loss}
Under Assumptions A1 and A3, for all $\thetab$ satisfying Assumption~A2 and for any $\epsilon > 0$,
$$
\lim_{n \to \infty} \Pr \left\{ L_n\left(\thetab,\qhatb(\tauh^{\A}_n) \right) 
~\geq~ L_n(\thetab,\qhatb(\tau_n^{\sf OR})) \,+\, \epsilon \right\} = 0~.
$$
\end{thm}
The above theorem shows that the loss of  our proposed estimator converges in probability
to the optimum oracle value $L_n(\thetab,\qhatb(\tau_n^{\sf OR}))$. We also show that, under the same conditions, it is asymptotically as good as $\tau_n^{\sf OR}$ in terms of the risk (expected loss).

\begin{thm}[Oracle Optimality in Predictive Risk] \label{origin.are.oracle.risk}
Under Assumptions A1 and A3 and for all $\thetab$ satisfying Assumption~A2,
$$
\lim_{n \to \infty}  R_n \left( \thetab,\qhatb(\tauh_n^{\A}) \right)  - \E \left[ L_n(\thetab,\qhatb(\tau_n^{\sf OR})) \right] = 0~.
$$
\end{thm}

We extend the implications of the preceding theorems to show that our proposed estimator is as good as  any other estimator in $\mathcal{S}^0$ in terms of both the loss and~risk.

\begin{cor}\label{cor:origin.are.oracle}
Under Assumptions A1 and A3, for all $\thetab$ satisfying Assumption~A2, for any $\epsilon >0$, and any estimator $\hat{\tau}_n \geq 0$,
	\begin{align*}
	\mathrm{I.} \;\; & \lim_{n \to \infty} \Pr \left\{ L_n\left(\thetab,\qhatb(\tauh^{\A}_n) \right)  ~\geq~ L_n(\thetab,\qhatb(\tauh_n)) + \epsilon \right\}= 0\\
	\mathrm{II.}\;\; & \lim_{n \to \infty}  R_n \left( \thetab,\qhatb(\tauh_n^{\A}) \right)  -  R_n(\thetab,\qhatb(\tauh_n)) ~\leq~0. \hspace{10cm}
	\end{align*}
\end{cor}
Next, we present two very popular, standard EB approaches for choosing estimators in $\mathcal{S}^0$. The {Empirical Bayes ML} (EBML) estimator  $\qhatb(\tauh^{\text{ML}})$  is built by maximizing the marginal likelihood of $\X$ while the {method of moments} (EBMM) estimator $\qhatb(\tauh^{\text{MM}})$ is based on  the moments of the marginal distribution of $\X$. Following \citet[Section~2]{xie2012} the hyperparameter estimates are given by
\begin{align}\label{eqn:MM}
\begin{split}
\tauh_n^{\text{ML}}  & =\argmin_{\tau \in [0,\infty]} \frac{1}{n}\sum_{i=1}^n \bigg( \frac{X_i^2}{\tau+\sigma_{p,i}} +\log(\tau+\sigma_{p,i}) \bigg ) \\
\tauh_n^{\text{MM}}  & =\max \left\{\frac{1}{n} \sum_{i=1}^p \big(X_i^2-\sigma_{p,i}\big)~,~0 \right\}
\end{split}~~.
\end{align}
For standard EB	estimates $\qhatb(\tauh_n^{\text{EB}})$,	such	as	those	in \eqref{eqn:MM} the hyperparameter estimate $\tauh_n^{\text{EB}}$ does not	depend	on	the shape of the individual loss functions $\{(b_i, h_i): 1\leq i \leq n\}$. We provide a complete definition of $\Ahat_n$ and  $\tauh^{\A}_n$ in the next section from where it will be evident that  our	asymptotically	optimal	estimator $\tauh^{\A}_n$ depends	on the ratios	$\{b_i/(b_i + h_i): 1\leq i \leq n\}$ 	in	an	essential	way that remains important as $n \to \infty$. Hence, even asymptotically, the ML and MM estimates do not always agree with $\tauh^{\A}$, particularly in cases when the ratios are not all the same. In the homoscedastic case it is easy to check that the loss function $L(\thetab,\qhat(\tau))$ has a unique minima in $\tau$ for any $\thetab$ obeying assumption A2; and so,  by Theorems \ref{origin.loss.are} and \ref{origin.are.oracle.loss}, it follows that any estimator as efficient as the {\sf OR} estimator must asymptotically agree with $\tauh^{\A}$. Hence, unlike our proposed ARE based estimator,  EBML and EBMM are not generally asymptotically optimal in the class of estimators $\mathcal{S}^0$. In Section~\ref{ex1}, we provide an explicit numerical example to demonstrate \pr{the sub-optimal behavior of  the EBML and EBMM estimators}. 
\subsection{Construction of Asymptotic Risk Estimates}\label{ARE.intro}
In this section, we describe the details for the construction of the Asymptotic Risk Estimation ($\A$) function $\tau \mapsto \Ahat_n(\tau)$, which is the core of our estimation methodology. The estimators in class $\mathcal{S}^0$ are coordinatewise rules, and the risk of such an estimate $\qhat(\tau)$ is
\begin{align*}
R_n(\thetab,\qhatb(\tau)) &= \frac{1}{n} \sum_{i=1}^n r_i(\theta_i, \qhat_{i}(\tau))~,
\end{align*}
where  $r_i(\theta_i, \qhat_i(\tau))$ is the risk associated with the $i^\text{th}$ coordinate.  By Lemma~\ref{linear.estimates}, we have that
\begin{align}\label{eq:univariate.risk.eq.origin}
r_i(\theta_i, \qhat_{i}(\tau)) = (b_i+h_i) \sqrt{\sigma_{f,i}+\sigma_{p,i} \alpha_i^2(\tau)} \, G\left(c_i (\tau)+d_i (\tau) \theta_i,  \bt_i\right)~,
\end{align}
where for all $i$, $\alpha_i(\tau) =  \tau / ( \tau~+~\sigma_{p,i})$,~$\bt_i = b_i / (b_i + h_i)$, and
$$ 
c_i(\tau) =\sqrt{ \frac{ 1+\alpha_i(\tau) \sigma_{p,i} }{ 1+\alpha_i(\tau)^2 \sigma_{p,i}} }\, \Phi^{-1}( \bt_i )
\quad {\rm and} \quad
 d_i(\tau)  =- \frac{1 - \alpha_i(\tau)}{\sqrt{ \sigma_{f,i}+\sigma_{p,i} \alpha_i(\tau)^{2}}}\,.
$$
The function $G(\cdot)$ is the same function as that associated with the predictive loss and was defined in \eqref{eq:loss-function}. The dependence of $c_i(\tau)$ and $d_i(\tau)$ on $\tau$ is only through $\alpha_i$. Note that, the risk $r_i(\theta_i, \qhat_{i}(\tau))$ is non-quadratic, non-symmetric and not centered around $\theta_i$.  However, it is a $C^{\infty}$ function of $\theta_i$ which we will use afterwards.
We propose an estimate $\Ahat_n(\tau)$ of the multivariate risk $R_n(\thetab,\qhatb(\tau))$ by using coordinate-wise estimate $\That_i(X_i,\tau)$ of $G(c_i(\tau) + d_i (\tau) \theta_i; \bt_i)$; that is,
\begin{align}\label{def.A}
\Ahat_n(\tau) = \frac{1}{n} \sum_{i=1}^n (b_i+h_i) \sqrt{\sigma_{f,i}+\sigma_{p,i} \alpha_i(\tau)^2}  \, \That_i(X_i,\tau)~.
\end{align}

\paragraph{Defining \pr{the} Coordinatewise Estimate $\That_i(X_i,\tau)$ -- Heuristic Idea} Temporarily keeping the dependence on $\tau$ and $i$ implicit, we now describe how we develop an estimate of the non-linear functional $G(c + d \, \theta, \bt)$ of the unknown parameter~$\theta$.

\par Depending on the magnitude of  $c + d \, \theta$ we use two different kinds of estimation strategy for $G(c + d \, \theta, \bt)$.
If $|c+d \theta|$ is not too large we approximate the functional by $G_K(c+d \,\theta, \bt)$ -- its $K$ order Taylor series  expansion around $0$:
$$ G_K(c+d \,\theta, \bt) = G(0,\bt)+G'(0,\bt) (c + d\,\theta)+ \phi(0) \sum_{k=0}^{K-2} \frac{(-1)^{k}H_k(0)}{(k+2)!} \,(c + d\,\theta)^{k+2}\,,$$
where $H_k$ is the $k^{\text{th}}$ order probabilists' Hermite polynomial \citep[Ch. 1.1]{thangavelu-book}. 
If $W\sim N(\mu,\sigma)$ denotes a normal random variable with mean $\mu$ and variance $\sigma$, then we can construct an unbiased estimator of the truncated functional $G_K$ by using the following property of Hermite polynomials: 
\begin{align}\label{her.property}
 \textit{ If } W \sim N(\mu,\sigma), \textit{ then } \sigma^{k/2} \, \ex_{\mu} \big\{ H_k( W / \sqrt{\sigma}) \big\}= \mu^k \textit{ for } k \geq 1.
\end{align}
\par
Now, if $|c+d \theta|$ is large, then the truncated Taylor's expansion $G_K(\cdot)$ would not be \pr{a} good approximation of $G(c + d \, \theta, \bt)$. However, in that case, as shown in Lemma~\ref{lem:bias.bound}, we can use linear approximations with
$$
G(c + d \, \theta, \bt) ~\approx~ (1-\bt) (c + d \, \theta)^+ + \bt (c + d \, \theta)^- ~,
$$
and their corresponding unbiased estimates can be used. Note that for all $x \in \Real$, $x^+ = \max\{x,0\}$ and $x^- = \max\{-x, 0\}$.

\par
\textit{The Details.} We need to combine the aforementioned estimates together in a data-driven framework. For this purpose, we use threshold estimates. We use the idea of \textit{sample splitting}. We 
use the observed data to create two independent samples by adding white noise $\Z=\{Z_i: i=1,\ldots,n\}$ and define
$$U_i=X_i+\sqrt{\sigma_{p,i}} Z_i, \quad V_i=X_i -\sqrt{\sigma_{p,i}} Z_i \text{ for } i=1,\ldots,n.$$
Noting that $U_i$ and $V_i$ are independent, we will use $V_i$ to determine whether or not $c_i(\tau)+d_i(\tau) \,\theta $ is large, and then use $U_i$ to estimate $G(c_i(\tau)+d_i(\tau) \,\theta, \bt)$ appropriately. 
For any fixed $\tau \in [0,\infty]$ and $i=1,\ldots,n$, we transform
$$U_i(\tau)=c_i(\tau)+d_i(\tau) U_i, \quad V_i(\tau)=c_i(\tau)+d_i(\tau) V_i.$$
Note that $U_i(\tau) \sim N(c_i (\tau) + d_i (\tau) \theta_i\,,\, 2 \sigma_{p,i} \pr{\, d_i^2 (\tau)} )$. By Lemma~\ref{her.property}, we construct an unbiased estimate of $G_K(c_i (\tau)+ d_i (\tau) \theta_i,\bt_i)$ as
\begin{align*}
S_i(U_i(\tau))=&G(0,\bt_i)+G'(0,\bt_i) U_i(\tau) \\
& + \phi(0) \sum_{k=0}^{\Kn(i)-2} \frac{(-1)^{k}H_k(0)}{(k+2)!} \big(2\sigma_{p,i} d_i^2(\tau)\big)^{(k+2)/2} H_{k+2}\bigg(\frac{U_i(\tau)}{(2 \sigma_{p,i} d^2_i(\tau))^{1/2}}\bigg).
\end{align*}
We use a further truncation on this unbiased estimate by restricting its absolute value to $n$. The truncated version
\begin{align*}
\Sm_i(U_i(\tau)) &= S_i(U_i(\tau)) \, I\{|S_i(U_i(\tau))| \leq n\} \,+ \,n \,  I\{S_i(U_i(\tau)) > n\} \,-\, n \, I\{S_i(U_i(\tau)) < - n\}\\
&= {\rm sign}\left( S_i(U_i(\tau)) \right) \min \left\{|S_i(U_i(\tau))|~,~ n\right\}
\end{align*} 
is biased. But, because of its restricted growth, it is easier to control its variance, \pr{which greatly facilitates our analysis.}
\par
\textit{Threshold Estimates.} \pr{For each coordinate $i$}, we then construct the following coordinatewise threshold~estimates:  
\begin{align*}
\That_{i}( X_i,Z_i,\tau)= \left \{ \begin{array}{ccc} 
- \bt_i \, U_i(\tau) 	& \text{ if } & V_i(\tau) 	< 		-\thr(i)\\[1ex] 
\Sm_i(U_i(\tau)) 		& \text{ if } & \pr{- \thr(i) \leq V_i(\tau)	\leq 	\thr(i)} \\[1ex]
(1-\bt_i) \, U_i(\tau)	& \text{ if } & V_i(\tau)	> 		\thr(i)
\end{array}\right. \text{ for } i=1,\ldots,n
\end{align*}
with the threshold parameter 
\begin{align}\label{thr.definition}
\thr(i) = \gamma(i) \sqrt{2 \log n}~,
\end{align}
where $\gamma(i)$ is any positive number less than  $\left( \; 1/\sqrt{ 4 e} - \sqrt{\sigma_{p,i}/\sigma_{f,i}} \; \right)$. Assumption A2 ensures \pr{the} existence of $\gamma(i)$ because $\sigma_{p,i}/\sigma_{f,i} < 1 /(4 e)$ for all $i$.
\par
The other tuning parameter that we have used in our construction process is the truncation parameter $\Kn(i)$, which is involved in the approximation of $G$ and is used in the estimate $\Sm$. We select a choice of $\Kn(i)$ that is independent of $\tau \in [0,\infty]$, and is given by
\begin{align}\label{eq:trun.parameter}
\Kn(i) =1 + \bigg \lceil e^2  \Big( \; \gamma(i) + \sqrt{ 2 \sigma_{p,i}/\sigma_{f,i} } \; \Big)^2(2 \log n)  \bigg \rceil\,. 
\end{align}
\textit{Rao-Blackwellization.} $\That_{i}(X_i,Z_i,\tau)$ are randomized estimators as they depend on the user-added noise $\Z$. And so, in the final step of the risk estimation procedure we apply Rao-Blackwell adjustment  \citep[Theorem 7.8, Page 47]{lehmann1998theory} to get 
$\That_i(X_i,\tau)=\ex \left[ \That_i(X_i,Z_i,\tau) |\X \right]$. Here, the expectation is over the distribution of $\Z$, which is independent of $\X$ and follows $N(0,I_n)$.
\par

\subsubsection{Bias and Variance of the coordinatewise Risk Estimates}

The key result that allows us to establish Theorem \ref{origin.loss.are} is the following proposition (for proof see Section \ref{sec-2-main-proof}) on estimation of the univariate risk components $G(c_i (\tau)+d_i (\tau) \theta_i,  \bt_i)$ defined in \eqref{eq:univariate.risk.eq.origin}. It  shows that the bias of  $\That_{i}( X_i,Z_i,\tau)$  as an estimate of $G(c_i (\tau)+d_i (\tau) \theta_i,  \bt_i)$ converges to zero as $n \to \infty$. The scaled variance of each of the univariate threshold estimates $\That_{i}( X_i,Z_i,\tau)$ also converges to zero. 

\begin{prop}\label{univariate.origin}
Under Assumptions A1 and A3, we have  for all $i=1,\ldots,n$
\begin{align*}
\mathrm{I.} &\lim_{n \to \infty} \sup_{\tau \in [0,\infty], \, \theta_i \in \Real} a_n^8 \, \bias_{\theta_i}(\That_i(X_i,Z_i,\tau)) = 0~, \\
\mathrm{II.} &\lim_{n \to \infty}  \sup_{\tau \in [0,\infty], \, \theta_i \in \Real}  n^{-1} a_n^8 \,\var_{\theta_i} (\That_i(X_i,Z_i,\tau)) = 0~, \text{ where } a_n = \log \log n \text{\hspace{10cm}}
\end{align*}
and the random vectors $\X$ and $\Z$ are independent, with $\X$ following \eqref{pred.model} and $\Z$ has $N(0,I)$ distribution.
\end{prop}

\subsection{Background and Previous Work}\label{lit.review}
Here, we follow the compound decision theory framework introduced in \citet{robbins1985asymptotically}. 
In the statistics literature, there has been substantial research on the construction of linear EB  estimates in such frameworks \citep{morris1983parametric,zhang2003}. Since the seminal work by \cite{james1961estimation}, shrinkage estimators are widely used in real-world applications \citep{efron1975data}. Stein's shrinkage is related to hierarchical empirical Bayes methods \citep{stein1962confidence}, and several related parametric empirical Bayes estimators have been developed \citep{efron1973stein}. As such, Stein's Unbiased Risk Estimate (SURE) is one of the most popular methods for obtaining the estimate of tuning parameters. \citet{donoho1995adapting} used  SURE
to choose the threshold parameter in their SureShrink method. However, most of these developments have been under quadratic loss or other associated loss functions \citep{brown1975estimation,berger1976admissible,dey1985estimation}, which admit unbiased risk estimates. \citet{Dasgupta1999} discussed the role of Steinien shrinkage under $L_1$ loss, which is related to our predictive loss only when $b=h$. If $b \neq h$, their proposed estimator do not incorporate the asymmetric nature of the loss function and are sub-optimal (See Corollary~\ref{cor:origin.are.oracle}). 
To construct risk estimates that are adapted to the shape of the cumulative check loss functions, we develop new methods for  efficiently estimating the  risk functionals associated with our class of shrinkage estimators.  In our construction, we concentrate on obtaining uniform convergence of the estimation error over the range of the associated hyperparameters. This enables us to efficiently fine-tune the shrinkage parameters through minimization over the class of risk estimates. Finally, in contrast to quadratic loss results \citep[Section~3]{xie2012}, we develop a more flexible moment-based concentration approach that translates our risk estimation efficiency into the decision theoretic optimality of the proposed shrinkage estimator.

\subsection{Further Results}\label{further.results}
We now describe our results for efficient estimation in  class $\mathcal{S}$, where we shrink towards a data-driven direction $\eta$, and the hyperparameters  $\eta$ and $\tau$ are simultaneously estimated. Here, we restrict the location hyperparameter $\eta$ to lie in the set $\hat{M}_n=[\hat{m}_n(\alpha_1), \hat{m}_n(\alpha_2)]$
where $\hat{m}_n(\alpha)=\text{quantile}\{X_i: 1 \leq 1 \leq n ; \alpha\}$ is the $\alpha$ th quantile of the observed data and $\alpha_1=\inf \{ b_i/(b_i+h_i): 1 \leq i \leq n \}$ and $\alpha_2=\sup \{ b_i/(b_i+h_i): 1 \leq i \leq n \}$. By Lemma~\ref{loss.properties}, we know that if the true distributions were known, the optimal predictor for dimension $i$ is given by the  $b_i / (b_i + h_i)$ quantile.  In this context, it make sense to restrict the shrinkage location  parameter $\eta$ in the aforementioned range as we do not want to consider non-robust estimators that shrink toward locations that lie near undesired periphery of the data.
\par
The predictive risk of estimators $\qhatb(\eta,\tau)$ of the form \eqref{eq:linear.est} is given by $R_n(\thetab,\qhatb(\eta,\tau))=\ex_{\thetab}\left[ L_n(\thetab,\qhatb(\eta,\tau))\right]$. We estimate the risk function by $(\eta,\tau) \mapsto \AD_n(\eta,\tau)$. The estimation procedure and the detailed proof for the results in this section are presented in Section~\ref{sec-data-driven}.  We estimate the tuning parameters $\tau$ and $\eta$ for the class of shrinkage estimates $\mathcal{S}$ by minimizing the  $\AD_n(\eta,\tau)$ criterion jointly over the domain of $\tau$ and $\eta$. Computationally, it is done by minimizing over a discrete grid:
$$
\big(\etah^{\ADD}_n,\tauh^{\ADD}_n\big)=\argmin_{(\eta,\tau) \in (\nett_{n,1} \cap \hat{M}_n) \otimes \nett_{n,2} } \AD_n(\eta,\tau)~,
$$
where $\nett_{n,2}$ is a discrete grid spanning $[0,\infty]$ and $\nett_{n,1}$ is a discrete grid spanning $[-a_n, a_n]$ with $a_n=\log \log n$. 
Both, $\nett_{n,1}$ and $\nett_{n,2}$ do not depend on $\X$  but only on $n$.  The minimization is conducted only over $\eta$ values in $\nett_{n,1}$ which lie in the set $\hat{M}_n$.  Details on the construction of the grid is presented in Section~\ref{sec-data-driven}. We define the oracle loss estimator here by
$$
(\eta_n^{\sf DOR},\,\tau_n^{\sf DOR})=\argmin_{\tau \in [0,\infty],\; \eta \in [\hat{m}_n(\alpha_1), \hat{m}_n(\alpha_2)]} L_n(\thetab,\qhatb(\eta,\tau))~
$$


The following theorem shows that our risk estimates estimate the true loss uniformly well. 

\begin{thm}\label{data.loss.are}
	Under Assumptions A1 and A3, for all $\thetab$ satisfying Assumption~$\text{A2}$ \pr{and for all estimates $\qhatb(\eta,\tau) \in \mathcal{S}$,} 
	$$ \lim_{n \to \infty} \sup_{\tau \in [0,\infty],\;  |\eta| \leq a_n } a_n^4 \, \ex \left|\AD_n(\eta,\tau)-L_n(\thetab,\qhatb(\eta,\tau))\right| = 0~ \quad \text{ where } a_n = \log \log n. $$ 
\end{thm}

Based on the above theorem, we derive the decision theoretic optimality of our proposed estimator. The following two theorems show that our estimator is asymptotically nearly as good as the oracle loss estimator, whereas the corollary shows that it \pr{is} as good as any other estimator in $\mathcal{S}$. 

\begin{thm}\label{data.are.oracle.loss}
	Under Assumptions A1 and A3, and for all $\thetab$ satisfying Assumption~$\text{A2}$, we have, \pr{for any fixed $\epsilon > 0$,}
	$$\lim_{n \to \infty} P \left\{ L_n\big(\thetab,\qhatb(\etah^{\ADD}_n, \tauh^{\ADD}_n)\big)\geq L_n\big(\thetab,\qhatb (\eta_n^{\sf DOR},\tau_n^{\sf DOR}) \big)+ \epsilon \right\} = 0~. $$
\end{thm}
\begin{thm}\label{data.are.oracle.risk}
	Under Assumptions A1 and A3,  and for all $\thetab$ satisfying Assumption~$\text{A2}$, 
	$$\lim_{n \to \infty} R_n\big(\thetab,\qhatb( \etah^{\ADD}_n, \tauh^{\ADD}_n)\big)- \ex \big [L_n\big(\thetab,\qhatb (\eta_n^{\sf DOR},\tau_n^{\sf DOR})\big) \big ] = 0.$$
\end{thm}
\begin{cor}\label{cor:data.are.oracle}
	Under Assumptions A1 and A3, for all $\thetab$ satisfying Assumption~$\text{A2}$ and for any estimator $\hat{\tau}_n \geq 0$ and $\etah_n  \in [\hat{m}_n(\alpha_1), \hat{m}_n(\alpha_2)]$, 
	\begin{align*}
	\mathrm{I.} \;\; & \lim_{n \to \infty} P \left\{ L_n\big(\thetab, \qhatb(\etah^{\ADD}_n, \tauh^{\ADD}_n)\big)\geq L_n\big(\thetab,\qhatb(\etah_n, \tauh_n)\big) + \epsilon \right\}= 0 \text{ for any fixed } \epsilon > 0. \\
	\mathrm{II.}\;\; & \lim_{n \to \infty} R_n\big(\thetab,\qhatb(\etah^{\ADD}_n,\tauh^{\ADD}_n)\big) - R_n\big(\thetab,\qhatb(\etah_n, \tauh_n)\big) \leq 0. \hspace{10cm}
	\end{align*}
\end{cor}

The EBML estimate of the hyperparameters are given by
$$
\tauh_n^{\text{ML}}=\argmin_{\tau \in [0,\infty]} \frac{1}{n}\sum_{i=1}^n \bigg( \frac{(X_i-f(\tau))^2}{\tau+\sigma_{p,i}} +\log(\tau+\sigma_{p,i}) \bigg ) \quad \text{ and } \quad
\etah_n^{\text{ML}}= f(\tauh_n^{\text{ML}})~,
$$
where $f(\tau)=f_1(\tau)\, I\{f_1(\tau)\in [\hat{m}_n(\alpha_1), \hat{m}_n(\alpha_2)]\} + \hat{m}_n(\alpha_1)\, I\{f_1(\tau)<\hat{m}_n(\alpha_1)\}+ \hat{m}_n(\alpha_2)\, I\{f_1(\tau)>\hat{m}_n(\alpha_2)\}$ and
$f_1(\tau) =( \sum_{i=1}^n (\tau+\sigma_{p,i})^{-1}X_i )   / ( \sum_{i=1}^n (\tau+\sigma_{p,i})^{-1})$.
The method of moments (MM) estimates are roots of the following equations:
$$\tau=\frac{1}{n-1}\bigg( \sum_{i=1}^n \big(X_i-\eta\big)^2- (1-1/n)\sigma_{p,i}\bigg)_+ \, \text{ and } \eta=f(\tau).$$
Unlike $(\etah_n^{\ADD},\tauh_n^{\ADD})$, the EBML and EBMM estimates of the hyperparameters  do not depend on the shape of the loss functions $\{(b_i,h_i): 1 \leq i \leq n\}$. Thus, the EBML and EBMM estimators $\qhatb(\etah^{\text{ML}}, \tauh^{\text{ML}})$ and $\qhatb(\etah^{\text{MM}}, \tauh^{\text{MM}})$ do not always agree with the ARE based estimator $\qhatb(\etah^{\ADD},\tauh^{\ADD})$. 
\par 
\textit{Results on Estimators in $\mathcal{S}^G$.} Following \eqref{eq:linear.est}, the class of estimators with shrinkage towards the Grand Mean ($\Xm$) of the past observations is of the following form:  for $i=1,\ldots,n$,
\begin{align}\label{eq:linear.est.g}
\qhat_i^G(\tau)=\alpha_i(\tau) X_i + (1-\alpha_i(\tau)) \Xm + (\sigma_{f,i}+\alpha_i(\tau) \sigma_{p,i})^{1/2}\, \Phi^{-1}(\bt_i)~, 
\end{align} 
where $\tau$ varies over $0$ to $\infty$, and $\alpha_i(\tau)$, and $\bt_i$ are defined just below Equation~\eqref{eq:linear.est}. For any fixed $\tau$, unlike estimators in $\mathcal{S}$, $\qhatb^G(\tau)$ is no longer a coordinatewise independent rule. In Section~\ref{sec-grand-mean}, we develop an estimation strategy which estimates the loss of estimators in $\mathcal{S}^G$ uniformly well over the grid $\nett_n$ of Section~\ref{subsec:main.results}.
\begin{thm}\label{mean.loss.are}
	Under Assumptions A1 and A3, for all $\thetab$ satisfying Assumption $\text{A2}$ \pr{and for all estimates $\qhatgb(\tau) \in \mathcal{S}^G$},
	$$ \lim_{n \to \infty}  \ex \bigg\{ \sup_{\tau \in \nett_n } \left|\AGhat_n(\tau)-L_n(\thetab,\qhatgb(\tau))\right| \bigg \} = 0~.$$ 
\end{thm}
We propose an estimate $\tauh^{\AG}_n=\argmin_{\tau \in \nett_n} \AGhat_n(\tau)$ for the hyperparameter in this class and compare its asymptotic behavior with the oracle loss $\tau^{\sf GOR}_n=\argmin_{\tau \in [0,\infty]} L_n(\thetab,\qhatgb(\tau))$. Like the other two classes, based on Theorem~\ref{mean.loss.are}, here we also derive the asymptotic optimality of our proposed estimate in terms of both the predictive risk and loss. 

\begin{thm}\label{mean.decision.theoretic.properties}
	Under Assumptions A1, A3, for all $\thetab$ satisfying Assumption~$\text{A2\,}$\\[1ex]
	\texttt{(A)} comparing with the oracle loss estimator, we have the following:
	\begin{align*}
	\mathrm{I.} \;\; &\lim_{n \to \infty} \Pr \left\{ L_n\big(\thetab,\qhatgb(\tauh^{\AG}_n)\big)\geq L_n\big(\thetab,\qhatgb(\tau^{\sf GOR}_n)\big)+ \epsilon\right\} = 0 \text{ for any fixed } \epsilon > 0.\\
	\mathrm{II.} \;\; & \lim_{n \to \infty} R_n\big(\thetab,\qhatgb(\tauh^{\AG}_n)\big)- \ex \left[ L_n\big(\thetab,\qhatgb(\tau^{\sf GOR}_n) \big) \right] = 0. \hspace{10cm}
	\end{align*}
	\texttt{(B)} for any estimate $\hat{\tau}_n \geq 0$ of the hyperparameter, we have the following:
	\begin{align*}
	\mathrm{I.} \;\; & \lim_{n \to \infty} \Pr \left\{ L_n\big(\thetab, \qhatgb(\tauh^{\AG}_n)\big)\geq L_n\big(\thetab,\qhatgb(\tauh_n)\big) + \epsilon \right\}= 0 \text{ for any fixed } \epsilon > 0. \\
	\mathrm{II.}\;\; & \lim_{n \to \infty} R_n\big(\thetab,\qhatgb(\tauh^{\AG}_n)\big) -  R_n\big(\thetab,\qhatgb(\tauh_n)\big) \leq 0\,. \hspace{10cm}
	\end{align*}
\end{thm}

\subsection{Organization of the Paper}
In Section~2, we provide a detailed explanation of the results involving the class of estimators $\mathcal{S}^0$. Treating this class as the fundamental case, through the proof of Theorem~\ref{origin.loss.are},  Section~2 explains the general  principle behind our asymptotic risk estimation methodology and the proof techniques used in this paper. The proofs of  Theorems~\ref{origin.risk.loss}, \ref{origin.are.oracle.loss} and \ref{origin.are.oracle.risk} and Corollary~\ref{cor:origin.are.oracle} are provided in Appendix~\ref{append-origin}. Section~\ref{section:experiments} discusses the performance of our prediction methodology in simulation experiments. Section~\ref{sec-data-driven} and its associated Appendix~\ref{append-data-driven} provide the proofs of  Theorems~\ref{data.loss.are}, \ref{data.are.oracle.loss} and \ref{data.are.oracle.risk} 
and Corollary~\ref{cor:data.are.oracle}, which deal with estimators in class $\mathcal{S}$. The proofs of Theorems~\ref{mean.loss.are} and \ref{mean.decision.theoretic.properties} involving class $\mathcal{S}^G$ are provided in Section~\ref{sec-grand-mean} and Appendix~\ref{append-grand-mean}.
In Table~\ref{table-glossary} of the Appendix, a detailed list of the notations used in the paper is provided.

%% file: sec-2.tex

In this section, we provide a detailed explanation of the results on the estimators in $\mathcal{S}^0$. 
This case serves as a fundamental building block and contains all the essential ingredients involved in the general risk estimation method. In subsequent sections, the procedure is extended to \pr{$\mathcal{S}$ and $\mathcal{S}^G$}. We begin by laying out the proof of Theorem~\ref{origin.loss.are}. The \textit{decision theoretic optimality results} -- Theorems~\ref{origin.are.oracle.loss} and \ref{origin.are.oracle.risk} and Corollary~\ref{cor:origin.are.oracle} -- follow easily from Theorems~\ref{origin.loss.are} and \ref{origin.risk.loss}; their proofs are provided in Appendix~\ref{append-origin}. To prove Theorem~\ref{origin.risk.loss}, we use the fact that the parametric space is bounded (Assumption A2) and apply the uniform SLLN  argument \citep[Lemma~2.4]{newey1994} to establish the desired concentration. The detailed proof is in the Appendix~\ref{append-origin}.

\subsection{Proof of Theorem~\ref{origin.loss.are}}
We will use Proposition \ref{univariate.origin} in the proof.
Let $\Ahat_n(\Z,\tau)$ denote a randomized risk estimate before the Rao-Blackwellization step in Section~\ref{ARE.intro}.  
For any fixed $\tau$, $\{\That_i(X_i,Z_i,\tau): 1\leq i \leq n\}$ are independent of each other, so the Bias-Variance decomposition~yields
\begin{align}\label{eq.bias.var}
\begin{split}
&\ex \left[ \big(R_n(\thetab,\qhatb(\tau)) -\Ahat_n(\Z,\tau) \big)^2 \right] \\
& \leq A_n \bigg\{ \bigg(\frac{1}{n} \sum_{i=1}^n\bias(T_i(X_i,Z_i,\tau))\bigg)^2 + \frac{1}{n^2} \sum_{i=1}^n \var(T_i(X_i,Z_i,\tau))\bigg\}~,
\end{split}
\end{align}
where $A_n=\sup\{(b_i+h_i)^2(\sigma_{f,i}+ \alpha_i(\tau) \sigma_{p,i}): i=1,\ldots,n\}$ and $\alpha_i(\tau) = \tau / (\tau + \sigma_{p,i})$.
By Assumption A1 and A3, $\sup_n A_n < \infty$. From Proposition \ref{univariate.origin}, both terms on the right hand side after being scaled by $a_n^8$ uniformly converge to 0 as $n \to \infty$.   This shows that
$$\lim_{n \to \infty} a_n^8 \sup_{\tau \in [0,\infty]} \ex \big[ \big(R_n(\thetab,\qhatb(\tau)) -\Ahat_n(\Z,\tau) \big)^2\big]~=~0~,$$
where the expectation is over the distribution of $\Z$ and $\X$. 
As $\Ahat_n(\tau)=\ex[\Ahat_n(\Z,\tau)|X]$, using Jensen's inequality for conditional expectation, we have
$\ex[(R_n(\thetab,\qhatb(\tau)) -\Ahat_n(\Z,\tau) )^2] \geq \ex [ (R_n(\thetab,\qhatb(\tau)) -\Ahat_n(\tau) )^2]$ for any $n$, $\thetab$ and $\tau$. Thus,
$$\lim_{n \to \infty} a_n^8 \sup_{\tau \in [0,\infty]} \ex \big[  \big(R_n(\thetab,\qhatb(\tau)) -\Ahat_n(\tau) \big)^2 \big]~=~0~.$$
Thus, to complete the proof, it remains to establish Proposition~\ref{univariate.origin}, which shows that both the bias and variance converge to zero as dimension of the model increases.    We undertake this in the next section.
Understanding how the bias and variance is controlled will help the reader to appreciate the elaborate construction process of $\A$ estimates and our prescribed choices of the threshold parameter $\lambda_n(i)$ and truncation parameter \pr{$\Kn(i)$}. 

\subsection{Proof of Proposition~\ref{univariate.origin} Overview and Reduction to the Univariate Case}\label{uni.main.thm}

In this section, we outline the overview of the proof techniques used to establish Proposition~\ref{univariate.origin}.  It suffices to consider a generic univariate setting and consider each coordinate individually.  This will simplify our analysis considerably.  In addition, we will make use of the following two results about the property of the loss function~$G$.
The proof of these lemmas are given in Appendix~\ref{append-origin}.

\begin{lem}[Formula for the Loss Function]\label{loss.properties}
If $Y \sim N(\theta,\sigma)$, then
\begin{align}\label{eq:identity}
\ex_{\theta} \left[ b   (Y-q)^+  + h   (q-Y)^+ \right] =(b+h) \sqrt{\sigma} \,  G\left((q-\theta)/\sqrt{\sigma} \,,\,\bt\right)~,
\end{align}
where \bt=b/(b+h) and for all $w \in \Real$ and $\beta \in [0,1]$, $G(w,\beta) = \phi(w) + w \Phi(w) - \beta w$.
Also, if $\theta$ is known, the loss $l(\theta,q)$ is minimized at $q=\theta+\sqrt{\sigma}\Phi^{-1}(\bt)$ and the minimum value is
$(b+h)\sqrt{\sigma}\phi(\Phi^{-1}(\bt))$. 
\end{lem}

The next lemma gives an explicit formula for the Bayes estimator and the corresponding Bayes risk in the univariate setting.

\begin{lem}[Univariate Bayes Estimator] \label{linear.estimates}
Consider the univariate prediction problem where the past $X \sim N(\theta,\sigma_p)$, the future $Y \sim N(\theta,\sigma_f)$ and $\theta \sim N(\eta,\tau)$.
Consider the problem of minimizing the integrated Bayes risk.  Then,
$$
	\min_{q} \int R(\theta, q) \pi_{(\eta,\tau)}(\theta\,|\,x) d\theta ~=~ (b+h)\sqrt{\sigma_f+\alpha \sigma_p} \;\phi(\Phi^{-1}(\bt))~,
$$
where $\bt=b/(b+h)$, $\alpha = \tau / (\tau + \sigma_p)$, and $\pi_{(\eta,\tau)}(\theta~|~x)$ is the posterior density of $\theta$ given $X = x$.  Also, the  Bayes estimate $\qhat^{\sf Bayes}(\eta,\tau)$ that achieves the above minimum 
is given by
$$\qhat^{\sf Bayes}(\eta,\tau)=\alpha x + \pr{(1-\alpha)} \eta +\sqrt{\sigma_{f}+\alpha \sigma_{p}}\, \Phi^{-1}(\bt)~.$$
Finally, the risk $r(\theta, \qhat^{\sf Bayes}(\eta,\tau))$ of the Bayes estimator is
$$ (b+h) \sqrt{\sigma_f+\alpha^2 \sigma_p}\, G(c_{\tau}+d_{\tau} \, (\theta - \eta),\bt)~,$$
where
$$ c_{\tau} = \sqrt{ (1+\alpha \sigma_p) / (1+\alpha^2 \sigma_p)} \, \Phi^{-1}(\bt) \qquad \text{ and } \qquad  d_{\tau}=-(1-\alpha)/\sqrt{ \sigma_f+\alpha^2 \sigma_p}.$$
\end{lem}

By Lemma~\ref{loss.properties}, note that the loss function  is scalable in \pr{$\sigma_{f}$}.  Also by Lemma~\ref{linear.estimates}, 
we observe that the risk calculation  depends only on the ratio $\sigma_p / \sigma_f$ and scales with $b+h$.  Thus,
without loss of generality, henceforth we will assume that $\sigma_f = 1$, $b+h = 1$ and write $\sigma = \sigma_p$ and $\bt = b/(b+h) = b$. As a convention, for any number $\beta \in [0,1]$, we write $\bar{\beta} = 1 - \beta$.

\textit{Reparametrization and some new notations.} In order to prove \pr{the desired result}, we will work with generic univariate risk estimation problems where $X \sim N(\theta,\sigma)$ and $Y \sim N(\theta,1)$.  Note that Assumption A3 requires that $\sigma < 1/(4e)$.
For ease  of presentation, we restate and partially reformulate the univariate version of \pr{the} methodology stated in Section~\ref{ARE.intro}.  We conduct sample splitting by adding independent Gaussian noise $Z$: 
$$U=X+\sqrt{\sigma} Z, \quad V=X -\sqrt{\sigma} Z.$$
Instead of \pr{$\tau \in [0,\infty]$},  we   reparameterize the problem using $\alpha= \tau/(\tau+\sigma) \in [0,1]$.  By Lemma~\ref{linear.estimates} and the fact that $\sigma_f = 1$ and $b+h = 1$, the univariate risk function (with $\eta = 0$) is given by $\alpha \mapsto G(\ca+\da \theta, b)$,  where $b < 1$ and 
$$\ca=\Phi^{-1}(b) \sqrt{(1+\alpha \sigma)/(1+\alpha^2 \sigma)}\qquad {\rm and} \qquad \da=-\bar{\alpha}/\sqrt{1+\alpha^2 \sigma}~.$$
Now, consider $\Ua=\ca+\da U, \Va=\ca+\da V \text{ and } \tha= \ca+\da \theta$.  By construction $(\Ua,\Va)\sim N(\tha,\tha,2 \sigma \da^2, 2 \sigma \da^2, 0)$ and $\alpha \mapsto G(\tha,b)$ is estimated 
by the ARE estimator $\alpha \mapsto \Ta(X,Z)$, where
$$
\Ta(X,Z)= -b \, \Ua \, \indicator{\{ \Va < -\thr\}} +  \Sm(U_{\alpha}) \, \indicator{\{ |\Va| \leq \thr\}} + \bar{b} \, \Ua \, \indicator{\{ \Va  > \thr\}}~,
$$
where $\bar{b}=1-b$, and the threshold is given $\thr = \gamma \sqrt{2 \log n } $, where $\gamma$ is any positive number less than $\sqrt{2 \sigma} \left( (\, 1/\sqrt{4 e \sigma} \,) - 1 \right) = (1/\sqrt{2e}) - \sqrt{ 2\sigma} $, which is well-defined by~Assumption~A3 because $\sigma < 1/(4 e)$.

The estimator $\Sm(U_{\alpha})$ is the truncated Taylor series expansion of $G(\tha,b)$, defined as follows.  
Let 
$$
 \Kn =1 +  \left \lceil e^2 \big( \gamma +  \sqrt{2 \sigma}  \big)^2  (2 \log n)  \right \rceil~. 
$$
Let $G_{\Kn}(\tha,b)$ denote the the $\Kn^{th}$ order Taylor series expansion of  $G(\tha,b)$.  Let $S(U_{\alpha})$ denote  an unbiased estimate of $G_{\Kn}(\tha,b)$; that is,
\begin{equation}\label{def.uni}
\begin{split}
S(U_{\alpha}) & =G(0,b)+G'(0,b) \, {U_{\alpha}} \\
			& \qquad + \phi(0) \sum_{l=0}^{\Kn-2} \frac{(-1)^{l}H_l(0)}{(l+2)!} \left( \sqrt{2\sigma d_{\alpha}^2}\right)^{l+2} H_{l+2}\left(\frac{U_{\alpha}}{\sqrt{2 \sigma d_{\alpha}^2}}\right)~,
\end{split}
\end{equation}
and finally, we have $\Sm(U_{\alpha})=\mathrm{sign}(S(U_{\alpha})) \, \min\{|S(U_{\alpha})|, n\}$, which is the truncated version of $S(U_{\alpha})$.  This completes the definition of the estimator $\Ta(X,Z)$. This reparametrization allows us to deal with the stochasticity of the problem only through the random variables $\{\Ua, \Va : \alpha \in [0,1]\}$ and saves us the inconvenience of dealing with the varied functionals of $X$ and $Z$ separately. 


\paragraph{Proof Outline}
We partition the univariate parameter space into 3 cases:
\texttt{Case~1:} $|\tha| \leq \thr/2$,~\texttt{Case~2:} $\thr/2 < |\tha| \leq \big( 1+   \sqrt{2 \sigma} /  {\gamma}  \big) \thr$ and \texttt{Case~3:} $|\tha| > \big( 1+   \sqrt{2 \sigma} /  {\gamma}  \big) \thr$. We present a heuristic argument for considering such a decomposition.
The following lemma, whose proof is provided in Appendix~\ref{append-origin}, establishes a bound on the bias in different regimes.  

\begin{lem}[Bias Bounds] \label{lem:bias.bound}
There is an absolute constant $c$ such that for all \pr{$b \in [0,1]$ and  $\alpha \in [0,1]$},
\begin{align*}
\mathrm{I.} \quad & |G(y,b)-G_{\Kn} (y,b) | \leq c \, \frac{ n^{-(e^2-1) (\gamma + \sqrt{2\sigma})^2} }{ e^4 (\gamma + \sqrt{2\sigma})^2 }\quad \text{\rm for all } \quad  |y| \leq \big( 1+ \sqrt{2 \sigma} /  {\gamma}  \big)\thr ~.\\
\mathrm{II.} \quad & |G(y,b)-\bar{b} y | \leq \frac{  e^{-y^2/2} }{ y^2} \quad \text{\rm for all } y > 0 ~.\\
\mathrm{III.} \quad & |G(y,b) - (-b y) | \leq \frac{ e^{-y^2/2} }{ y^2} \quad \text{\rm for all } y < 0 ~.
\end{align*}
\end{lem}

Thus, we would like to use linear estimates when $|w|$ is large and $S(\Ua)$ otherwise. The choice of threshold $\thr$ is chosen such that this happens with high probability. 
As we have a normal model in Case 3, which includes unbounded parametric values, we will be mainly using the linear estimates of risk because when $|\tha| \geq \left(1+{\sqrt{2\sigma} / \gamma}\right) \thr$, the probability of selecting $\Sm$ over the linear estimates is very low. Similarly, in Case 1, we will be mainly using $\Sm$. Case 2 is the buffering zone where we may use either $\Sm$ or the linear estimates.

We also need to control the variances of the $3$ different kind of estimates used in $\Ta(X,Z)$. 
While the variances of the linear estimators are easily controlled, we needed to pay special attention to control the variance of $S(\Ua)$. In the following lemma, we exhibit an upper bound on the quadratic growth of the estimator $S(\Ua)$. 
The choice of the truncation parameter $\Kn$ in $\Sm(\Ua)$ was done is such a way that both its bias and squared growth are controlled at the desired limits.

\begin{lem}[Variance Bounds]\label{lem:variance.bound}
For any \pr{$b \in [0,1]$} and $a_n=\log \log n$, 
\begin{align*}
\lim_{n \to \infty} \sup_{\alpha\,:\,\abs{\tha} \;\leq\; (1  \,+\,  \sqrt{2\sigma}/\gamma) \thr }  \,n^{-1}\,a_n^8\, {\ex_{\tha} \left[ \{S(\Ua)\}^2 \right]} = 0~,
\end{align*}
where the expectation is over the distribution of $\Ua$, which has $N(\tha, 2 \sigma \da^2)$ distribution for all $\alpha \in [0,1]$. 
\end{lem}

Our proof also makes use of the following large deviation bounds.
\begin{lem}[Large Deviation Bounds]\label{lem:large.deviation.bound}
\begin{align*}
\text{For \texttt{Case 1},} \qquad	\lim_{n \to \infty} \sup_{\alpha \,:\, |\tha| \;\leq\; \thr/2}  a_n^8 \, \thr^2 \cdot \Pr_{\tha} \{ \abs{\Va} > \thr \} = 0~. \text{\hspace{8cm}}
\end{align*}
For \texttt{Case 2}, 
\begin{align*}
\qquad \qquad	\lim_{n \to \infty} \sup_{\alpha \,:\, \thr/2 \;<\; \tha \;\leq\; (1 \,+\, \sqrt{2\sigma}/\gamma)\thr } a_n^8 \,  \abs{\tha} \cdot \Pr_{\tha} \{ \Va <  -\thr \} &= 0 \,. \\
	\lim_{n \to \infty} \sup_{\alpha \,:\, -(1 + \sqrt{2\sigma} /\gamma)\thr \;\leq\;  \tha  \;<\;  - \thr/2  } a_n^8 \,  \abs{\tha} \cdot \Pr_{\tha} \{\Va > \thr \} &= 0 \,.
\end{align*}
For \texttt{Case 3}, 
\begin{align*}
\qquad \qquad  &\lim_{n \to \infty} \sup_{\alpha \,:\, |\tha| \;>\; \left(1 \,+\, {\sqrt{2\sigma} / \gamma}\right) \thr}  n \, a_n^8 \,  \cdot \Pr \left\{ |\Va| \leq \thr \right\} &&\hspace{-30pt}= 0 \,. \\
&\lim_{n \to \infty} \sup_{\alpha \,:\, |\tha| \;>\; \left(1 \,+\, {\sqrt{2\sigma} / \gamma}\right) \thr} a_n^8 \, \tha^2 \cdot \Pr \left\{ |\Va| \leq \thr \right\} &&\hspace{-30pt}= 0\,.\\
&	\lim_{n \to \infty} \sup_{\alpha \,:\, \tha \;>\; (1 \,+\, \sqrt{2\sigma}/\gamma)\thr } a_n^8 \,  \tha^2 \cdot \Pr_{\tha} \{ \Va <  -\thr \} &&\hspace{-30pt}= 0\,.\\
&	\lim_{n \to \infty} \sup_{\alpha \,:\, \tha \;<\; - (1 \,+\, \sqrt{2\sigma}/\gamma)\thr  }  a_n^8 \,  \tha^2 \cdot \Pr_{\tha} \{\Va > \thr \} &&\hspace{-30pt}= 0 \,.\text{\hspace{1.5cm}}
\end{align*}
\end{lem}

The proofs of the above \pr{three} lemmas are presented in Appendix~\ref{append-origin}.
\subsection{Detailed Proof of Proposition~\ref{univariate.origin}} \label{sec-2-main-proof}
$ $\\
\vspace{-0.2in}

\noindent {\bf \underline{Bounding the Bias}:} As $\ex[\Ua] = \tha$, by definition $|\bias_{\tha}(\Ta)|$ equals    
\begin{align*}
\left|\ex [\Sm(\Ua)] - G(\tha,b)\right| \cdot \Pr \left\{ |\Va| \leq \thr \right\} +  |\ex [\Sm(\Ua)]  - G(\tha,b) | \cdot \Pr\left\{ \abs{\Va} > \thr \right\}~.
\end{align*}
We will now show \pr{that} each of the two terms on the  RHS converges uniformly in $\alpha$ as $n$ increases to infinity.

\noindent {\bf First Term:}  Consider $\tha$ in \texttt{Cases 1 and 2}; that is, $|\tha| \leq \left(1 + {\sqrt{2 \sigma} / \gamma}\right) \thr$. Since $\ex S(\Ua) = G_{\Kn}(\tha,b)$, by our construction, we have that
$$
|\ex \Sm(\Ua) - G(\tha,b)| \leq  |\ex \Sm(\Ua)  - \ex S(\Ua)|+ |G_{\Kn}(\tha,b) - G(\tha,b)|, 
$$
and it follows from Lemma~\ref{lem:bias.bound} that
$\lim_{n \to \infty} \sup_{\alpha \,:\,|\tha| \leq \left(1 + {\sqrt{2\sigma} / \gamma}\right) \thr} a_n^8 \, |G_{\Kn}(\tha,b) - G(\tha,b)| = 0$.  
By Markov's Inequality, 
\begin{align*}
\abs{\ex \Sm(\Ua)  - \ex S(\Ua)} \leq \ex \left[ \abs{S(\Ua)} \, \indicator{ \{ |S(\Ua)| \geq n\} }\right] \leq \ex \left[ S^2 (\Ua) \right]/n~, 
\end{align*}
whose $a_n^8$ multiplied version converges to zero uniformly in $\alpha$ as  $n \to \infty$ by Lemma~\ref{lem:variance.bound}. 

Now, consider \texttt{Case 3}, where 
$|\tha| > \left(1 + {\sqrt{2\sigma} / \gamma}\right) \thr$.   By definition, $\abs{\Sm(\Ua)} \leq n$, and by Lemma~\ref{lem:G.bound}, $G(\tha,b) \leq \phi(0) + \max\{\bar{b},b\} |\tha|$.  
From Lemma \ref{lem:large.deviation.bound}, we have that 
$
\lim_{n \to \infty} \sup_{\alpha \,:\, |\tha| \;>\; \left(1 \,+\, {\sqrt{2\sigma} / \gamma}\right) \thr}a_n^8 \,\max\{ n, \tha^2\} \cdot \Pr \left\{ |\Va| \leq \thr \right\} = 0~,
$
and thus,
$$\lim_{n \to \infty} \sup_{\alpha \,:\, |\tha| \;>\; \left(1 \,+\, {\sqrt{2\sigma} / \gamma}\right) \thr}  a_n^8 \,\left|\ex [\Sm(\Ua)] - G(\tha,b)\right| \cdot \Pr \left\{ |\Va| \leq \thr \right\}  = 0~.
$$
Therefore, in all three cases, the first term of the bias multiplied by $a_n^8$ converges to zero.

\noindent {\bf Second Term:} The second term in the bias formula is equal to
$$
B_{\alpha, n} \equiv |\bar{b} \tha  - G(\tha,b) | \cdot \Pr\left\{ \Va > \thr \right\} + | G(\tha,b) - (-b \tha) | \cdot \Pr\left\{ \Va < -\thr \right\}~.
$$
For $\tha$ in \texttt{Case 1} with  $|\tha| \leq \thr/2$, note that by Lemma~\ref{lem:G.bound},
$$
\max\{ |\bar{b} \tha  - G(\tha,b) |~,~ | G(\tha,b) - (-b \tha) |\}  \leq  \abs{\tha} + \phi(0) + \abs{\tha}
\leq \thr + \phi(0)~,
$$
and thus $B_{\alpha, n} \leq \left( \thr + \phi(0) \right) \Pr \{ \abs{\Va} > \thr \}$.  The desired result then follows from 
Lemma \ref{lem:large.deviation.bound} for Case 1.

Now, consider $\tha$ in \texttt{Case 2}; that is, $\thr/2 < |\tha| \leq (1 + \sqrt{2\sigma}/\gamma)\thr$.  We will assume that $\thr/2 < \tha \leq (1 + \sqrt{2\sigma}/\gamma)\thr$; the case $- (1 + \sqrt{2\sigma}/\gamma)\thr < \tha < -\thr/2$ follows analogously.  Since $\tha > \thr/2$, it follows from Lemma \ref{lem:bias.bound} that $$
\abs{\bar{b} \tha  - G(\tha,b) } \leq    e^{-\tha^2/2} / \tha^2  \leq 4 \,  e^{-\thr^2/8} / \thr^2  = 4 \,  n^{-\gamma^2/4} / \thr^2~.
$$ 
Also, by Lemma~\ref{lem:G.bound}, $| G(\tha,b) - (-b \tha) | \leq 2 \abs{\tha} + \phi(0)$.  Therefore,
$$
B_{\alpha, n} \leq  4 \, c \, n^{-\gamma^2/4}/\thr^2  ~+~  (2 \abs{\tha} + \phi(0) ) \Pr\left\{ \Va < -\thr \right\}~,
$$
and the desired result then follows from Lemma \ref{lem:large.deviation.bound} for \texttt{Case 2}.

Now, consider $\tha$ in \texttt{Case 3}; that is, $|\tha| > (1 + \sqrt{2\sigma}/\gamma)\thr$.   We will assume that $\tha > (1 + \sqrt{2 \sigma}/\gamma)\thr$; the case $\tha < - (1 + \sqrt{2 \sigma}/\gamma)\thr$ follows analogously.   As before, it follows from Lemma \ref{lem:bias.bound} that 
$$
\abs{\bar{b} \tha  - G(\tha,b) } \leq  c\, e^{-(1 + \sqrt{2\sigma}/\gamma)^2 \thr^2/2} / \big( (1+\sqrt{2\sigma}/\gamma)^2 \thr^2 \big)  =   c \, n^{-(\gamma+ \sqrt{2\sigma})^2} / \big( (\gamma +\sqrt{2\sigma})^2 (2 \log n) \big) ~.
$$ 
By Lemma~\ref{lem:G.bound}, $| G(\tha,b) - (-b \tha) | \leq 2 \abs{\tha} + \phi(0)$.  Therefore,
$$
B_{\alpha, n} \leq  \frac{ c \, n^{-(\gamma+ \sqrt{2\sigma})^2}}{ (\gamma +\sqrt{2\sigma})^2 (2 \log n) }  +  (2 \abs{\tha} + \phi(0) ) \Pr\left\{ \Va < -\thr \right\}~,
$$
and the desired result then follows from Lemma \ref{lem:large.deviation.bound} for \texttt{Case 3}.   Note that in \texttt{Case~3},  $\abs{\tha} \leq \tha^2$ for  sufficiently large $n$.  Thus, in all three cases, the first term of the bias multiplied by $a_n^8$ converges to zero and we have the desired result for the bias terms in proposition.

\noindent {\bf \underline{Bounding the Variance}:} According to the definition of $\Ta$, 
it follows from Lemma~\ref{lem:var.bound.2} that
\begin{equation}\label{var.decomp}
\begin{split}
&\var_{\tha}(\Ta)\leq 4 \var(A^1_{\alpha,n}) + 4 \var(A^2_{\alpha,n}) + 4 \var(A^3_{\alpha,n}), \text{ where }\\
& A^1_{\alpha,n}= \Sm(\Ua)  \indicator{ \{ |\Va| < \thr\} },~ 
A^2_{\alpha,n}= - b \Ua \indicator{ \{ \Va < - \thr\}},~ \text{and}~
A^3_{\alpha,n}= \bar{b} \Ua \indicator{\{ \Va > \thr\}}.
\end{split}
\end{equation}
To establish the desired result, we will  show that each  term on the RHS is $\smallo{n}$ uniformly in $\alpha$; that is, for $i = 1, 2, 3$, $\lim_{n \to \infty} a_n^8 \, n^{-1} \sup_{\alpha \in [0,1]}  \var(A^i_{\alpha,n}) = 0$.

\noindent {\bf \texttt{Case 1:}} $|\tha| \leq \thr/2$. Since $\Sm(U_{\alpha})=\mathrm{sign}(S(U_{\alpha})) \, \min\{|S(U_{\alpha})|, n\}$, it follows from Lemma~\ref{lem:var.bound.1} that
$
\var \left( A^1_{\alpha,n} \right) \leq  \ex_{\tha} \Sm^2(\Ua) \leq \ex_{\tha} S^2(\Ua) = \smallo{n},
$
where the last equality follows from Lemma \ref{lem:variance.bound}.
Again, by Lemma~\ref{lem:var.bound.1},
\begin{align*}
& \var(A^2_{\alpha,n}) + \var(A^3_{\alpha,n}) \\
&\leq b^2 \ex \left[\Ua^2 \right] \cdot \Pr\{ \Va < -\thr\} + \bar{b}^2 \left[\Ua^2 \right] \cdot \Pr\{ \Va > \thr\} 
\leq \ex\left[\Ua^2 \right] \Pr\{|\Va| > \thr\}\\
&= \left( \tha^2 + 2 \sigma \da^2 \right) \Pr\{|\Va| > \thr\} \leq \left( \thr^2/4  + 2 \sigma \right) \Pr\{|\Va| > \thr\}~,
\end{align*}
where the  equality follows from the definition of $\Ua$.  The desired result then follows from Lemma \ref{lem:large.deviation.bound} for \texttt{Case 1}.

\noindent {\bf \texttt{Case 2:}} $\thr/2 < |\tha| \leq (1 + \sqrt{2\sigma} /\gamma)\thr$. 
Suppose that $\thr/2 < \tha \leq (1 + \sqrt{2\sigma} /\gamma)\thr$; the proof for the case where $ - (1 + \sqrt{2\sigma} /\gamma)\thr \leq  \tha < -\thr/2$ is the same.
By Lemma~\ref{lem:var.bound.1},
$$\var(A^1_{\alpha,n}) \leq  \ex \Sm^2(\Ua) \leq \ex  S^2(\Ua) = \smallo{n}~,$$
where the  equality follows from Lemma \ref{lem:variance.bound}.
By Lemma~\ref{lem:var.bound.1},
$$ 
\var(A^2_{\alpha,n}) \leq {b}^2 \, \ex \left[  \Ua^2 \right] \Pr\{ \Va < -\thr \} \leq \left( 2 \sigma  +  \tha^2 \right)\Pr\{ \Va < -\thr \}~.
$$ 
\pr{For the range of $\tha$ in \texttt{Case 2}, $\tha/n \to 0$ uniformly in $\alpha$, and it follows that}
$$
\lim_{n \to \infty} \sup_{\alpha \,:\, \thr/2 \;<\; \tha \;\leq\; (1 \,+\, \sqrt{2\sigma}/\gamma)\thr } a_n^8 \, n^{-1} \var(A^2_{\alpha,n}) 
= 0\,,
$$
\pr{where the equality follows from Lemma \ref{lem:large.deviation.bound} for \texttt{Case 2}.}
Note that $\var(A^3_{\alpha,n}) \leq 4 \ex[ b^2 \Ua^2] \Pr\{\Va < -\thr\} \leq 4 \ex[ b^2 \Ua^2] \leq  4 \left( 2 \sigma  +  \tha^2 \right) = \smallo{n}$ uniformly in $\alpha$.

\noindent {\bf \texttt{Case 3:}} $|\tha| > (1 + \sqrt{2\sigma} /\gamma)\thr$.  Note that
$$
\var_{\tha}(A^1_{\alpha,n}) \leq \ex[\Sm^2(\Ua) \indicator{\{|\Va| < \thr\}}] \leq n^2  \Pr\{|\Va|<\thr\}~,
$$
and by Lemma \ref{lem:large.deviation.bound} for \texttt{Case 3}, $\lim_{n \to \infty} \sup_{\alpha : |\tha| > (1 \,+\, \sqrt{2\sigma}/\gamma)\thr } {\var{\tha}(A^1_{\alpha,n})}/{n} = 0$.

Note that $\E[\Ua] = \tha$ and $\var(\Ua) = 2 \sigma \da^2 \leq 2 \sigma$. By Lemma~\ref{lem:var.bound.1}, 
\begin{align*}
\var_{\tha}(A^2_{\alpha,n}) &\leq \ex [\Ua^2] \, \Pr\{\Va < -\thr\} \leq (2 \sigma + \tha^2) \Pr\{\Va < -\thr\}\\
\var_{\tha}(A^3_{\alpha,n}) &\leq \var(\bar{b} \Ua) + (\E [ \bar{b} \Ua])^2 \Pr\{ \Va \leq \thr\}  \leq  2 \sigma + (2 \sigma + \tha^2) \Pr\{\Va \leq \thr\}\\
&=  2 \sigma + (2 \sigma + \tha^2) \Pr\{|\Va| \leq \thr\} + (2 \sigma + \tha^2) \Pr\{\Va < -\thr\}~,
\end{align*}
which implies that
$$
\var_{\tha}(A^2_{\alpha,n}) + \var_{\tha}(A^3_{\alpha,n}) \leq 2 \sigma  + (2 \sigma + \tha^2 )\Pr\{\abs{\Va} \leq \thr\} + (2 \sigma + \tha^2 )\Pr\{\Va < -  \thr\}~.
$$
Note that, by Lemma \ref{lem:large.deviation.bound} for \texttt{Case 3},  both $\sup_{\alpha : |\tha| > \left(1 \,+\, {\sqrt{2\sigma} / \gamma}\right) \thr} a_n^8 \,\tha^2 \cdot \Pr \left\{ |\Va| \leq \thr \right\}$ 
and $\sup_{\alpha : \tha > \left(1 \,+\, {\sqrt{2\sigma} / \gamma}\right) \thr} a_n^8 \tha^2 \cdot \Pr \left\{ \Va < - \thr \right\}$ converge to zero as $n$ increases. Thus, we have that 
$\lim_{n \to \infty} \sup_{\alpha \,:\, |\tha| \;>\; (1 \,+\, \sqrt{2\sigma}/\gamma)\thr }  a_n^8 \left( \var_{\tha}(A^2_{\alpha,n}) + \var_{\tha}(A^3_{\alpha,n}) \right)/{n}  = 0$, which is the desired result.
\par
This completes the proof of Proposition~\ref{univariate.origin}. We end this section with a remark on the choice of threshold. The proof will work similarly for $\sqrt{2 \log n}$ thresholds that are scalable with $\sqrt{\sigma_{p,i}}$ and $|\da|$ for $1\leq i \leq n,\, \alpha \in [0,1]$. Our  choice $\lambda_n$ being uniform over $\tau \in [0,\infty]$, however, yields a comparatively cleaner proof. 

%% file: sec-3.tex

In this section, we study the performances of our proposed estimators through numerical experiments. 
In the first example, we display a case where the performance of our proposed $\A$-based estimate is  close to that of the oracle estimator, but the traditional EBML and EBMM estimators perform poorly.  It supports the arguments (provided below Corollaries~\ref{cor:origin.are.oracle} and \ref{cor:data.are.oracle}) that as the formulae of the ML and MM estimates of the hyper parameters do not depend on the shape of the loss functions, they can be significantly different from the ARE-based estimates and hence sub-optimal. We calculate the inefficiency of an estimate $\tau$ of the shrinkage hyperparameter of members in $\mathcal{S}^0$ by comparing it with its corresponding Oracle risk-based estimator:
$\taut_{OR}=\argmin_{\tau \in [0,\infty]} R_n(\thetab,\qhatb(\tau))$. We define:
$$ \text{Inefficency of } \tauh =  \frac{R_n(\thetab,\qhatb(\tauh))-R_n(\thetab,\qhatb(\taut_{OR}))}{\max_{\tau \geq 0} R_n(\thetab,\qhatb(\tau))- \min_{\tau \geq 0} R_n(\thetab,\qhatb(\tau))}  \times 100 \, \%.$$
The measures for the other classes are defined analogously. In the other two examples, we study the performance of our proposed estimators as we vary the model parameters. Throughout this section, we set $\sigma_{f,i}=1$ and $b_i+h_1=1$ for all $i=1,\ldots,n$. The R codes used for these simulation experiments can be downloaded from \url{http://www-bcf.usc.edu/~gourab/inventory-management/}.

\subsection{Example 1}\label{ex1}
Here, we study a simple setup in a homoskedastic model where $\sigma_{p,i}=1/3$ for all $i=1,\ldots,n$. 
We consider two different choices of $n$ (a) $n=20$, which yields comparatively low dimensional models, and  (b) $n=100$, which is large enough to expect our high-dimensional theory results to set in. We consider only two different values for the $\theta_i$: $1/\sqrt{3}$ and $-3 \sqrt{3}$. Also, we design the setup such that $b_i$ is related to the $\theta_i$: when  $\theta_i = 1/\sqrt{3}$, $b_i = 0.51$ and when $\theta_i = -3 \sqrt{3}, \, b_i = 0.99$. For the case when $n=20$, we consider $(\thetab,\bm{b})$ with $18$ replicates of the $(\theta_i, b_i)$ pair of $(1/\sqrt{3},  0.51)$ and 2 replicates of $(-3\sqrt{3},  0.99)$. For $n=100$, we have $90$ replicates of the former and $10$ replicates of the latter. Note that in both the cases, the mean of $\thetab$  across dimensions is $0$.
\par
In this homoskedastic setup, the MM and ML estimates of the hyperparameter are identical. In Table~\ref{table1}, we present their relative inefficiencies as well as that of the $\A$ with respect to the Oracle risk estimate. For computation of the $\A$ risk estimates, $5$ Monte-Carlo simulations were used for the evaluation of the unconditional expectation in the Rao-Blackwellization step. In Table~1, based on $50$ independent simulation experiments, we report the mean and standard deviation of the estimates as well as their inefficiency percentages. The EBML/EBMM perform very poorly in both cases. When $n=100$, the $\A$-based estimates are close to the Oracle risk-based estimates and are quite efficient.
When $n=20$, the $\A$ method is not as efficient as before but still performs remarkably better than the EBML/EBMM methods.  The plots of the univariate risks of $\qhat_i(\tau)$ for the $(\theta_i, b_i)$ pairs $(1/\sqrt{3},  0.51)$ and $(-3\sqrt{3},  0.99)$ (as $\alpha_i=\tau/(\tau+\sigma_{p,i})$ varies) are very different (see Figure~\ref{fig1}).  For the former, the oracle minimizer is at $\alpha_{OR}=0.51$; that is,  $\tau_{OR} =0.35$. For the latter, the oracle minimizer is at $\alpha_{OR}=1$; that is,   $\tau_{OR}=\infty$. The  multivariate risk plot of our setup is different than those of the two univariate risk plots but is closer to the former than to the later.  ARE approximates this multivariate risk function well and does a good job in estimating the shrinkage parameter. However, the ML/MM estimate of the hyperparameter is swayed by the extremity of fewer $(\theta_i, b_i)=(-3\sqrt{3},  0.99)$ cases and fail to properly estimate the shrinkage parameter in the combined multivariate case.

\begin{table*}[ht!]
\caption{Comparison of the performances of $\A$-, MM- and ML-based estimates with the Oracle risk estimator in Example~1. The mean and standard deviation (in \pr{parentheses}) across $50$ independent simulation experiments are reported.}
\label{table1}
\begin{tabular}{|c||c|c||c|c|}
 \hline
METHODS  & \multicolumn{2}{c||}{ $n = 20$ } &  \multicolumn{2}{c|}{ $ n = 100 $ } \\[1ex]
 \cline{2-3}  \cline{4-5}
 & Inefficiency  (\%) &  $\hat{\tau}$ & Inefficiency  (\%) &  $\hat{\tau}$ \\[1ex]
\hline & &  & & \\
\multirow{1}{*}{ARE}  & 16.78 (30.42) & 1.214 (4.823) & 1.15 (2.57) & 0.344 (0.079) \\[1ex] 
\hline & &  & & \\
\multirow{1}{*}{MM/ML}  & 48.01  (3.55) & 0.037 (0.006) & 47.96 (2.01) & 0.037 (0.003) \\[1ex] 
\hline & &  & & \\
\multirow{1}{*}{ORACLE}  & - &  0.296  (0.000)  &  - & 0.296 (0.000) \\[1ex] 
\hline
\hline
\end{tabular}
\end{table*}

 \begin{figure}[ht!]
 \includegraphics[width=\textwidth]{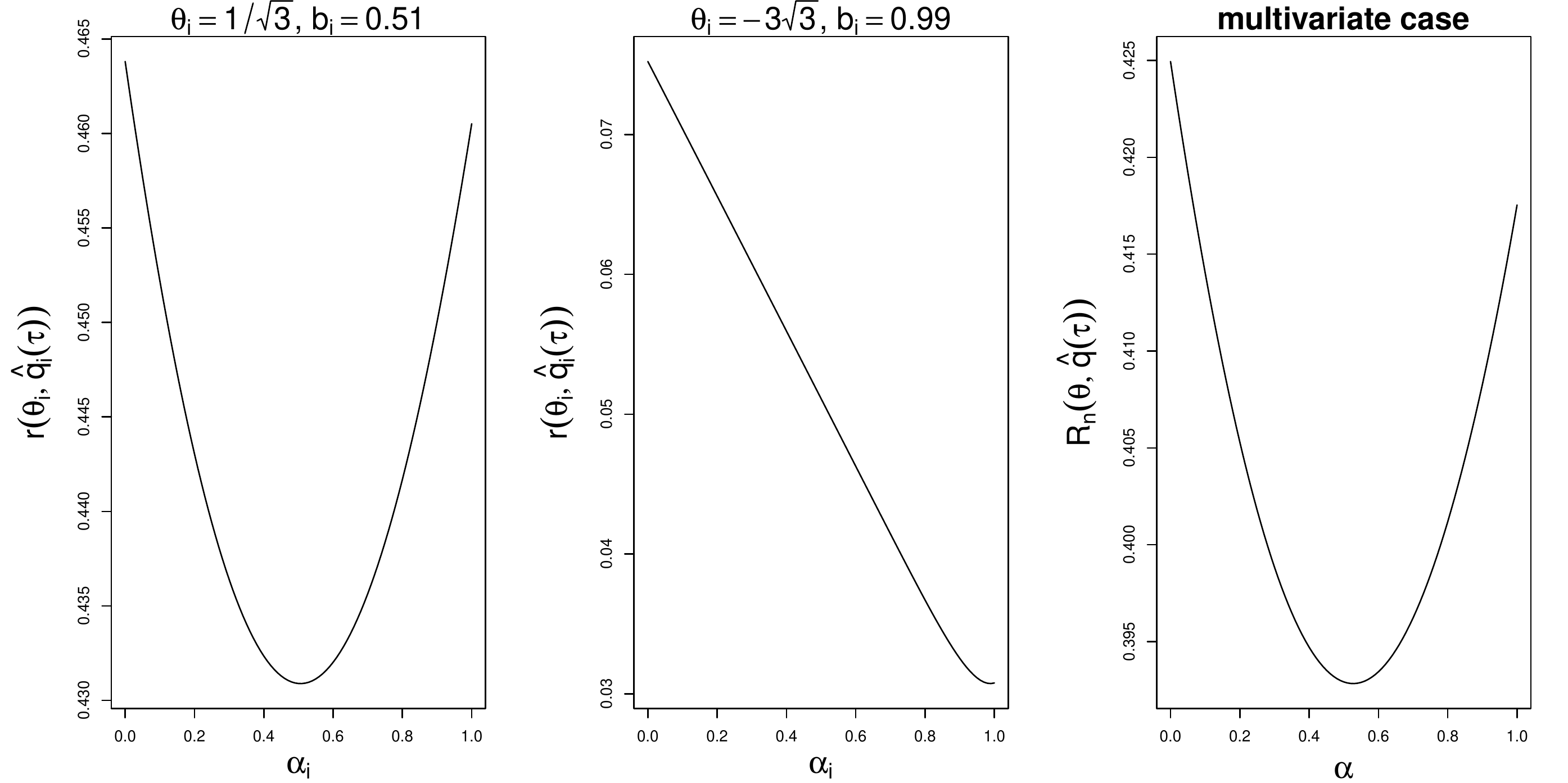}
 \caption{From left to right we have the following: the plots of the univariate risks of $\qhat_i(\tau)$ for the $(\theta_i, b_i)$ pairs $(1/\sqrt{3},  0.51)$ and $(-3\sqrt{3},  0.99)$, respectively, as $\alpha_i=\tau/(\tau+\sigma_{p,i})$ varies and the plot of the multivariate risk of $\qhatb(\tau)$ for the $(\thetab,\bm{b})$  choices described in Example~\ref{ex1}.}\label{fig1}
 \end{figure}

\subsection{Example 2} We consider homoskedastic models with $\sigma_{f,i}=1$ and $\sigma_{p,i}=\sigma_p$ for all $i=1,\ldots, n$. We vary $\sigma_p$ to numerically test the performance of the $\A$  methodology when Assumption~A3 of Section~\ref{subsec:main.results} is violated. We generate $\{\theta_i: i=1,\ldots,n\}$ independently from $N(0,1)$, and 
$\{b_i: i=1,\ldots,n\}$ are generated uniformly from $[0.51,0.99]$. Table~\ref{table2} reports the mean and standard deviation (in brackets) of the inefficiency percentages across $20$ simulation experiments from each regime. We see that the ARE methodology does not work for larger values of the ratio $\sigma_p/\sigma_f$ and starts performing reasonably when $\sigma_p /\sigma_f\leq 1/3$, which is quite higher than the prescribed theoretical bound in \eqref{assumption.A2}. 

\begin{table*}[h!]
	\caption{Inefficiency \pr{(\%)} of $\A$ estimators in Example~2 as the ratio $\sigma_p/\sigma_f$ varies.}
	\label{table2}
	\begin{tabular}{ ||c || c || c|| }
		\hline 
		$\sigma_p/\sigma_f$&  $\qquad \qquad n=20 \qquad \qquad$ &  $\qquad \qquad n=100 \qquad \qquad$  \\
		\hline 
		1/1 & 75.34 (28.55) & 88.88 (14.70) \\ \hline
		1/2 & 31.70 (20.85) & 27.81  (07.95) \\ \hline
		1/3 & 19.21 (14.44) & 12.91  (03.63) \\ \hline
		1/4 &   06.93  (03.58) & 07.43  (02.07) \\ \hline
		1/5 & 05.56  (03.93) & 04.36 (01.38) \\ \hline
		1/6 & 04.07  (03.06) & 03.06 (00.97) \\ \hline    
	\end{tabular}
\end{table*}

\subsection{Example 3} We now study the performance of our proposed $\AG$ methodology in $6$ heteroskedastic models, which are modified predictive versions of those used in Section~7 of \citet{xie2012}. Here, $\{b_i: i=1,\ldots,n\}$ are generated uniformly from $[0.51,0.99]$ and $\sigma_{f,i}=1$ for all $i$.  Also, based on Example~2, we impose the constraint $\max \{\sigma_{p,i}/\sigma_{f,i}: 1 \leq i \leq n \} \leq 1/3$. Next, we  outline the $6$ experimental setups by describing the parameters used in the predictive model~of~\eqref{pred.model}: 

\textit{Case I.}  $\thetab$ are i.i.d. from Uniform(0,1), and  $\sigma_{p,i}$  are i.i.d. from Uniform(0.1,1/3). 

\textit{Case II.} $\thetab$ are i.i.d. from N(0,1), and  $\sigma_{p,i}$  are i.i.d. from Uniform(0.1,1/3). 

\textit{Case III.} Here, we bring in dependence between $\sigma_{p,i}$ and $\thetab$. We generate 
$\{\sigma_{p,i}: 1 \leq i \leq n\}$ independently from  Uniform(0.1,1/3) and 
$\theta_i=5\, \sigma_{p,i}$ for $i=1,\ldots,n$.

\textit{Case IV.} Instead of uniform distribution in the above case, we now generate $\{\sigma_{p,i}: 1 \leq i \leq n\}$ independently from 
$ \text{Inv-}\chi^2_{10}$, which is the conjugate distribution for normal variance.  

\textit{Case V.} This model reflects grouping in the data. We draw the past variances independently from the $2$-point distribution $2^{-1}(\delta_{0.1}+\delta_{0.5})$, and the $\theta_i$ are drawn conditioned on the past variances:
$$(\theta_i|\sigma_{p,i}=0.1) \sim N(0,0.1) \quad \text{ and } \quad (\theta_i|\sigma_{p,i}=0.5) \sim N(0,0.5).$$
Thus, there are two groups in the data.
 
\textit{Case VI.} 
In this example, we assess the sensitivity in the performance of  the $\AG$ estimators to the normality assumption by allowing $\X$ to depart from the normal model of \eqref{pred.model}.
We generate $\{\sigma_{p,i}: 1 \leq i \leq n\}$ independently from  Uniform(0.1,1/3) and $\theta_i=5\, \sigma_{p,i}$ for $i=1,\ldots,n$. The past observations are generated independently from
$$X_i \sim \text{Uniform}\big(\theta_i - \sqrt{3 \, \sigma_{p,i}}, \, \theta_i + \sqrt{3 \, \sigma_{p,i}} \big) \text{ for } i=1,\ldots, n.$$
\par
Table~\ref{table3} reports the mean and standard deviation (in brackets) of the inefficiency percentages of our methodology in $20$ simulation experiments from each of the $6$ models. The $\AG$ estimator performs reasonably well across all $6$ scenarios. 

\begin{table*}[ht!]
	\caption{Inefficiency \pr{(\%)} of $\AG$ estimators in $6$ different heteroskedastic models of Example~3.}
	\label{table3}
	\begin{tabular}{ |l | c | c| }
		\hline
		& $\qquad \qquad n=20 \qquad \qquad$ &  $\qquad \qquad n=100 \qquad \qquad$\\
		\hline
		Case I &  02.79  (02.70) & 01.81  (01.83) \\ \hline
		Case II &  12.90 (21.16) & 11.31  (01.73)\\ \hline
		Case III & 13.90 (19.21) & 07.84  (02.08) \\ \hline
		Case IV &  08.75 (14.26) & 10.47 (20.65) \\ \hline
		Case V &  03.80  (04.32) & 01.52  (03.13) \\ \hline
		Case VI & 06.20  (08.45) & 08.74  (03.19) \\ \hline
	\end{tabular}
\end{table*}

%% file: sec-4.tex

We first describe the $\AD(\eta,\tau)$ risk estimation procedure. Note that, by Lemma~\ref{linear.estimates}, for any fixed $\eta \in \Real$, the risk of estimators in $\mathcal{S}$  is related to risk of estimators in $\mathcal{S}^0$ as $R_n(\thetab, \qhatb(\eta,\tau))=R_n(\thetab -\eta, \qhatb(\tau))$. We rewrite the $\A$ risk estimate defined in \eqref{def.A} by explicitly denoting the dependence on $\X$ as
\begin{align}
\Ahat_n(\tau,\X) = \frac{1}{n} \sum_{i=1}^n (b_i+h_i) (\sigma_{f,i}+\sigma_{p,i} \alpha_i^2)^{1/2}  \, \That_i(X_i,\tau).
\end{align}
The $\AD$ risk estimate is defined as
$\AD_n(\eta,\tau,\X)=\A_n(\tau,\X-\eta)$. Henceforth, whenever we use the relation between $\AD$ and $\A$, we will explicitly denote the dependence of the risk estimates on the data. Otherwise, we will stick to our earlier notation where the dependence on the data is kept implicit.  We next prove Theorem~\ref{data.loss.are}.

 \subsection{Proof of Theorem~\ref{data.loss.are}}
The proof  follows from the following two lemmas. The first one shows that our proposed risk estimate does a good job in estimating the risk of estimators in $\mathcal{S}$. This lemma \pr{holds} for all estimates $q(\eta,\tau)$ in $\mathcal{S}$ and does not need any restrictions on $\thetab$.  The second lemma shows that the loss is uniformly close to the risk. It needs the restriction $|\eta| \leq a_n$ on estimates in $\mathcal{S}$ and also the assumption $\text{A2\,}$ on $\thetab$.

\begin{lem}\label{data.risk.are}
	Under Assumptions A1 and A3 with $a_n=\log \log n$, for all $\thetab$, 
	$$  \lim_{n \to \infty} \sup_{\tau \in [0,\infty], \, \eta \in \Real} a_n^8\, \ex \left [  \big(\AD_n(\eta,\tau)-R_n(\thetab,\qhatb(\eta,\tau))\big)^2\right ]=0~.$$
\end{lem}

\begin{lem}\label{data.risk.loss}
Under Assumption A1, for all $\thetab$ satisfying Assumption $\text{A2}$, 
$$ \lim_{n \to \infty}  \sup_{\tau \in [0,\infty], |\eta| \leq a_n} a_n^4 \, \ex \big |  R_n(\thetab, \qhatb(\eta,\tau)) -  L_n(\thetab, \qhatb(\eta,\tau)) \big| = 0~ \text{ where } a_n=\log \log n.$$
\end{lem}

The proof of Lemma~\ref{data.risk.are} is provided in Appendix~\ref{append-data-driven}. \textit{For the proof of Lemma~\ref{data.risk.loss}}, we show uniform convergence of the expected absolute loss over the set of location parameters $\{|\eta| \leq a_n\}$ by undertaking a 
moment-based approach. Here, we show that  for any  $\thetab$ obeying $\text{Assumption A2\,}$
$$ \sup_{\tau \in [0,\infty], |\eta| \leq a_n} a_n^8 \var_{\thetab}(L_n(\thetab, \qhatb(\eta,\tau)) \to 0 \text{ as } n \to \infty~,$$
from which the proof of the lemma follows easily.
Now,  note that, due to independence across coordinates, we have
$$ \var_{\thetab}\left( L_n(\thetab, \qhatb(\eta,\tau) \right)=n^{-2} \sum_{i=1}^n \var_{\theta_i}\left( l_i(\theta_i,  \qhat_i(\eta,\tau))\right) \leq n^{-2} \sum_{i=1}^n \ex_{\theta_i}\left[ l^2_i(\theta_i,  \qhat_i(\eta,\tau)) \right].$$
By definition of the predictive loss, we have the following relation between the loss  of estimators in $\mathcal{S}$ and $\mathcal{S}^0$: 
$\ex_{\theta_i}[l^2_i(\theta_i,  \qhat_i(\eta,\tau))]=\ex_{\theta_i}[l^2_i(\theta_i - \alpha_i (\tau) \eta,  \qhat_i(\tau))]$ and using the inequality in Equation~\eqref{eqn.loss.bound} of the Appendix we see that it is dominated by  $\bigo{1+\ex[\theta_i- \alpha_i (\tau)  \eta]^2} \leq \bigo{1 + 2 \,\ex_{\theta_i}[\theta^2_i + \eta^2]}$ as $|\alpha_i (\tau)|  \leq 1$ for any $\tau \in [0,\infty]$.
Thus, we have
$$\var_{\thetab}\left( L_n(\thetab, \qhatb(\eta,\tau)\right) \leq \mathcal{O}{\bigg(n^{-2} \sum_{i=1}^n \theta^2_i
+ n^{-1} a_n^2\bigg)}\text{ for all } \tau \in[0,\infty], |\eta| \leq a_n.$$
For any $\thetab$ satisfying $\text{assumption A2}$, both the terms in the RHS, even after being multiplied by $a_n^8$ uniformly converge to $0$, which completes the proof of Lemma~\ref{data.risk.loss}.

We next present the proof of the decision theoretic properties of our estimators. We first define discretized set $\net$ used in the construction of the ARE estimator. $\net=\nett_{n,1} \otimes \nett_{n,2}$ is constructed as a product grid over the space of  $\eta$ and $\tau$.
We consider an invertible transformation on $\tau$ and re-parametrize it by $\taut=\tau/(\tau+1)$. As $\tau$ varies over $[0,\infty]$, $\taut$ is contained in $[0,1]$.  We  construct an equi-spaced discretized set $\{0=\taut_1 \leq \taut_2 \leq \cdots \leq \taut_m \leq 1\}$ with this transformed variable where $\taut_i = (i+1) \delta_{n,2} $ and $m=\lceil 1/\delta_{n,2} \rceil$ where
\begin{align}
\delta_{n,2} = \big\{2 \, C_1 \, C_2  (2 \phi(0) +  C_3 + \sqrt{a_n} C_4 + a_n +a_n^2 ) \big \}^{-1}  \text{ where } a_n = \log \log n
\end{align}
and $C_1$, $C_2$ and $C_3$ are defined in \eqref{eq:constants.temp}. We retransform the aforeconstructed grid on $\taut$  back to get the set $\nett_{n,1}$ on $\tau \in [0,\infty]$. On $\eta$ the grid, $\nett_{n,1}$ is equispaced in the interval $[-a_n,a_n]$ with the spacing equalling $\delta_{n,1}=(2 C_1 a_n)^{-1}$.
The cardinality of the set $\nett_n$ is: $|\nett_n|=|\nett_{n,1}|\times |\nett_{n,2}| = 2 a_n/\delta_{n,1} \times \delta_{n,2}^{-1} = O(a_n^4)$ \text{ as } $n \to \infty$. We conduct our computations for the ARE estimate by restricting this larger grid $\nett_n$ to the smaller set $(\nett_{n,1} \cap \hat{M}_n) \otimes \nett_{n,2}$ where the $\eta$ values lie in the set $\hat{M}_n$ defined in Section~\ref{further.results}. 
\par
We define the corresponding discretized version of the oracle estimator as
$$
(\eta_{n}^{\nett},\tau_{n}^{{\nett}}) =\argmin_{(\eta,\tau) \in (\nett_{n,1} \cap \hat{M}_n) \otimes \nett_{n,2}} L_n(\thetab,\qhatb(\eta,\tau))~.
$$
The following lemma whose proof is presented in the Appendix~\ref{append-data-driven}  shows that the difference between $L_n(\thetab,\qhatb(\eta_{n}^{\nett},\tau_{n}^{{\nett}}))$ and  the original oracle loss $L_n(\thetab,\qhatb(\eta^{\sf DOR}_n, \tau^{\sf DOR}_n))$ is asymptotically controlled at any prefixed level.

\begin{lem}\label{lem:sec4.temp}
	For any fixed $\epsilon > 0$,
	\begin{align*}
	\texttt{I.} \quad & P\big\{ L_n(\thetab,\qhatb(\eta_{n}^{\nett},\tau_{n}^{{\nett}})) - L_n(\thetab,\qhatb(\eta^{\sf DOR}_n, \tau^{\sf DOR}_n)) > \epsilon \big\} \to 0 \text{ as } n \to \infty	 \;\; \text{ and },\\
	\texttt{II.} \quad  & \ex |L_n(\thetab,\qhatb(\eta_{n}^{\nett},\tau_{n}^{{\nett}})) - L_n(\thetab,\qhatb(\eta^{\sf DOR}_n, \tau^{\sf DOR}_n))| \to 0 \text{ as } n \to \infty.
	\end{align*}
\end{lem}

We now present the proof of the decision theoretic properties of our estimators.

\subsection{Proof of Theorem~\ref{data.are.oracle.loss}}
We know that: 
\begin{align*}
&\Pr \left\{ L_n(\thetab,\qhatb(\etah^{\ADD}_n,\tauh^{\ADD}_n) )\geq L_n(\thetab,\qhatb(\eta^{\sf DOR}_n,\tau^{\sf DOR}_n))+ \epsilon \right\}\qquad \qquad \qquad \qquad \qquad \qquad \\
&\leq \Pr \left\{ L_n(\thetab,\qhatb(\etah^{\ADD}_n,\tauh^{\ADD}_n) )\geq L_n(\thetab,\qhatb(\eta_{n}^{\nett},\tau_{n}^{{\nett}})) + \epsilon/2 \right\}\\
&+ \Pr \left\{ L_n(\thetab,\qhatb(\eta_{n}^{\nett},\tau_{n}^{{\nett}})) \geq L_n(\thetab,\qhatb(\eta^{\sf DOR}_n,\tau^{\sf DOR}_n))+ \epsilon/2 \right\}.
\end{align*}
the second term converegs to $0$ by Lemma~\ref{lem:sec4.temp}. We concentrate on the first term. Note, 
by construction, $\AD_n(\etah^{\ADD}_n,\tauh^{\ADD}_n) \leq \AD_n(\eta_{n}^{\nett},\tau_{n}^{{\nett}})$. Thus, for any fixed $\epsilon > 0$, the first term is bounded above by
\begin{align*}
&\Pr \left\{ A_n(\thetab,\etah^{\ADD}_n,\tauh^{\ADD}_n) \geq B_n(\thetab,\eta_{n}^{\nett},\tau_{n}^{{\nett}}) + \epsilon/2 \right\}~, \text{ where } \qquad \qquad \qquad \qquad \qquad \qquad \\
&A_n(\thetab,\etah^{\ADD}_n,\tauh^{\ADD}_n) = L_n(\thetab,\qhatb(\etah^{\ADD}_n,\tauh^{\ADD}_n)) - \AD_n(\etah^{\ADD}_n,\tauh^{\ADD}_n), \text{ and }
\\
&B_n(\thetab,\eta_{n}^{\nett},\tau_{n}^{{\nett}}) L_n(\thetab,\qhatb(\eta_{n}^{\nett},\tau_{n}^{{\nett}})) - \AD_n(\eta_{n}^{\nett},\tau_{n}^{{\nett}}).
\end{align*}
Now, using Markov inequality, we have
$$ \Pr \left\{ A_n(\thetab,\etah^{\ADD}_n,\tauh^{\ADD}_n) \geq B_n(\thetab,\eta_{n}^{\nett},\tau_{n}^{{\nett}}) + \epsilon/2\right\} \leq  2\epsilon^{-1} \ex|A_n(\thetab,\etah^{\ADD}_n,\tauh^{\ADD}_n) - B_n(\thetab,\eta_{n}^{\nett},\tau_{n}^{{\nett}})|,$$
which, again, by the triangle inequality  is less than $$ 4 \epsilon^{-1}  \ex \bigg[ \sup_{(\eta,\tau) \in (\nett_{n,1} \cap \hat{M}_n) \otimes \nett_{n,2}}|\AD_n(\eta,\tau)-L_n(\thetab,\qhatb(\eta,\tau))|\bigg].$$
We can bound the supremum by the sum of the absolute loss over the grid and so, the above term is less than:
\begin{align*}
4 \epsilon^{-1}  |\nett_n| \sup_{(\eta,\tau) \in \nett_n} \ex \bigg[ |\AD_n(\eta,\tau)-L_n(\thetab,\qhatb(\eta,\tau))|\bigg]\\
=\bigo{ \epsilon^{-1}  a_n^4 \sup_{\tau \in [0, \infty], |\eta| \leq a_n} \ex \big[ |\AD_n(\eta,\tau)-L_n(\thetab,\qhatb(\eta,\tau))|\big]}
\end{align*}
which, by Theorem~\ref{data.loss.are} converges to 0 as $n \to \infty$, and we have the required result.

\subsection{Proof of Theorem~\ref{data.are.oracle.risk}} We decompose $L_n(\thetab,\qhatb(\etah^{\ADD}_n,\tauh^{\ADD}_n)) - L_n(\thetab,\qhatb(\eta^{\sf DOR}_n,\tau^{\sf DOR}_n))$ into two positive components:
\begin{align*}
\{L_n(\thetab,\qhatb(\etah^{\ADD}_n,\tauh^{\ADD}_n) )- L_n(\thetab,\qhatb(\eta_{n}^{\nett},\tau_{n}^{{\nett}}))\} + \{L_n(\thetab,\qhatb(\eta_{n}^{\nett},\tau_{n}^{{\nett}}))-L_n(\thetab,\qhatb(\eta^{\sf DOR}_n,\tau^{\sf DOR}_n))\}.
\end{align*}
The expectation of the second term converges to $0$ by Lemma~\ref{lem:sec4.temp}. For the first term we decompose the difference of the losses into $3$ parts: 
\begin{align*}
& \lefteqn{L_n(\thetab,\qhatb(\etah^{\ADD}_n,\tauh^{\ADD}_n))- L_n(\thetab,\qhatb(\eta_{n}^{\nett},\tau_{n}^{{\nett}})}\\
&  = \left( L_n(\thetab,\qhatb(\etah^{\ADD}_n,\tauh^{\ADD}_n)) - \AD_n(\etah^{\ADD}_n,\tauh^{\ADD}_n) \right) - \left( L_n(\thetab,\qhatb(\eta_{n}^{\nett},\tau_{n}^{{\nett}})\} - \AD_n(\eta_{n}^{\nett},\tau_{n}^{{\nett}}) \right) \\
& \qquad + \left( \AD_n(\etah^{\ADD}_n,\tauh^{\ADD}_n) - \AD_n(\eta_{n}^{\nett},\tau_{n}^{{\nett}}) \right).
\end{align*}
As the third term is less than $0$, so $\ex \left[ L_n(\thetab,\qhatb(\etah^{\ADD}_n,\tauh^{\ADD}_n))- L_n(\thetab,\qhatb(\eta_{n}^{\nett},\tau_{n}^{{\nett}}))\right]$ is bounded above by
$2 \, \ex \{ \sup_{(\eta,\tau) \in \Lambda_n} \big| \AD_n(\eta,\tau)-L_n(\thetab,\qhatb(\eta,\tau)) \big |\}$ which is less than
$$2  \ex \bigg \{ \sum_{(\eta,\tau) \in \Lambda_n} \big| \AD_n(\eta,\tau)-L_n(\thetab,\qhatb(\eta,\tau)) \big |\bigg\}.$$
It converges to $0$ by Theorem~\ref{data.loss.are}. Hence, the result follows.

\subsection{Proof of Corollary~\ref{cor:origin.are.oracle}}
The results follow directly from Theorems \ref{data.are.oracle.loss} and \ref{data.are.oracle.risk}  as $(\eta^{\sf DOR}_n,\tau^{\sf DOR}_n)$ minimizes the loss $L_n(\thetab,\qhatb(\eta,\tau))$ among the class $\mathcal{S}$.

%% file: sec-5.tex

By \eqref{eq:loss-function}, the predictive loss an estimator $\qhatgb(\tau)$ in  $\mathcal{S}^G$  is given by
$L_n(\thetab,\qhatgb(\tau)) = \frac{1}{n}\sum_{i=1}^n l_i(\theta_i,\qhatg_i(\tau))$,  where
$$l_i(\theta_i,\qhatg_i(\tau)) = \sigma_{f,i}^{1/2}\,(b_i+h_i) \, G(\sigma_{f,i}^{-1/2}(\qhat_i(\tau)+(1-\alpha_i)\Xm-\theta_i),\bt).$$
We define a surrogate of the loss by plugging in $\Tm$ -- the mean of the unknown parameter $\thetab$ in the place of $\Xm$: $ \Lt_n(\thetab,\qhatgb(\tau)) = \frac{1}{n}\sum_{i=1}^n \lt_i(\theta_i,\qhatg_i(\tau))$, where 
$$\lt_i(\theta_i,\qhatg_i(\tau)) = \sigma_{f,i}^{1/2}\,(b_i+h_i) \, G(\sigma_{f,i}^{-1/2}(\qhat_i(\tau)+(1-\alpha_i)\Tm-\theta_i),\bt).$$
The following lemma, whose proof is provided in Appendix~\ref{append-grand-mean}, shows the surrogate loss is uniformly close to the actual predictive loss.
\begin{lem}\label{lem.mean.1}
For any $\thetab \in \Real^n$ and $\qhatgb(\tau) \in \mathcal{S}^G$, we have
\pr{$$ \lim_{n \to \infty} \E \left[ \sup_{\tau \in[0,\infty]} \big|L_n(\thetab,\qhatgb(\tau)) - \Lt_n(\thetab,\qhatgb(\tau)) \big|  \right]
= 0~.$$}
\end{lem}
We define the associated surrogate risk by $\rt_i(\theta_i,\qhatg_i(\tau))=\ex_{\thetab} \lt_i(\theta_i,\qhatg_i(\tau))$. From Lemma~\ref{linear.estimates}, it follows that this surrogate risk is connected with the risk function of estimators in $\mathcal{S}$ as:
$\rt_i(\theta_i,\qhatg_i(\tau))=r(\theta_i-\Tm,\qhat(\tau))$. Thus, the associated multivariate surrogate risk $ \Rt_i(\thetab,\qhatgb(\tau)) = \sum_{i=1}^n \rt_i(\theta_i,\qhatg_i(\tau))$ equals $R_n(\theta-\Tm,\qhatb(\tau))$.  Also by Lemma~\ref{lem.mean.1}, it follows that for any $\thetab\in \Real^n$
\begin{align}\label{temp.1}
\lim_{n \to \infty} \E \left[  \sup_{\tau \in[0,\infty]} \big|R_n(\thetab,\qhatgb(\tau)) - \Rt_n(\thetab,\qhatgb(\tau)) \big| \right] = 0~. 
 \end{align}
Now we will describe our proposed $\AG$ estimator. Explicitly denoting the dependence of the estimators on the data,  for any fixed value of $\tau \in [0,\infty]$, we define 
$\AGhat_n(\tau,\X)=\Ahat_n(\tau,\X-\eta)|_{\eta=\Xm}$. Note that $\X$ and $\Xm$ are correlated, and $\X-\Xm$ has a normal distribution with a non-diagonal covariance structure. However, we can still use the asymptotic risk estimation procedure described in Section~\ref{sec-2} by just plugging in the value of $\Xm$. We avoid the complications of incorporating the covariance structure in our calculations by cleverly using the concentration properties of $\Xm$ around $\Tm$. To explain this approach, we again define a surrogate to our $\AG$ estimator $\Ahat_n(\tau,\X-\eta)|_{\eta=\Xm}=\sum_{i=1}^n c_i \, \That_i(X_i-\eta,\tau)|_{\eta=\Xm}$ by
$$\AT_n(\tau,\X-\Tm)=\sum_{i=1}^n c_i \, \TT_i(X_i-\Tm,\tau)~,$$
where we plugin $\Tm$ in the place of $\Xm$. Here, $c_i=(b_i+h_i) \sqrt{\sigma_{f,i}+\sigma_{p,i} \alpha_i(\tau)^2}$. Note that $\AT$ and $\TT$  have the same functional form as $\Ahat$ and $\That$, respectively, but with $\Xm$ replaced by $\Tm$ and so are not estimators. We now present the proof of Theorem~~\ref{mean.loss.are}.

\textit{Proof of Theorem~\ref{mean.loss.are}.} 
We will prove the theorem by establishing
\begin{align*}
\textsf{(a)} & \quad  \lim_{n \to \infty}  \ex \big \{ \sup_{\tau \in \nett_n}\big|L_n(\thetab,\qhatgb(\tau)) - R_n(\thetab,\qhatgb(\tau)) \big|\} = 0 \;\text{ and,} \\
\textsf{(b)} & \quad \lim_{n \to \infty}  \ex\big \{ \sup_{\tau \in \nett_n}\big|R_n(\thetab,\qhatgb(\tau)) - \AGhat(\tau)\big|\}= 0. \text{\hspace{10cm}}
\end{align*}
For the proof of \textsf{(a)}, based on \eqref{temp.1} and Lemma~\ref{lem.mean.1}, it suffices to show
$$\lim_{n \to \infty} \ex \big \{ \sup_{\tau \in \nett_n} \big|\Lt_n(\thetab,\qhatgb(\tau)) - \Rt_n(\thetab,\qhatgb(\tau)) \big|\} = 0\;.$$
We will prove it by showing:
$$\lim_{n \to \infty} |\nett_n| \sup_{\tau \in [0,\infty]}\ex \big|\Lt_n(\thetab,\qhatgb(\tau)) - \Rt_n(\thetab,\qhatgb(\tau)) \to 0 \text{ as } n \to \infty.$$
Recalling, $|\nett_n|=O(a_n)$, we show that as $n \to \infty$,  $a_n^2 \var_{\thetab}(\Lt_n(\thetab,\qhatgb(\tau)))$ converges to $0$ uniformly over $\tau$ for any $\thetab$ satisfying Assumption $\text{A3\,}$. 
Again, as in the proof of Lemma~\ref{data.risk.are}, we have the bound
$$\var_{\thetab}(\Lt_n(\thetab,\qhatgb(\tau))) \leq  \mathcal{O} \bigg(\frac{1}{n^2} \sum_{i=1}^n \ex_{\theta_i}(\theta_i  - \alpha_i (\tau)  \Tm)^2\bigg).$$
As $|\alpha_i (\tau)|  \leq 1$  for all $\tau \in [0,\infty]$, the RHS above is at most $\mathcal{O} (n^{-2} \sum_{i=1}^n \theta_i^2 + \Tm^2/n )$. Even after being scaled by $a_n^2$, it converges to $0$ as $n \to \infty$ for any $\thetab$ satisfying Assumption~$\text{A3}$.
\par
Now for the proof of \textsf(b),  using \eqref{temp.1} as $n \to \infty$, we have
\begin{align*}
 & \ex \big\{\sup_{\tau \in \nett_n} \big|R_n(\thetab,\qhatgb(\tau)) - \AGhat(\tau)\big| \big\}\to  \ex\big\{\sup_{\tau \in \nett_n}\big|\Rt_n(\thetab,\qhatgb(\tau)) - \AGhat(\tau,\X)\big|\big\}\\
 & \leq |\nett_n| \sup_{\tau \in[0,\infty]} \ex \big|R_n(\thetab-\Tm,\qhatb(\tau)) - \AGhat(\tau,\X)\big|~,
\end{align*}
which is bounded above by the sum of $ |\nett_n| \sup_{\tau \in[0,\infty]} \ex_{\thetab}|\AT_n(\tau,\X-\Tm) - \AG(\tau)|$ and 
$ |\nett_n|\sup_{\tau \in[0,\infty]} \ex_{\thetab} |R_n(\thetab-\Tm,\qhatb(\tau)) - \AT_n(\tau,\X-\Tm)|$. 
Again,  by Lemma~\ref{data.risk.are}, the second term converges to $0$ as $n \to \infty$. The first term is bounded above by
$$ |\nett_n| \sup_{\tau \in[0,\infty]} \frac{1}{n}\sum_{i=1}^n c_i \, \ex_{\thetab} \bigg|(\Xm-\Tm)\cdot \bigg [ \frac{\partial }{\partial \eta} \That_i(X_i-\eta,\tau)\bigg]_{\eta=\mu_i}\bigg|~,$$
where each $\{\mu_i: 1 \leq i \leq n\}$ lies between $\Tm$ and $\Xm$. Using Cauchy-Schwarz inequality, the above term is less than
$$ \lim_{n \to \infty} |\nett_n| \sup_{\tau \in[0,\infty]} \frac{1}{n}\sum_{i=1}^n c_i \,  \bigg\{\ex_{\thetab}(\Xm-\Tm)^2 \cdot \ex_{\thetab}\bigg [ \frac{\partial }{\partial \eta} \That_i(X_i-\eta,\tau)\bigg]^2_{\eta=\mu_i}\bigg\}^{1/2} = 0\;.$$
As $c_i$ are bounded by Assumptions A1 and A3 and $|\nett_n|=\bigo{a_n}$, the asymptotic convergence above follows by using $\ex_{\thetab}(\Xm-\Tm)^2=n^{-1}$ and the following lemma, whose proof is provided in Appendix~\ref{append-grand-mean}.

\begin{lem}\label{lem.mean.2}
For any $\thetab \in \Real^n$ and  $\mu_i$ lying in between $\Xm$ and $\Tm$ for all $i = 1, \ldots, n$
$$\lim_{n \to \infty }  n^{-1} \,a_n^2\, \bigg\{\sup_{1 \leq i \leq n}\;\sup_{\tau \in [0,\infty]} \ex_{\thetab} \bigg [ \frac{\partial }{\partial \eta} \That_i(X_i-\eta,\tau)\bigg]^2_{\eta=\mu_i} \bigg\}= 0\;.$$
\end{lem}

This completes the proof of Theorem~\ref{mean.loss.are}. 

The proof of Theorem~\ref{mean.decision.theoretic.properties} follows similarly from the proofs of Theorems~\ref{data.are.oracle.loss}, \ref{data.are.oracle.risk} and Corollary~\ref{cor:data.are.oracle} and is not presented here to avoid repetition.

%% file: discussion.tex
Here, we have developed an Empirical Bayes methodology for prediction in large dimensional Gaussian models. 
Our proposed method involves the calibration of the tuning parameters of skrinkage estimators by minimizing risk estimates that are adapted to the shape of the loss function. 
It produces asymptotically optimal prediction.  Our risk estimation method and its proof techniques can also be used to construct optimal empirical Bayes predictive rules for general piecewise linear and related asymmetric loss functions, where we do not have any natural unbiased risk estimate.  In this paper, we have worked in a high-dimensional Gaussian model. Though normality transformations exist for a wide range of high-dimensional models \citep{brown2008a}, future works in extending the methodology  to non-Gaussian models, particularly discrete setups, would be interesting. 
Extending our Empirical Bayes approach from the one-period predictive setup to a multi-period setup  would be another interesting future direction.



%% file: appendix.tex
\section{Proof details for estimators in the class $\mathcal{S}^0$ and the lemmas used in Section~2}\label{append-origin}
\smallskip
We begin this section by first discussing about the discretization step conducted in the ARE estimation. We define the oracle estimator over any discretized set $\nett_n$ as
$$
\tau_{n}^{{\sf OR}}[\nett] =\argmin_{\tau \in \net} L_n(\thetab,\qhatb(\tau))~.
$$
Next, we show that under the assumptions A1-A3, this discretization step does not change the precision of our estimator by constructing a discrete set $\net$ such that the resultant loss from the oracle estimator over  that discrete set is very close to that of the oracle estimator $\tau_{n}^{{\sf OR}}$ computed over the entire domain of $\tau$.
The rationale for such construction of $\net$ is described in the proof of the following lemma. 
\begin{lem}\label{lem:append.1.temp}
	 There exists a discrete set $\Lambda_{n}$ such that for any $\epsilon > 0$,
	  \begin{align*}
	 \texttt{I.} \quad & P\big\{ L_n(\thetab,\qhatb(\tau^{\sf OR}_n[\Lambda_{n}])) - L_n(\thetab,\qhatb(\tau^{\sf OR}_n)) > \epsilon \big\} \to 0 \text{ as } n \to \infty	 \;\; \text{ and },\\
	 \texttt{II.} \quad  & \ex |L_n(\thetab,\qhatb(\tau^{\sf OR}_n[\Lambda_{n}])) - L_n(\thetab,\qhatb(\tau^{\sf OR}_n))| \to 0 \text{ as } n \to \infty.
	  \end{align*}
\end{lem}

\textbf{Construction of the Discrete Set $\nettt$ which attains the any prefixed $\epsilon$-precision:} First note that, by Assumption A1-A3  there exists dimension independent constants $C_1$, $C_2$, $C_3$ and $C_4$ such that for all large $n$ we have:
\begin{align}\label{eq:constants.temp}
&\sup_{1 \leq i \leq n} (b_i+h_i) \leq C_1, \\
&\max\{\sup_{1 \leq i \leq n} \sigma_{p,i}, \;  \sup_{1 \leq i \leq n} \sigma_{p,i}^{-1}\} \leq C_2, \\
&\bigg(\sup_{1 \leq i \leq n} \sigma_{p,i}/\sigma_{f,i} \bigg) \times  \bigg(\sup_{1 \leq i \leq n} |\Phi^{-1}(\bt_i)| \bigg) \leq C_3  \text{ and, } \\
&\sup_{\thetab \in \Theta_n}\frac{1}{n} \sum_{1=1}^n |\theta_i| \leq C_4 (\log \log n)^{1/2} 
\end{align}
where $\Theta_n \subseteq \Real^n$  contains all $\thetab$ obeying assumption A3. Define,
\begin{align}\label{eq:delta}
\delta_n =  \big\{2 \, C_1 \, C_2  (2 \phi(0) +  C_3 + \sqrt{a_n} C_4 + a_n ) \big \}^{-1}  \text{ where } a_n = \log \log n. 
\end{align}
 We consider an invertible transformation on $\tau$ and re-parametrize it by $\taut=\tau/(\tau+1)$. As $\tau$ varies over $[0,\infty]$, $\taut$ is contained in $[0,1]$.  We  construct an equi-spaced discretized set $\{0=\taut_1 \leq \taut_2 \leq \cdots \leq \taut_m \leq 1\}$ with this transformed variable where $\taut_i = (i+1) \delta_n $ and $m=\lceil 1/\delta_n \rceil$. We will invert it back to get the set $\nettt$ on $\tau \in [0,\infty]$.  The role of the re-parametrization is to restrict the domain of the variable of interest (which in this case is $\taut$) to a bounded set. 
 \par
 The cardinality of the set $\nettt$ is: $|\nettt|= a_n(1+o(1))$ \text{ as } $n \to \infty$. We conduct our computations for the ARE estimate based on the grid with $\nett_{n}$.
 
\noindent\textbf{Proof of Lemma~\ref{lem:append.1.temp}.} Using the above constructed grid, we prove the lemma.
Note that, the only $\tau$ dependent quantity in the expression of $\qhat_i(\tau)$ in \eqref{eq:linear.est.o} is $\alpha_i$ which can  be expressed as $\taut/(\taut+(1-\taut)\sigma_{p,i})$. The univariate predictive loss is:
$$ l_i(\theta_i, \qhat_i(\taut)) = (b_i+h_i)\, {\sigma_{f,i}}^{1/2}\, G\bigg(\frac{\alpha_i Z_i+ (\sigma_{f,i}+\alpha_i \sigma_{p,i})^{1/2} \Phi^{-1}(\bt_i) -\bar{\alpha}_i 
\theta_i}{{\sigma_{f,i}}^{1/2}}; \,\bt_i\bigg)$$
where  $Z_i$ and has $N(0,\sigma_{p,i})$ distribution. Note, $l_i(\theta_i,\qhat_i(\taut))$ is a.e. differentiable in $\taut$  and
$$ \bigg \vert \frac{\partial}{\partial \taut}\, l_i(\theta_i, \qhat_i(\taut)) \bigg \vert \leq (b_i+h_i)\, \cdot \bigg|C(\theta_i,Z_i,\taut)\bigg| \times \sup_{\omega \in \Real} \bigg \vert \frac{\partial}{\partial \omega}  G (\omega,\bt_i) \bigg \vert $$
where $$C(\theta_i,Z_i,\taut)=
\frac{\partial}{\partial \alpha_i} \big\{\alpha_i Z_i+ (\sigma_{f,i}+\alpha_i \sigma_{p,i})^{1/2} \Phi^{-1}(\bt_i) -\bar{\alpha}_i \theta_i\big \} \cdot \frac{\partial \alpha_i}{\partial \taut} .$$
Noting that  
${\partial \alpha_i}/ {\partial \taut} =  \sigma_{p,i}/(\taut+(1-\taut)  \sigma_{p,i})^2 $ which is bounded in magnitude by $\max \{\sigma_{p,i}, \sigma_{p,i}^{-1}\}$ we arrive at 
 $$ |C(\theta_i,Z_i,\taut)| \leq \{
  |Z_i|+ \sigma_{p,i}/\sigma_{f,i} \cdot |\Phi^{-1}(\bt_i)| +|\theta_i|\} \cdot \max \{\sigma_{p,i}, \sigma_{p,i}^{-1}\} .$$
Again,
$$ \bigg \vert \frac{\partial}{\partial \omega}  G (\omega,\bt_i) \bigg| = |\Phi(\omega)-\bt_i| \leq 2,$$  
and so, the derivative of the predictive loss is bounded above  by:
\begin{align}\label{eq:temp.I}
 \bigg \vert \frac{\partial}{\partial \taut}\, l_i(\theta_i, \qhat_i(\taut)) \bigg \vert \leq 2 (b_i+h_i)\, \max \{\sigma_{p,i}, \sigma_{p,i}^{-1}\}   \vert \cdot (|Z_i| + |\theta_i|+\sigma_{p,i}/\sigma_{f,i} \cdot |\Phi^{-1}(\bt_i)|).
 \end{align}
As $L_n(\thetab,\qhatb(\taut))=n^{-1} \sum_{i=1}^n  l_i(\theta_i, \qhat_i(\taut))$, we have:
$$\vert L_n(\thetab,\qhatb(\taut))-L_n(\thetab,\qhatb(\taut_j)) \vert \leq D_n |\taut-\taut_j| \text{ where } D_n = \sup_{\taut \in [0,1]} n^{-1} \sum_{i=1}^n \bigg \vert \frac{\partial}{\partial \taut}\, l_i(\theta_i, \qhat_i(\taut)) \bigg \vert.$$ 
Thus, we have
$\inf_{\taut_j \in \nettt} \vert L_n(\thetab,\qhatb(\taut))-L_n(\thetab,\qhatb(\taut_j)) \vert \leq D_n \delta_n$ which implies:
$$\vert L_n(\thetab,\qhatb(\tau^{\sf OR}_n[\Lambda_{n,\epsilon}])) - L_n(\thetab,\qhatb(\tau^{\sf OR}_n)) \vert \leq D_n \delta_n.$$
Again from \eqref{eq:temp.I}, it follows that
$$D_n \leq 2 \, C_1 \, C_2 \bigg (\frac{1}{n} \sum_{i=1}^n |Z_i| + \frac{1}{n} \sum_{1=1}^n |\theta_i| + C_3\bigg)$$ 
where $C_1$, $C_2$ and $C_3$ are defined in \eqref{eq:constants.temp}. 
\par
Now note that ${n}^{-1} \sum_{i=1}^n |Z_i|\sim N(2\phi(0),n^{-1})$. Thus, by definition \eqref{eq:delta} we have:\\
$P(D_n \delta_n > \epsilon) \to 0$ as $n \to \infty$, and  $E(D_n) \delta_n \to 0 $  as $n \to \infty$. Thus, the result follows.

\subsection{Proof of Theorem~\ref{origin.are.oracle.loss}} 
We know that
\begin{align*}
& \Pr \left\{ L_n(\thetab,\qhatb(\tauh^{\A}_n))\geq L_n(\thetab,\qhatb(\tau^{\sf OR}_n))+ \epsilon \right\} \qquad \qquad \qquad\\
& \qquad \leq \Pr \left\{ L_n(\thetab,\qhatb(\tauh^{\A}_n))\geq L_n(\thetab,\qhatb(\tau^{\sf OR}_n[\nett]))+ \epsilon/2 \right\} \\
&\qquad + \Pr \left\{ L_n(\thetab,\qhatb(\tau^{\sf OR}_n[\nett]))\geq L_n(\thetab,\qhatb(\tau^{\sf OR}_n))+ \epsilon/2 \right\}
\end{align*}
and the second term converges to $0$ as $n \to \infty$ by Lemma~\ref{lem:append.1.temp}. Next, we concentrate on the first term. 
By construction $\Ahat_n(\tauh^{\A}_n) \leq \Ahat_n(\tau^{\sf OR}_n)$. So, for any fixed $\epsilon > 0$ we have:
\begin{align*}
&\Pr \left\{ L_n(\thetab,\qhatb(\tauh^{\A}_n))\geq L_n(\thetab,\qhatb(\tau^{\sf OR}_n[\nett]))+ \epsilon/2 \right\} 
\leq \Pr \left\{ A_n(\thetab,\tauh^{\A}_n) \geq B_n(\thetab,\tau^{\sf OR}_n[\nett]) + \epsilon/2 \right\}
\\
& \text{ where } A_n(\thetab,\tauh^{\A}_n) = L_n(\thetab,\qhatb(\tauh^{\A}_n)) - \Ahat_n(\tauh^{\A}_n) \\ & \text{ and }
B_n(\thetab,\tau^{\sf OR}_n[\nett])= L_n(\thetab,\qhatb(\tau^{\sf OR}_n[\nett])) - \Ahat_n(\tau^{\sf OR}_n[\nett]).
\end{align*}
Now, using Markov inequality we get:
$$
\Pr \left\{ A_n(\thetab,\tauh^{\A}_n) \geq B_n(\thetab,\tau^{\sf OR}_n [\nett]) + \epsilon/2 \right\} \leq  2\epsilon^{-1} \ex|A_n(\thetab,\tauh^{\A}_n) - B_n(\thetab,\tau^{\sf OR}_n[\nett])|$$
which again is less than $ 4 \epsilon^{-1}  \ex \{\sup_{\tau \in \net} | \Ahat_n(\tau)-L_n(\thetab,\qhatb(\tau))| \}$. 
It is upper bounded by 
\begin{align*}
& 4 \epsilon^{-1}  \ex \{\sup_{\tau \in \net} | \Ahat_n(\tau)-R_n(\thetab,\qhatb(\tau))| \} +  4 \epsilon^{-1}  \ex \{\sup_{\tau \in [0,\infty]} | L_n(\thetab,\qhatb(\tau)) - R_n(\thetab,\qhatb(\tau)| \}
\end{align*}
The second term convereges to $0$ by Theorem~\ref{origin.risk.loss} and the first term is bounded above by
\begin{align*}
& 4 \epsilon^{-1}  \ex \bigg\{\sum_{\tau \in \net} | \Ahat_n(\tau)-R_n(\thetab,\qhatb(\tau))| \bigg\} 
\leq 4 \epsilon^{-1} |\net|\sup_{\tau \in [0,\infty]} \ex \{ | \Ahat_n(\tau)-R_n(\thetab,\qhatb(\tau))| \} 
 \end{align*}
 where $|\net|$ is the cardinality of $\net$. By construction of $\net$,  $|\net|=O(a_n)$ and by Theorem~\ref{origin.loss.are} the above expression converge to 0 as $n \to \infty$. Thus, we have the required result.

\subsection{Proof of Theorem~\ref{origin.are.oracle.risk}}
We upper bound $L_n(\thetab,\qhatb(\tauh^{\A}_n))- L_n(\thetab,\qhatb(\tau^{\sf OR}_n))$ by:
\begin{align*}
\{L_n(\thetab,\qhatb(\tauh^{\A}_n))- L_n(\thetab,\qhatb(\tau^{\sf OR}_n[\nett]))\}+\{L_n(\thetab,\qhatb(\tauh^{\A}_n[\nett]))- L_n(\thetab,\qhatb(\tau^{\sf OR}_n))\}.
\end{align*}
The expectation of the second term converges to $0$ by Lemma~\ref{lem:append.1.temp}. For the first term, 
we decompose the difference of the losses into the following $3$ parts:
\begin{align*}
&L_n(\thetab,\qhatb(\tauh^{\A}_n))- L_n(\thetab,\qhatb(\tau^{\sf OR}_n[\nett]))\\
&=\{L_n(\thetab,\qhatb(\tauh^{\A}_n)) - \Ahat_n(\tauh^{\A}_n) \} - \{ L_n(\thetab,\qhatb(\tau^{\sf OR}_n[\nett])\} - \Ahat_n(\tau^{\sf OR}_n[\nett])) \\
& \;\;\;+ \{\Ahat_n(\tauh^{\A}_n) - \Ahat_n(\tau^{\sf OR}_n[\nett])\}.
\end{align*}
Now, by construction the third term is less than $0$ and so, 
\begin{align*}
&\ex \left[ L_n(\thetab,\qhatb(\tauh^{\A}_n))- L_n(\thetab,\qhatb(\tau^{\sf OR}_n[\nett])) \right] \leq 2  \ex \bigg\{ \sup_{\tau \in \net} \big| \Ahat_n(\tau)-L_n(\thetab,\qhatb(\tau)) \big | \bigg\}\\
&\leq 2  \ex \bigg\{ \sum_{\tau \in \net } \big| \Ahat_n(\tau)-R_n(\thetab,\qhatb(\tau)) \big | \bigg\} + 2  \ex \bigg\{ \sup_{\tau \in \net} \big| L_n(\thetab,\qhatb(\tau)-R_n(\thetab,\qhatb(\tau)) \big |\bigg\}\\
&\leq 2 |\net|  \sup_{\tau \in [0,\infty]} \ex \bigg\{ \big| \Ahat_n(\tau)-L_n(\thetab,\qhatb(\tau)) \big | \bigg\} +  2  \ex \bigg\{ \sup_{\tau \in \net} \big| L_n(\thetab,\qhatb(\tau)-R_n(\thetab,\qhatb(\tau)) \big |\bigg\}
\end{align*}
which converges to 0 by Theorems~\ref{origin.loss.are} and \ref{origin.risk.loss}. Hence, the result follows.

\subsection{Proof of Corollary~\ref{cor:origin.are.oracle}}
The results follows directly from Theorems \ref{origin.are.oracle.loss} and~\ref{origin.are.oracle.risk}  as $\tau^{\sf OR}_n$ minimizes the loss $L_n(\thetab,\qhatb(\tau))$ among the class of all linear estimates  with shrinkage towards the origin.

\subsection{Proof of Theorem~\ref{origin.risk.loss}}
We need to prove:
$$ \sup_{\tau \in [0,\infty]}    \big |  R_n(\thetab, \qhatb(\tau)) -  L_n(\thetab, \qhatb(\tau)) \big| \to 0  \text{ in } L_1 \text{ as } n \to \infty.$$
In our proof we will use a version of the uniform SLLN \cite[Lemma 2.4]{newey1994}. 
Based on form the loss functions in \eqref{eq:loss-function} and the form of the linear estimators in \eqref{eq:linear.est.o}, we can reparametrize this problem with respect to $\taut=\tau/(\tau+1)$ instead of $\tau$. The only $\tau$ dependent quantity in the expression of $\qhat_i(\tau)$ in \eqref{eq:linear.est.o} is $\alpha_i$ which is reparametrized to $\taut/(\taut+(1-\taut)\sigma_{p,i})$. 
As $\taut \in[0,1]$, the supremum here is actually over compact set. Also, $l_i(\theta_i,\qhat_i(\taut)(x))$ is continuous at each $\taut$ for all most all $x$ and $\theta$.
Also, 
$$ l_i(\theta, \qhat_i(\taut)) = (b_i+h_i)\, {\sigma_{f,i}}^{1/2}\, G\bigg(\frac{\alpha_i z_i+ (\sigma_{f,i}+\alpha_i \sigma_{p,i})^{1/2} \Phi^{-1}(\bt_i) -\bar{\alpha}_i \theta_i}{{\sigma_{f,i}}^{1/2}}; \,\bt_i\bigg)$$
where and $z_i=x_i-\theta_i$ and has $N(0,\sigma_{p,i})$ distribution.
By  Lemma~\ref{lem:G.bound} we know $G(y,\bt)\leq \phi(0)+(1-\bt) |y|$ and we use  $\alpha_i \in[0,1]$ to arrive at: for each $\theta_i$ and for all $\taut \in [0,1]$ we have,
\begin{align}\label{eqn.loss.bound}
 l_i(\theta_i,\qhat_i(\tau)) \leq (b_i+h_i) \big[ \sigma_{f,i}^{1/2} \phi(0) + (1-\bt_i) \big \{|z_i|+ (\sigma_{f,i}+\sigma_{p,i})^{1/2} \Phi^{-1}(\bt_i)+|\theta_i| \big\} \big].
 \end{align}
So, for any $\thetab$ and $\tau \in [0,\infty]$  we have:
\begin{align}\label{tempor}
L_n(\thetab,\qhatb(\tau))\leq A_n \bigg(\phi(0)+ n^{-1} \sum_{i=1}^n |\Phi^{-1}(\bt_i)|\bigg) + B_n \bigg (n^{-1} \sum_{i=1}^n| z_i| + n^{-1} \sum_{i=1}^n |\theta_i|\bigg)
\end{align}
where $A_n=\sup\{(b_i+h_i) \sigma_{f,i}^{1/2}:i=1,\ldots,n\}$ and $B_n=\sup\{b_i+h_i:i=1,\ldots,n\}$. By Assumptions A1, A3 we have $\limsup_n A_n \leq \infty$ and $\limsup_n B_n < \infty$. So, the expectation of the RHS in \eqref{tempor} is finite under Assumption A2. As all the conditions of  \citet[Lemma 2.4]{newey1994} hold,  we can apply  the SLLN uniformly. So, the loss converge to the risk and we have:
$$\sup_{\tau \in [0,\infty]}    \big |  R_n(\thetab, \qhatb(\tau)) -  L_n(\thetab, \qhatb(\tau)) \big| \to 0  \text{ in } P \text{ as } n \to \infty.$$
Now,  noting  that the upper bound in \eqref{tempor} is uniformly integrable as  $ \sum_{i=1}^n| z_i|$ is U.I. by the extra integrability condition (See Lemma~\ref{lem:ui}). So, $\sup_{\tau \in [0,\infty]}|  R_n(\thetab, \qhatb(\tau)) -  L_n(\thetab, \qhatb(\tau)) \big|$ is U.I. and we also have $L_1$ convergence. Hence, the result follows.

\par

Next, we provide proofs of all the main lemmas used in Section~\ref{sec-2-main-proof}.

\subsection{Proof of Lemma \ref{loss.properties}}

As both the L.H.S. and R.H.S. scale in $(b+h)$ without loss of generality we assume $b+h=1$.
We will first prove the result for $\sigma=1$ before proceeding to the general case.
	Noting that  for all $y$ and $q$, $(y-q)^+ = y-q + (q-y)^+$, we have:
	\bex
	b\, \E [ Y - q ]^+  + h\, \E [ q - Y ]^+
	= b (\theta - q) + (b+h)  \E [ q - \theta - Z ]^+
	\eex
	where $Z$ is standard normal random variable. Direct calculation yields:
	\begin{align*}
	\E [ q - \theta - Z ]^+ &= \int_{-\infty}^{q-\theta} (q-\theta - x) \phi(x) dx \\
	&= (q-\theta) \Phi(q-\theta) + \int_{-\infty}^{q-\theta} -x \phi(x) dx 
	= (q-\theta) \Phi(q-\theta) + \phi(q-\theta)~,
	\end{align*}
	which gives the desired result.  Also, note that in this case 
	$\partial_w G(w,b)= \Phi(w)-b$. So, $G(w,b)$ is minimized at $\Phi^{-1}(b)$ and the minimum value is $\phi(\Phi^{-1}(b))$.
	
	For general $\sigma$ we rewrite the L.H.S. using $Y\stackrel{d}{=}\sigma^{1/2}Z+\theta$ where $Z$ is a standard normal random variable to obtain: 
	$$\sigma^{1/2} \{	b\, \E [ Z -  \sigma^{1/2} (q-\theta) ]^+  + h\, \E [  \sigma^{1/2} (q-\theta) - Z ]^+\}.$$  Now, the result stated in the lemma follows by using the already proven result for  
	the unit variance case.

\subsection{Proof of Lemma \ref{linear.estimates}}

	With out loss of generality we can assume that $b_i+h_i=1$ as the univariate loss is just scaled by that factor. Now, the minimizer of the the Bayes risk $B_1(\eta,\tau)$ is given by
$$\qhat(\eta,\tau)(x)=\argmin_{\qhat} \int l(\theta,\qhat(x)) \pi(\theta|x) \, d\theta.$$
The posterior distribution $\pi(\theta|x) \sim N(\alpha x + \alphab \eta, \alpha \sigma_p)$. So, for any fixed $x$ we have
\begin{align*}
\int l(\theta,\qhat(x)) \pi(\theta|x) \, d\theta={\sigma_f^{1/2}}\; \ex \bigg\{G\bigg(\frac{\qhat-T}{\sqrt{\sigma_f}}, b\bigg)\bigg\}
\end{align*}
where the expectation is over $T$ which follows $N(\alpha x + \alphab \eta, \alpha \sigma_p)$. The above expectation equals
$\sigma_f^{1/2} \,\ex\, G(\sigma_f^{-1/2}\{\qhat(x)-(\alpha x + + \alphab \eta + \alpha^{1/2} \sigma_p^{1/2} Z)\},b)$
where $Z$ is a standard normal random variable.
To evaluate the aforementioned expression  we now use the identity in \eqref{eq:identity} with $Y\sim N(a(x), \sigma_f)$ and $a(x)=\alpha x + \alphab \eta - \qhat(x)$. Finally we get $\int l(\theta,\qhat(x)) \pi(\theta|x) \, d\theta$ equals 
\begin{align*}
&\ex_{Z\sim N(0,1)} \big \{  \ex_{Y\sim N(a(x),\sigma_f)} \big( b\,(Y+ \alpha^{1/2} \sigma_p^{1/2} Z)^+ + h \, (-Y- \alpha^{1/2} \sigma_p^{1/2} Z)^+ \big)\big\}\\
&=\ex \big\{b\, \big(a(x)+(\sigma_f+\alpha\sigma_p)^{1/2} Z\big)^+ + h\big(-a(x)-(\sigma_f+\alpha \sigma_p)^{1/2} Z\big)^+ \big\}. \text{\hspace{5cm}}
\end{align*}
As $Y+\alpha^{1/2} \sigma_p^{1/2} Z \sim N(-a(x),\sigma_f+\alpha \sigma_p)$ the above equality follows using  $Y+(\alpha \sigma_p)^{1/2} Z \stackrel{d}{=} a(x)+ (\sigma_f+\alpha \sigma_p)^{1/2} Z$. Again, using change of variable, the problem can be ultimately reduced to finding the minimizer for:
$$(\sigma_f+\alpha \sigma_p)^{1/2}\big \{ b\, \ex\{Z-\tilde{a}(x)\}^+ + h \, \ex\{\tilde{a}(x)-Z\}^+ \big\} $$
 where   $\tilde{a}(x)=-a(x)(\sigma_f+\alpha \sigma_p)^{-1/2}$.
By Lemma~\ref{loss.properties}, it is minimized when $\tilde{a}(x)=\Phi^{-1}(b)$ which implies $\qhat_{\alpha}(x)=\alpha x + \alphab \eta + (\sigma_f+\alpha \sigma_p)^{1/2} \Phi^{-1}(b)$. Also, the minimum value is $(\sigma_f+\alpha \sigma_p)^{1/2}\phi(\Phi^{-1}(b))$ which gives us the expression for the Bayes risk. 
\par 
The risk of the Bayes estimate $\qhat(\eta,\tau)$  is given by:
$$b \cdot \ex_{\theta} \big(Y-\alpha X - \alphab \eta - (\sigma_f+\alpha \sigma_p)^{1/2} \Phi^{-1}(b)\big)^+  + h \cdot \ex_{\theta} \big(\alpha X + \alphab \eta +  (\sigma_f+\alpha \sigma_p)^{1/2} \Phi^{-1}(b)-Y\big)^+$$
where $X \sim N(\theta,\sigma_p)$, $Y \sim N(\theta,\sigma_f)$ and given $\theta$, $X \perp Y$. And so, the above equals, 
$$ (\sigma_f+\alpha^2 \sigma_p)^{1/2} \bigg\{ b \cdot \ex_{0} \big( Z  -  J  \big)^+  + h  \ex_{0} \big( Z  -  J  \big)^- \bigg\}=  (\sigma_f+\alpha^2 \sigma_p)^{1/2} G(J,b)$$
where $J =(\sigma_f+\alpha^2 \sigma_p)^{-1/2}\{-\bar{\alpha} (\theta-\eta) + (\sigma_f+\alpha \sigma_p)^{1/2} \Phi^{-1}(b)\}$. This completes the proof.

\subsection{Proof of Lemma \ref{lem:bias.bound}}

By Taylor's Theorem,
$$|G_{\Kn}(y,b)- G(y,b)| =  \frac{ \phi(\zeta) \, |H_{\Kn-1}(\zeta)| \, |\zeta|^{\Kn+1} }{(\Kn+1)!}~,$$
where $\zeta$ lies between $0$ and $y$.
Noting that  $\phi(\zeta) \leq 1 $ for all $\zeta$. By Lemma \ref{lem:hermite-bound}, there exists an absolute constant $c$ such that:
$$ \abs{ H_{\Kn-1}(\zeta)} \leq c\,  { e^{ \zeta^2 / 4}  }{ (\Kn-1)!} { (\Kn-1)^{-1/3}  \, \{ (\Kn-1)/ e \}^{-(\Kn-1)/2} } $$
which provide us with the following error bound:
\begin{equation}\label{eq.tempo}
\begin{split}
|G_{\Kn}(y,b)- G(y,b)|
&\leq c\,  \frac{ e^{ y^2 / 4} }{(\Kn+1)!}  \times \frac{ (\Kn-1)!}{ (\Kn-1)^{1/3}  \, {\left( { (\Kn-1) /  e} \right)}^{(\Kn-1)/2} }\times
\abs{y}^{\Kn+1} \\
& =  c \left( \frac{  e \, y^2 }{\Kn-1} \right)^{(\Kn-1)/2} \frac{e^{ y^2 / 4} \, y^2}{(\Kn-1)^{1/3}(\Kn+1) \Kn}.
\end{split}
\end{equation}
By definition of $\Kn$ just before Equation~\eqref{def.uni},~$ \Kn-1 \geq  e^2 (\gamma +\sqrt{2\sigma})^2 (2 \log n).$
Since $|y| \leq (1 + \sqrt{2\sigma}/\gamma) \thr = (\gamma + \sqrt{2\sigma})\sqrt{2 \log n}$,
\begin{align*}
\left( \frac{  e \, y^2 }{\Kn-1} \right)^{(\Kn-1)/2}  &\leq e^{-(\Kn-1)/2} \leq e^{-  e^2 (\gamma + \sqrt{2\sigma})^2 \log n } = n^{- e^2 \,(\gamma + \sqrt{2\sigma})^2} \\
\frac{y^2}{(\Kn-1)^{1/3}(\Kn+1) \Kn} &\leq   \frac{y^2}{(K_n-1)^2}   \leq \frac{1}{e^4 (\gamma + \sqrt{2\sigma})^2 (2 \log n)}\\
e^{ y^2 / 4} &\leq e^{ (\gamma + \sqrt{2\sigma})^2 (\log n)/2} = n^{ (\gamma + \sqrt{2\sigma})^2 / 2} ~\leq~ n^{ (\gamma + \sqrt{2\sigma})^2}~,
\end{align*}
which implies that 
\begin{align*}
\sup_{\abs{y} \leq (1 + \sqrt{2\sigma}/\gamma) \thr} |G_{\Kn}(y,b)- G(y,b)|
&\leq  c \, \frac{ n^{-(e^2-1) (\gamma + \sqrt{2\sigma})^2} }{ e^4 (\gamma + \sqrt{2\sigma})^2 }~,
\end{align*}
which is the desired result.

The second part follows  because $G(y,b)=\phi(y)-y \tilde{\Phi}(y) + \bar{b} y$, and thus,
$$
 \abs{ G(y,b) - \bar{b} y} = \abs{\phi(y)-y \tilde{\Phi}(y)}  \leq \frac{\phi(y)}{y^2} \leq \frac{ e^{-y^2/2}}{y^2}~,
$$
where the first inequality follows from Lemma ~\ref{lem:mills.ratio}. 
For the proof of the third statement, note that for $y < 0$,
$$
G(y,b)=\phi(y)-y \tilde{\Phi}(y) + \bar{b} y = \phi(-y) - (-y) \tilde{\Phi}(-y) - b y 
$$
and we can then apply  Lemma~\ref{lem:mills.ratio} as before because $-y$ is now positive.


\subsection{Proof of Lemma \ref{lem:variance.bound}}
By Lemma~\ref{lem:Cai}, $\E \left[ S^2(\Ua) \right]$ is bounded above by
$$
\left( G(0,b) +  \abs{G'(0,b)} \sqrt{\E \Ua^2} +  \sum_{l=0}^{\Kn} \frac{ \abs{H_l(0) }}{(l+2)!}  ( 2 \sigma \da^2)^{(l+2)/2}   \sqrt{ \E \, H^2_{l+2}\bigg( \frac{\Ua}{\sqrt{2 \sigma \da^2}}  \bigg)} \;  \right)^2
$$
We will now bound each of the term in the above expression.
Note that $G(0,b)=\phi(0)$, $G'(0,b)=\frac{1}{2}-b$, and $\E \Ua^2 = 2 \sigma \da^2 + \tha^2$.  So, the first two terms are $o(\sqrt{n})$ as $n \to \infty$. Thus, it suffices to show that the last term is also $o(\sqrt{n})$.
Let $\ba = 2 e \sigma \da^2$.  Then, by Lemma~\ref{lem:Cai2}, we have that for all $l \geq 0$
\begin{align*}
 ( 2 \sigma \da^2)^{(l+2)/2}   \sqrt{ \E \, H^2_{l+2}\bigg( \frac{\Ua}{\sqrt{2 \sigma \da^2}}  \bigg)}
 & =  ( 2 \sigma \da^2)^{(l+2)/2}   (l+2)^{(l+2)/2}  \bigg(1+ \frac{\tha^2}{2 \sigma \da^2 (l+2)}\bigg)^{(l+2)/2} \\
 & =  \left( \frac{l+2}{e} \right)^{(l+2)/2} \bigg(\ba+ \frac{\tha^2 \ba}{2 \sigma \da^2 (l+2)}\bigg)^{(l+2)/2}\\
 & \leq \left( \frac{l+2}{e} \right)^{(l+2)/2} \bigg(1 + \frac{\tha^2 \ba }{2 \sigma \da^2 (l+2)}\bigg)^{(l+2)/2}\\
  & \leq \left( \frac{l+2}{e} \right)^{(l+2)/2} e^{ \tha^2 \ba / (4 \sigma \da^2)}~,
\end{align*}
where the first inequality follows from $\ba \leq 2 e \sigma < 1$  because $\abs{\da} \leq 1$ and Assumption A3 implies that $\sigma < 1/(4e)$.
The final inequality follows Lemma~\ref{lem:inequality}. 
Since  $\abs{\tha} \leq (1 + \sqrt{2 \sigma}/\gamma) \thr = (\gamma + \sqrt{2 \sigma}) \sqrt{ 2 \log n}$ and $\ba / (2 \sigma \da^2) = e$, 
$$
\frac{\tha^2 \ba }{4 \sigma \da^2} \leq e (\gamma + \sqrt{2 \sigma})^2  \log n  = n^{ e (\gamma + \sqrt{2 \sigma})^2 }
$$
Using the above abound and Lemma \ref{lem:hermite-bound}, it follows that there exists an absolute constant $c$ such that
\begin{align*}
& \sum_{l=0}^{\Kn} \frac{ \abs{H_l(0) }}{(l+2)!}  ( 2 \sigma \da^2)^{(l+2)/2}   \sqrt{ \E \, H^2_{l+2}\bigg( \frac{\Ua}{\sqrt{2 \sigma \da^2}}  \bigg)} \\
& \quad \leq n^{ e (\gamma + \sqrt{2 \sigma})^2 } \sum_{l=0}^{\Kn} \frac{ \abs{H_l(0) } \, (l +2)^{(l+2)/2}}{(l+2)! \, e^{(l+2)/2}} \\
& \quad\leq c\, n^{ e (\gamma + \sqrt{2 \sigma})^2 } \sum_{l=1}^{\Kn} \frac{  (l+2)^{(l+2)/2}}{(l+2)! \, e^{(l+2)/2}} \times \frac{  l!   }{ l^{1/3} \, {\left( {l \over  e} \right)}^{l /2} }  \\
& \quad \leq \, c \, n^{ e (\gamma + \sqrt{2 \sigma})^2 }\sum_{l=1}^{\Kn} \frac{ 1}{ l^{4/3} }~,
\end{align*}
where the last inequality follows from Lemma \ref{lem:inequality} because
$$
\frac{  (l+2)^{(l+2)/2}}{(l+2)! \, e^{(l+2)/2}} \times \frac{  l!   }{ l^{1/3} \, {\left( {l \over  e} \right)}^{l /2} } 
= \frac{1}{e} \left( \frac{l+2}{l} \right)^{l/2} \frac{1}{(l+1) l^{1/3}} \leq \frac{1}{(l+1)l^{1/3}} \leq \frac{1}{l^{4/3}}
$$
Note that $\sum_{l=1}^{\infty} \frac{ 1}{ l^{4/3} }  < \infty$.  Also,  by our definition, $0 < \gamma < (1/\sqrt{2e}) - \sqrt{ 2\sigma}$, which implies that $e (\gamma + \sqrt{2 \sigma})^2 < 1/2$. So, $a_n^8 n^{-1} \ex[S(\Ua)]^2 \leq O(a_n^8 n^{ e (\gamma + \sqrt{2 \sigma})^2 }) = o(\sqrt{n})$, which completes the proof.

\subsection{Proof of Lemma \ref{lem:large.deviation.bound}}
Since $\Va = \tha + \sqrt{ 2 \sigma \da^2} Z$ and $\da^2 \leq 1$,  for \textbf{Case 1} where $\abs{\tha} \leq \thr/2$, 
\begin{align*}
\Pr\{|\Va| > \thr\} & \leq  2 \Pr\left\{ {\thr / 2} + \sqrt{2 \sigma \da^2} Z > \thr \right\} 
~\leq~ 2 \Pr\left\{   Z > \thr/(2 \sqrt{2 \sigma}) \right\} \\
& = 2 \Phit \left(\gamma \sqrt{ \log n / (4 \sigma)} \right) \leq \frac{2 \phi \left(\gamma \sqrt{ \log n / (4 \sigma)}   \right)}{\gamma \sqrt{ \log n / (4 \sigma)} } 
\leq \frac{ 2 n^{-\gamma^2/(8 \sigma)} }{\gamma \sqrt{ \log n / (4 \sigma)}} ~,
\end{align*} 
where the next to last inequality follows from Lemma \ref{lem:mills.ratio}. Thus,
$$
\lim_{n \to \infty} \sup_{\alpha \,:\, |\tha| \;\leq\; \thr/2} a_n^8\, \thr^2 \cdot  \Pr \{ \abs{\Va} > \thr \} = 0~,
$$
which is the desired result.

\textbf{For Case 2}, we will assume that $\thr/2 < \tha \leq (1 \,+\, \sqrt{2\sigma}/\gamma)\thr$; the proof for the other case is the same by symmetry. Since $\Va = \tha + \sqrt{ 2 \sigma \da^2} Z$ and $\da^2 \leq 1$,
\begin{align*}
\Pr\{\Va < -\thr\} & \leq  2 \Pr\left\{ {\thr / 2} + \sqrt{2 \sigma \da^2} Z < -\thr \right\} 
~\leq~  \Pr\left\{   Z < -3 \thr/(2 \sqrt{2 \sigma}) \right\} \\
& \leq \frac{ \phi \left(\gamma \sqrt{ 9 \log n / (4 \sigma)}   \right)}{\gamma \sqrt{ 9 \log n / (4 \sigma)} } 
\leq \frac{ n^{-9\gamma^2/(8 \sigma)} }{\gamma \sqrt{ 9 \log n / (4 \sigma)}} ~,
\end{align*} 
and since $\tha \leq (1 \,+\, \sqrt{2\sigma}/\gamma)\thr = (\gamma \,+\, \sqrt{2\sigma})\sqrt{2 \log n}$, we have that
$$
\lim_{n \to \infty} \sup_{\alpha \,:\, \thr/2 \;<\; \tha \;\leq\; (1 \,+\, \sqrt{2\sigma}/\gamma)\thr } a_n^8\, \abs{\tha} \cdot \Pr \{ \Va <  -\thr \} = 0~,
$$
which is the desires result.

\textbf{For Case 3}, suppose that $\tha > \left(1 \,+\, {\sqrt{2\sigma} / \gamma} \right) \thr$; the proof for the other case is the same.  Since $\Va = \tha + \sqrt{ 2 \sigma \da^2} Z$ and $\da^2 \leq 1$,
\begin{align*}
\Pr \left\{ |\Va| \leq \thr \right\}
& \leq \Pr \left\{ \Va \leq \thr \right\}
 \leq  \Pr \left\{ \left(1 \,+\, {\sqrt{2\sigma} / \gamma} \right) \thr + \sqrt{ 2 \sigma \da^2} Z  \leq \thr \right\} \\
& \leq \Pr \left\{  Z  \leq - ( \sqrt{2 \sigma}/\gamma)  \thr / ( \sqrt{ 2\sigma}) \right\} 
 \leq \frac{ \phi \left(   \thr / \gamma  \right)}{ \thr / \gamma } 
= \frac{ n^{-1} }{\sqrt{ 2 \log n}} ~,
\end{align*}
which implies that
$\lim_{n \to \infty} \sup_{\alpha \,:\, |\tha| \;>\; \left(1 \,+\, {\sqrt{2\sigma} / \gamma}\right) \thr} n a_n^8\,\cdot \Pr \left\{ |\Va| \leq \thr \right\} ~=~ 0$.
Also, 
\begin{align*}
\Pr \left\{ |\Va| \leq \thr \right\} & \leq \Pr \left\{ \Va \leq \thr \right\}
 \leq  \Pr \left\{ \tha + \sqrt{ 2 \sigma \da^2} Z  \leq \thr \right\} 
\leq \Pr \left\{  Z  \leq - ( \tha - \thr) / ( \sqrt{ 2\sigma}) \right\} \\
& \leq \frac{ \phi \left(   ( \tha - \thr)/ ( \sqrt{ 2\sigma})  \right)}{ ( \tha - \thr) / ( \sqrt{ 2\sigma}) } 
= \frac{  e^{ -\tha^2 \left( 1-\frac{\thr}{\tha} \right)^2 / (4 \sigma)  }}{ \tha \left( 1-\frac{\thr}{\tha}  \right) / (\sqrt{2 \sigma}) } 
\leq  \frac{ e^{ -\tha^2  / ( 2 ( \gamma + \sqrt{2 \sigma})^2)} }{\tha/( \gamma + \sqrt{2 \sigma})} ~,
\end{align*}
where  the last inequality follows from the fact that  $\tha > \left(1 \,+\, {\sqrt{2\sigma} / \gamma} \right) \thr$, which implies that
$ 1 > 1-(\thr/\tha) > \sqrt{2 \sigma} / ( \gamma + \sqrt{ 2 \sigma})$.  Note that for any $a > 0$, $\max_{x \geq 0} x e^{-a x} = 1/(e a)$, which implies that
$$
\tha^2 \cdot \Pr \left\{ |\Va| \leq \thr \right\} 
\leq \frac{\tha^2  e^{ -\tha^2  / ( 2 ( \gamma + \sqrt{2 \sigma})^2)} }{\tha/( \gamma + \sqrt{2 \sigma})}
\leq  \frac{ 2 (\gamma + \sqrt{ 2 \sigma} )^2 / e}{\tha/( \gamma + \sqrt{2 \sigma})} = \frac{ 2 (\gamma + \sqrt{ 2 \sigma} )^3 / e}{\left(1 \,+\, {\sqrt{2\sigma} / \gamma} \right) \thr}~,
$$
which implies that $\lim_{n \to \infty} \sup_{\alpha \,:\, |\tha| \;>\; \left(1 \,+\, {\sqrt{2\sigma} / \gamma}\right) \thr} a_n^8 \tha^2 \cdot \Pr \left\{ |\Va| \leq \thr \right\} ~=~ 0$, which is the desired result.
To complete the proof for Case 3, we will show that 
$$\lim_{n \to \infty} \sup_{\alpha \,:\, \tha \;>\; (1 \,+\, \sqrt{2\sigma}/\gamma)\thr }  a_n^8\,\tha^2~\cdot~\Pr \{ \Va <  -\thr \} = 0~.$$
This follows immediately from the above analysis as
$\Pr \{ \Va <  -\thr \} \leq \Pr \{ \Va <  \thr \}$, and we have just shown that 
$\lim_{n \to \infty} \sup_{\alpha \,:\, \tha \;>\; (1 \,+\, \sqrt{2\sigma}/\gamma)\thr } a_n^8 \tha^2 \cdot \Pr \left\{ \Va \leq \thr \right\} = 0$.

\section{Proof details for estimators in the class $\mathcal{S}$ and the lemmas used in Section~4}\label{append-data-driven}
\smallskip
\subsection{Proof of Lemma~\ref{data.risk.are}} Using the relation between $\Ahat$ and $\AD$ it follows that
$$\ex \big[  \big(\AD_n(\eta,\tau)-R_n(\thetab,\qhatb(\eta,\tau))\big)^2\big]= \ex \big[  \big(\Ahat_n(\tau, \X-\eta)-R_n(\thetab-\eta,\qhatb(\tau))\big)^2\big]~.$$
Now, similarly as in the proof of Theorem~\ref{origin.loss.are}, following the Bias-Variance decomposition and the argument below \eqref{eq.bias.var} we upper bound the RHS above by
\begin{align*}
& A_n \bigg\{ \bigg(\frac{1}{n} \sum_{i=1}^n\bias_{\theta_i}(T_i(X_i-\eta,Z_i,\tau))\bigg)^2 + \frac{1}{n^2} \sum_{i=1}^n \var_{\theta_i}(T_i(X_i-\eta,Z_i,\tau))\bigg\}\\
& = A_n  \bigg\{ \bigg(\frac{1}{n} \sum_{i=1}^n\bias_{\theta_i - \eta}(T_i(X_i,Z_i,\tau))\bigg)^2 + \frac{1}{n^2} \sum_{i=1}^n \var_{\theta_i-\eta}(T_i(X_i,Z_i,\tau))\bigg\}
\end{align*}
where $\Z=\{Z_1,\ldots,Z_n\}$ follows $N(0,I_n)$; $T_i(X_i-\eta,Z_i,\tau)$ are randomized rules defined in Section~\ref{ARE.intro}; $A_n=\sup\{(b_i+h_i)^2(\sigma_{f,i}+ \alpha_i \sigma_{p,i}): i=1,\ldots,n\}$ which by Assumptions A1, A3 satisfies $\sup_n A_n \leq \infty$. Now, using Lemma~\ref{univariate.origin} we get the required result:
$\sup_{\tau \in[0,\infty], \eta \in \Real} a_n^8 \, \ex  \big(\AD_n(\eta,\tau)-R_n(\thetab,\qhatb(\eta,\tau))\big)^2 \to 0$ as $n \to \infty$.  

\subsection{Proof of Lemma~\ref{lem:sec4.temp}.} We need the following lemma to prove  Lemma~\ref{lem:sec4.temp}.

\begin{lem}\label{lem:sec4.temp.1}
	For any fixed $0 < \alpha_1 < \alpha_2 < 1$, $a_n=\log \log n$, the event $M_n(\X)=$ $\{[\hat{m}_n(\alpha_1), \hat{m}_n(\alpha_2)] \subseteq [-a_n,a_n]\}$ satisfies:
	\begin{align*}
	\texttt{I.} \quad & P\big\{ M_n \big\} \to 1 \text{ as } n \to \infty	 \;\; \text{ and },\\
	\texttt{II.} \quad  & \ex \big\{|L_n(\thetab,\qhatb(\eta_{n}^{\nett},\tau_{n}^{{\nett}})) - L_n(\thetab,\qhatb(\eta^{\sf DOR}_n, \tau^{\sf DOR}_n))| \cdot I\{M^c_n(\X)\} \big\} \to 0 \text{ as } n \to \infty.
	\end{align*}
\end{lem}

\noindent\textbf{Proof.} 
For any $\alpha$ th quantile of $\X$  with  $\epsilon_i$ i.i.d. from standard normal, note that
\begin{align*}
& |\hat{m}_n(\alpha)| =|\text{quantile}\{\theta_i + \sqrt{\sigma_{p,i}}\, \epsilon_i: 1 \leq i \leq n; \; \alpha \}|\\
&\leq \text{quantile}\{|\theta_i| + \sqrt{\sigma_{p,i}}\, |\epsilon_i|: 1 \leq i \leq n; \alpha \} \vee \text{quantile}\{|\theta_i| + \sqrt{\sigma_{p,i}}\, |\epsilon_i|: 1 \leq i \leq n; 1-\alpha \}\\
&\leq \text{quantile}\{|\theta_i| + c |\epsilon_i|: 1 \leq i \leq n; \alpha \} \vee \text{quantile}\{|\theta_i| + c\, |\epsilon_i|: 1 \leq i \leq n; 1-\alpha \}
\end{align*}
where $c=\{\sup_i \sigma_{p,i}\}^{1/2}$. Also note that for any $\beta$ th quantile
\begin{align*}
&\text{quantile}\{|\theta_i| + c\, |\epsilon_i|: 1 \leq i \leq n; \; \beta \} \\
&\leq \text{quantile}\{|\theta_i|: 1 \leq i \leq n; (1+\beta)/2 \} + c \cdot \text{quantile}\{ |\epsilon_i|: 1 \leq i \leq n; (1+\beta)/2 \}. 
\end{align*}
For any $\beta \in (0,1)$, by assumption $A_3$ the $\text{quantile}\{|\theta_i|: 1 \leq i \leq n; (1+\beta)/2 \}$ is asymptotically bounded and $n^{1/2}(\text{quantile}\{ |\epsilon_i|: 1 \leq i \leq n; (1+\beta)/2 \} - \Phi^{-1}((3+\beta)/4))$ has an asymptotic normal distribution as $n \to \infty$. Therefore, for any $\alpha \in (0,1)$,  $a_n P(|\hat{m}_n(\alpha)|> a_n) \to 0$ as $n \to \infty$.
The first result of the lemma follows directly by using it. 
\par 
As on $M^c_n(\X)$, $L_n(\thetab,\qhatb(\eta_{n}^{\nett},\tau_{n}^{{\nett}})\geq L_n(\thetab,\qhatb(\eta^{\sf DOR}_n, \tau^{\sf DOR}_n))\geq 0$, for the second result of the lemma it is enough to show that:
\begin{align*}
\ex \big\{L_n(\thetab,\qhatb(\eta_{n}^{\nett},\tau_{n}^{{\nett}})) \cdot I\{M^c_n(\X)\} \big\}
\leq \{\ex L^2_n(\thetab,\qhatb(\eta_{n}^{\nett},\tau_{n}^{{\nett}})) \cdot P(M^c_n(\X))\}^{1/2}
\end{align*}
where the above inequality uses the Cauchy-Schwarz inequality. By equation \eqref{tempor}, it follows  $\ex\, L^2_n(\thetab,\qhatb(\eta_{n}^{\nett},\tau_{n}^{{\nett}}))=o(a_n)$ and by the above calculations we have 
$a_n P(M_n^c(\X)) \to 0$ \text{ as } $ n \to \infty$.  Thus, second result of the lemma is also proved.

\noindent\textbf{Proof of Lemma~\ref{lem:sec4.temp}.} The univariate predictive loss in this case is:
$$ l_i(\theta_i, \qhat_i(\eta,\taut)) = (b_i+h_i)\, {\sigma_{f,i}}^{1/2}\, G\bigg(\frac{\alpha_i Z_i+ (\sigma_{f,i}+\alpha_i \sigma_{p,i})^{1/2} \Phi^{-1}(\bt_i) -\bar{\alpha}_i 
	(\theta_i-\eta)}{{\sigma_{f,i}}^{1/2}}; \,\bt_i\bigg)$$
where $\alpha_i=\taut/(\taut+(1-\taut)\sigma_{p,i})$ and $Z_i$ is $N(0,\sigma_{p,i})$ distributed. $l_i(\theta_i,\qhat_i(\taut))$ is a.e. differentiable in $\taut$ and $\eta$,  and by calculations similar to the proof of Lemma~\ref{lem:append.1.temp} we have:
\begin{align*}
&\bigg \vert \frac{\partial}{\partial \taut}\, l_i(\theta_i, \qhat_i(\eta,\taut)) \bigg \vert \leq 2 (b_i+h_i)\, \max \{\sigma_{p,i}, \sigma_{p,i}^{-1}\}   \vert \cdot (|Z_i| + |\theta_i|+ |\eta| + \sigma_{p,i}/\sigma_{f,i} \cdot |\Phi^{-1}(\bt_i)|)\\
&\bigg \vert \frac{\partial}{\partial \eta}\, l_i(\theta_i, \qhat_i(\eta,\taut)) \bigg \vert \leq 2 (b_i+h_i)\,
\end{align*}
As $L_n(\thetab,\qhatb(\eta,\taut))=n^{-1} \sum_{i=1}^n  l_i(\theta_i, \qhat_i(\eta,\taut))$, we have:
\begin{align*}
& \vert L_n(\thetab,\qhatb(\eta,\taut))-L_n(\thetab,\qhatb(\eta_k,\taut_l)) \vert \leq D_{n,1} |\eta-\eta_k| + D_{n,2} |\taut-\taut_l|\text{ where, } \\
& D_{n,1} = \sup_{\taut \in [0,1], |\eta| \leq a_n}n^{-1} \sum_{i=1}^n \bigg \vert \frac{\partial}{\partial \eta}\, l_i(\theta_i, \qhat_i(\eta,\taut)) \bigg \vert~ \text{ and,} \\ 
& D_{n,2} = \sup_{\taut \in [0,1], |\eta| \leq a_n}n^{-1}  \sum_{i=1}^n \bigg \vert \frac{\partial}{\partial \taut}\, l_i(\theta_i, \qhat_i(\eta,\taut)) \bigg \vert.
\end{align*}
Thus, for the grid $\nett_n=\nett_{n,1} \otimes \nett_{n,2}$ we have for any pair $(\eta,\taut)$ with $|\eta| \leq a_n$ and $\taut \in [0,1]$:
$$ \inf_{(\eta_k,\taut_l)\in \nett_n} \, \vert L_n(\thetab,\qhatb(\eta,\taut))-L_n(\thetab,\qhatb(\eta_k,\taut_l)) \vert \leq D_{n,1} \delta_{n,1} + D_{n,2} \delta_{n,2} $$ which implies on the set $M_n(\X)=\{[\hat{m}_n(\alpha_1), \hat{m}_n(\alpha_2)] \subseteq [-a_n,a_n]\}$ we have
$$\vert L_n(\thetab,\qhatb(\eta_{n}^{\nett},\tau_{n}^{{\nett}})) - L_n(\thetab,\qhatb(\eta^{\sf DOR}_n, \tau^{\sf DOR}_n)) \vert \leq D_{n,1} \delta_{n,1} + D_{n,2} \delta_{n,2}.$$
By the bounds on the respective partial derivaties it follows that:
$$D_{n,1} \leq 2 \, C_1 \text{ and } D_{n,2} \leq 2 \, C_1 \, C_2 \bigg (\frac{1}{n} \sum_{i=1}^n |Z_i| + \frac{1}{n} \sum_{1=1}^n |\theta_i| + a_n +  C_3\bigg)$$ 
where $C_1$, $C_2$ and $C_3$ are defined in \eqref{eq:constants.temp}. 
\par
Now note that ${n}^{-1} \sum_{i=1}^n |Z_i|\sim N(2\phi(0),n^{-1})$. Thus, by definition of $\nett_n$ we have:\\
$P( D_{n,1}\delta_{n,1} + D_{n,2} \delta_{n,2} > \epsilon \text{ and } M_n(X)) \to 0$ as $n \to \infty$, and \\
$\ex[( D_{n,1} \delta_{n,1} + D_{n,2} \delta_{n,2})I\{M_n(X)\}] \to 0 $  as $n \to \infty$. 
These coupled with Lemma~\ref{lem:sec4.temp.1} provide us the desired result.

\medskip

\section{Proof details for estimators in the class $\mathcal{S}^G$ and the lemmas used in Section~5}\label{append-grand-mean}
\smallskip
\subsection{Proof of Lemma~\ref{lem.mean.1}} 
By a first order Taylor series expansion we have: 
$$\lt_i(\theta_i,\qhatg_i(\tau))=l_i(\theta_i,\qhatg_i(\tau))+ a_i \, (\Xm - \Tm) \bigg[ \frac{\partial}{\partial \eta} G(\sigma_{f,i}^{-1/2}(\qhat_i(\tau)+(1-\alpha_i) \eta -\theta_i),\bt)\bigg]_{\eta=\mu_i}$$
where $\mu_i$ lies between $\Xm$ and $\Tm$ and $a_i= \sigma_{f,i}^{1/2}\,(b_i+h_i)$. Again, based on the definition of $G$ from Equation~\eqref{eq:loss-function} we have for any $\tau \geq 0$ and any $\eta \in \Real$: 
$$\bigg| \frac{\partial}{\partial \eta} G\big(\sigma_{f,i}^{-1/2}(\qhat_i(\tau)+(1-\alpha_i) \eta -\theta_i),\bt\big) \bigg| \leq \sigma_{f,i}^{-1/2} \text{ for all } i =1, \ldots, n.$$
Thus, for the multivariate versions we have: 
$$\sup_{\tau \in [0,\infty]} |L_n(\thetab,\qhatgb(\tau))-\Lt_n(\thetab,\qhatgb(\tau))| \leq |\Xm-\Tm| \cdot \frac{1}{n}\sum_{i=1}^n (b_i + h_i)$$
which converges to $0$ in $L^1$ as $\Xm \sim N(\Tm,n^{-1})$ and $n^{-1}\sum_{i=1}^n (b_i + h_i)$ is bounded by Assumption A1. This completes the proof.
\subsection{Proof of Lemma~\ref{lem.mean.2}}
From the description of the ARE procedure in Section~\ref{ARE.intro}, recall that,
$\That_{i}(X_i-\eta,\tau)=\ex\{\That_{i}(X_i-\eta,\Z,\tau)\}$ where the expectation is over $\Z$ which is independent of $X$ and follows $N(0, I_n)$ distribution. And,
\begin{align*}
\That_{i}(X_i-\eta,Z_i,\tau)= \left \{ \begin{array}{cl} 
- \bt_i \, U_i(\eta,\tau) & \text{ if } V_i(\eta,\tau)< -\lambda_n(i)\\[1ex] 
\Sm_i(U_i(\eta,\tau)) & \text{ if } |V_i(\eta,\tau)| \leq \lambda_n(i) \\[1ex]
(1-\bt_i) \, U_i(\eta,\tau) & \text{ if } V_i(\eta,\tau)  > \lambda_n(i)
\end{array}\right. \text{ for } i=1,\ldots,n
\end{align*}
where the threshold parameter defined in \eqref{thr.definition} and 
$$U_i(\eta,\tau)=c_i(\tau)+d_i(\tau) (X_i - \eta +\sigma_{p,i}^{1/2} Z_i), V_i(\eta,\tau)=c_i(\tau)+d_i(\tau) (X_i-\eta-\sigma_{p,i}^{1/2} Z_i)$$
with $|d_i(\tau)|$ is less than $\sigma_{f,i}^{-1/2}$. $\Sm_i(U_i(\eta,\tau))$ is a truncated version of
\begin{align*}
S_i(U_i(\eta,\tau))=&G(0,\bt_i)+G'(0,\bt_i) U_i(\eta,\tau) \\
& + \phi(0) \sum_{k=0}^{K-2} \frac{(-1)^{k}H_k(0)}{(k+2)!} \big(2\sigma_{p,i} d_i^2(\tau)\big)^{\frac{k+2}{2}} H_{k+2}\bigg(\frac{U_i(\eta,\tau)}{\sqrt{2 \sigma_{p,i} d^2_i(\tau)}}\bigg).
\end{align*}
So, the derivative exists almost everywhere and $\text{for  all } i=1,\ldots,n$:
\begin{align}\label{eq:deri}
\bigg|\frac{\partial}{\partial \eta}\That_{i}(X_i-\eta,Z_i,\tau)\bigg| \leq \left \{ \begin{array}{cl} 
\sigma_{f,i}^{-1/2} |\Sm'_i(U_i(\eta,\tau))| & \text{ if } |V_i(\eta,\tau)| < \lambda_n(i) \\[1ex]
\sigma_{f,i}^{-1/2} & \text{ if } |V_i(\eta,\tau)|  > \lambda_n(i)
\end{array}\right. .
\end{align}
Noting that for Hermite polynomials of order $k$ the derivative satisfies: $H'_k(x)=kH_{k-1}(x)$, we have  
$  \Sm'_i(U_i(\eta,\tau))$ exempting the two discontinuity points is either $0$ or given by:
\begin{align}\label{mean.ex.1}
\sigma_{f,i}^{-1/2} \bigg\{ G'(0,\bt_i) + \phi(0) \sum_{k=0}^{K-2} \frac{(-1)^{k}H_k(0)}{(k+1)!} \big(2\sigma_{p,i} d_i^2(\tau)\big)^{\frac{k+1}{2}} H_{k+1}\bigg(\frac{U_i(\eta,\tau)}{\sqrt{2 \sigma_{p,i} d^2_i(\tau))}}\bigg)\bigg\}.
\end{align}
Define, $\thai=c_i(\tau)+d_i(\tau) (\theta_i-\eta)$ and $\lamdai=(1+\sqrt{2\sigma_{p,i}}/\gamma_i) \lambda_n(i)$. Now, by exactly following the proof technique used in Lemma~\ref{lem:variance.bound}, it can be shown that:
\begin{align}\label{mean.ex.2}
\sup_{\tau \geq 0} \quad \sup_{|\thai| \leq  \lamdai} a_n^2 \, \ex_{\thai} \big\{\Sm'_i(U_i(\eta,\tau))\big\}^2 = \smallo{n} \text{ as } n \to \infty,
\end{align}
where $\gamma_i$ is defined below Equation~\eqref{thr.definition} and the expectation is over the distribution of $U_i(\eta,\tau)$ which follows $N(\thai, 2 d_i^2(\tau) \sigma_{p,i})$.
 So, by \eqref{eq:deri} for all values of $\eta$, $\tau$ and $\theta_i$ such that $|\thai| \leq \lamdai$ we have:
$$ a_n^2 \, \ex_{\theta_i} \left(\frac{\partial}{\partial \eta}\That_{i}(X_i-\eta,Z_i,\tau)\right)^2 = \smallo{n} $$
where the expectation is over the joint distribution of $X_i$ and $Z_i$. 

\par
We now concentrate on all values of $\eta$, $\tau$ and $\theta_i$ such that $|\thai| > \lamdai$. For this note that for all large $n$,  $|\Sm'_i(U_i(\eta,\tau)| \leq n $ a.e. It follows as: $\Sm_i(U_i(\eta,\tau)$ is truncated above $\pm n$ and so by definition of derivative we have $|\Sm'_i(U_i(\eta,\tau)| \leq n$ for all $|U_i(\eta,\tau)| \geq 1$; and when $|U_i(\eta,\tau)| < 1$, using uniform approximation bounds \citep{gabor1939} on the Hermite polynomials in the expression \eqref{mean.ex.1}, we have  $|\Sm'_i(U_i(\eta,\tau)| \leq n$. 
So, using \eqref{eq:deri} for all values of $\eta$,$\tau$ and $\theta_i$ such that $|\thai| > \lamdai$ we have the following upper bound: 
\begin{align*}
&  \ex_{\theta_i} \bigg(\frac{\partial}{\partial \eta}\That_{i}(X_i-\eta,Z_i,\tau)\bigg)^2 \leq \sigma_{f,i}^{-1}\bigg \{n^2 P\big( |V_i(\eta,\tau)|< \lambda_n(i)\big) + P\big( |V_i(\eta,\tau)| > \lambda_n(i)\big) \bigg\}
\end{align*}
where $V_i(\eta,\tau)$ has $N(\thai, 2 d_i^2(\tau) \sigma_{p,i})$ distribution.  
Also  from Lemma~\ref{lem:large.deviation.bound} we know that: 
$ \sup_{|\thai| > \lamdai} a_n^2 n^2 P( |V_i(\eta,\tau)|< \lambda_n(i))=\smallo{n}$ which produces the desired bound for  $|\thai| > \lamdai$. Thus, we have:
$$\sup_{i} \sup_{\tau \in [0,\infty]} a_n^2\, \ex_{\theta_i} \bigg(\frac{\partial}{\partial \eta}\That_{i}(X_i-\eta,Z_i,\tau)\bigg)^2 = \smallo{n}.$$
As $\That_{i}(X_i-\eta,\tau)=\ex\{\That_{i}(X_i-\eta,\Z,\tau)\}$,  by Jensen's inequality we have:
$$ \ex_{\theta_i} \bigg(\frac{\partial}{\partial \eta}\That_{i}(X_i-\eta,Z_i,\tau)\bigg)^2 \geq \ex_{\theta_i} \bigg\{\ex\bigg(\frac{\partial}{\partial \eta}\That_{i}(X_i-\eta,Z_i,\tau)\bigg|X_i\bigg)\bigg\}^2 $$
and the result of the lemma in terms of  $\That_{i}(X_i-\eta,\tau)$  follows.
\medskip

\section{Data-Informed Numerical Experiments on the Newsvendor Problem}\label{real-data}
\input{real-data}

\section{Auxiliary Lemmas}
\input{misc-lemmas}

\section{Glossary}
\input{glossary}

%% file: real-data.tex

In this section, we conduct numerical experiments based on sales data from the books department of an online retailer. We apply our proposed asymptotic risk estimation methodology  to estimate the inventory levels that retailer's should stock to optimize their future sales and operational costs. Before proceeding further, we present a brief introduction to the newsvendor problem studied here as well as its associated statistical challenges.
\par
\textit{Brief Introduction to the Newsvendor problem.} The newsvendor problem appeared in \cite{edgeworth-1888} in connection with optimizing cash reserves in a bank. Based on a subsequent formulation in \cite{arrow1951optimal}, the classical inventory theory  has been developed assuming the demand distribution is  {known in advance}, and the optimal solution is the newsvendor quantile of the underlying demand distribution \citep{karlin_scarf_58}. In contrast, here we work in a predictive setup where  the demand distribution is unknown \pr{and must be}  estimated from past data.
Within the inventory literature, when the information on the demand distribution is not available, the most common approach is the use of Bayesian updates \citep{azoury_85,lariviere1999}. Under this approach, the inventory manager has limited access to demand information; in particular, she knows the family of distributions to which the underlying demand  belongs, but she is uncertain about its parameters. 
She has an initial prior belief regarding the uncertainty of the parameter values, and this belief is continually updated based on historical realized demands by computing posterior distributions. 
Bayesian updates assume that the distribution of the prior belief is known and given in advance so that Bayesian updates can be computed explicitly.  In contrast, in our paper, we assume that the prior distribution of $\theta_i$ is $N(\eta,\tau)$ for all $i$, but the parameters $\eta$ and $\tau$ are {\em unknown} and must be estimated from data.   Note that when $\eta$ and $\tau$ are unknown, traditional Bayesian updates cannot be computed.
\par
\textit{Data-informed Numerical Experiments.} We consider books sold by the online retailer during the month of June, 2005. During that month 1.25 million books spanning across 274,558 book titles were sold. There was a wide variability in the sales across the booktiles. The average number of books sold per title was $5.06$ and the standard deviation was $74.47$. Around 47.4\% of book titles had only one book sold in that month. In practice, separate models are used to model different categories of books but joint modeling is done for each category. To simplify our analysis here we only consider a selection of $200$ book titles. Each of these book titles had at least $10$ sales. Our group consisted of 25\% book titles with very high volume sales (over 100 books sold per month). The rest, 75\% of the items in the group were moderately (M) sold with sales lying between 10 and 30.
\par
\textit{Our Model.} We consider a one-step ahead prediction model for the multivariate newsvendor problem. Based on observing the past two months sales data $\X_1$ and $\X_2$ and we would like to select an inventory level $\qhatb$  that would minimize the next month's cumulative lost sales and holding costs across the $200$ book titles. We use the retail data to judge the values of the true demands, which along with the associated prices of the books are used to form the data-informed parameter based on which our simulation experiments were set. In our numerical experiments we set the true demand $\thetab$ of the book titles  equal to the monthly sales observed in retail data. It should be noted that $\thetab$  no longer obeys any normal prior structure but is governed by the form of the sales data. 
The selling prices $\bm{p}$ of the books are calculated based on their corresponding average unit price over the month. Now, based on our Gaussian predictive model \eqref{pred.model} we generate the past sales data $\X_1, \X_2$ and the future month's demand $\Y$ based on the model:
$$\X_1, \X_2 \text{ and } \Y \stackrel{iid}{\sim} N(\thetab, \s I),$$
where, $\s$ is assumed to be known and is set equal to variation of sales in the data.  
For each product $i$, we assume that each unit of inventory incurs a holding cost $h_i > 0$, and each unit of \pr{lost} sale incurs a cost of $b_i  > 0$.  When we estimate the future demand $Y_i$ by $\qhat_i$, the loss corresponding to the ${i}^{\text{th}}$ product is $b_i \cdot  (Y_i-\qhat_i)^+  + h_i \cdot  (\qhat_i-Y_i)^+$.
We would like to minimize the total cost associated with the $n=200$ stocking quantities. Based on $\X_1$ and $\X_2$, we make estimates of the stocking quantities based on the (a) canonical unshrunken (US) linear estimator   (b) EB-MLE which is close to the James-Stein direction of shrinkage in this homeskedastic case (c) ARE based methodology. 
\par
\textit{Evaluation and Results.}
Given any inventory estimate $\qhatb(\x_1,\x_2)$ based on the past data, in our simulation experiment we evaluate its performance by the future loss on $\bm{y}$:  
$$
\hat{L}_n(\qhatb) = \frac{1}{n}\sum_{i=1}^n  b_i   (y_i-\qhat_i(\x_1,\x_2))^+  + h_i  (\qhat_i(\x_1,\x_2)-y_i)^+ ~.
$$
The simulation experiment is conducted for $50$ independent repetitions and the percentage reduction in loss over the cannonical unshrunken estimator is calculated. In Table~\ref{table-real-data} we provide the relative efficiency of the EBMLE and ARE based methods compared to the canonical non-shrunken estimator:
\begin{align*}
&\text{ARE efficiency over US} = (L_n(\thetab,\qhatb^{\rm US}))-L_n(\thetab,\qhatb^{\rm ARE}))/L_n(\thetab,\qhatb^{\rm US})) \\
&\text{EBMLE efficiency over US} = (L_n(\thetab,\qhatb^{\rm US}))-L_n(\thetab,\qhatb^{\rm EBMLE}))/L_n(\thetab,\qhatb^{\rm US})).
\end{align*}
Figure~\ref{fig-real-data} shows the plot of these efficiency measures.
The difference between the two aforementioned efficiency measures is reported in the table as \textit{ARE efficiency over EBMLE}. The summary statistics of the distribution of these efficacy measures over the $50$ independent simulation experiments are also reported along with the p-value of the Wilcoxon singed rank test for the alternative hypothesis that the efficacy levels are indeed positive.   
It was seen that the EBMLE method on avergage produced inventory levels much worse than the unshrunken estimator. This is due to the fact that the EBMLE method overestimated the magnitude of shrinkage. The ARE based inventory estimation strategy had on average around 7\% better relative prediction error than the EBMLE based method and around 3\% better error rates than the unshrunken estimator. In both the comparisons, we had significant p-values in the Wilcoxon test clearly proving on average better performance of the ARE based method in the $50$ experiments. Around 92\% of the cases, the ARE based method is better than the  EBMLE and around 72\% of the cases it is better than the unshrunken estimator.

\begin{table*}[ht!]
	\caption{Comparison of the performances of $\A$, EBMLE and the unshrunken (US) estimators. The summary statistics of their relative efficiency (in \%) over 50 independent simulation scenarios are reported. We also report the p-value based on Wilcoxon Signed Rank test with the alternative hypothesis being that the reported relative efficiency is positive.}
	\label{table-real-data}
	\begin{tabular}{|c||c|c|c|c|c|c||c|}
		\hline
		Relative  & \multicolumn{6}{c||}{ Summary Statistics } &  \multicolumn{1}{c|}{ P value for  } \\[1ex]
		\cline{2-7}  
		Efficiency (in \%) & Min. & $Q_1$ &  Median &    Mean &  $Q_3$ &     Max. & $H_A: \rm{Efficiency} > 0$ \\
		\hline & &  & & &  & & \\
		ARE over EBMLE &  -10.56   & 2.94  &  6.86  &  7.15 &   9.97 &  26.60 &  $2 \cdot 10^{-08}$\\ 
		\hline & &  & & &  & & \\
		ARE over US &  -15.42 &  -0.73 &   3.43 &   2.83&    8.34&   15.53 & 0.0031\\
		\hline & &  & & &  & & \\
		EBMLE over US & -42.02  &-10.07  & -1.77 &  -4.32 &   2.17 &  24.81 & 0.9868\\
		\hline
		\hline
	\end{tabular}
\end{table*}
We end this section by discussing the choice of weights $b_i$ and $h_i$ that we used here. The weights are determined by the internal revenue mechanism and business policy of the retailer and were not available with the data. Hence, we made a judicious choice for the weights based on common business practice. Consider a book title whose procurement cost is \$1 per unit to the retailer. Here, we assume that the retailer reviews the system and replenishes its inventory once a month, and sells the product
at \$(1 + m) per unit, where m represents the mark-up. 
The lost sales cost $b_i$ in this case is at least \$m per unit and it does not  include any loss in customer goodwill (reputation costs) due to unfulfilled demand. Hence, for a book whose selling price is $p_i$, we have $b_i \geq m/(1+m) p_i$. Here, we assume a uniform 15\% mark up on all books which is quite common in many retail environments.  
Unless there is depreciation in the product, the holding cost $h_i$ per unit per period is simply the cost of holding \$1 in the inventory for a month. Hence, at a cost of capital of 15\% per year, we have $h_i=b_i\cdot 0.15/12$ per unit per period. Following the aforementioned logic, the weights for the moderately sold book titles are set. However, the book titles in very high volume sales category need special attention. Usually, more sensitivity is associated with items having high volume of transactions and customer loyalty is encouraged through quicker services and discounts to ward off competitors. In this context, apart from the above mark-up dependent costs, we also attach a flat rate for the reputation  and depreciation costs (with the ratio still being much larger than 1) for the high volume category. The data and the code used for these experiments can be downloaded from \url{http://www-bcf.usc.edu/~gourab/inventory-management/}.


\begin{figure}[ht!]
	\includegraphics[width=\textwidth]{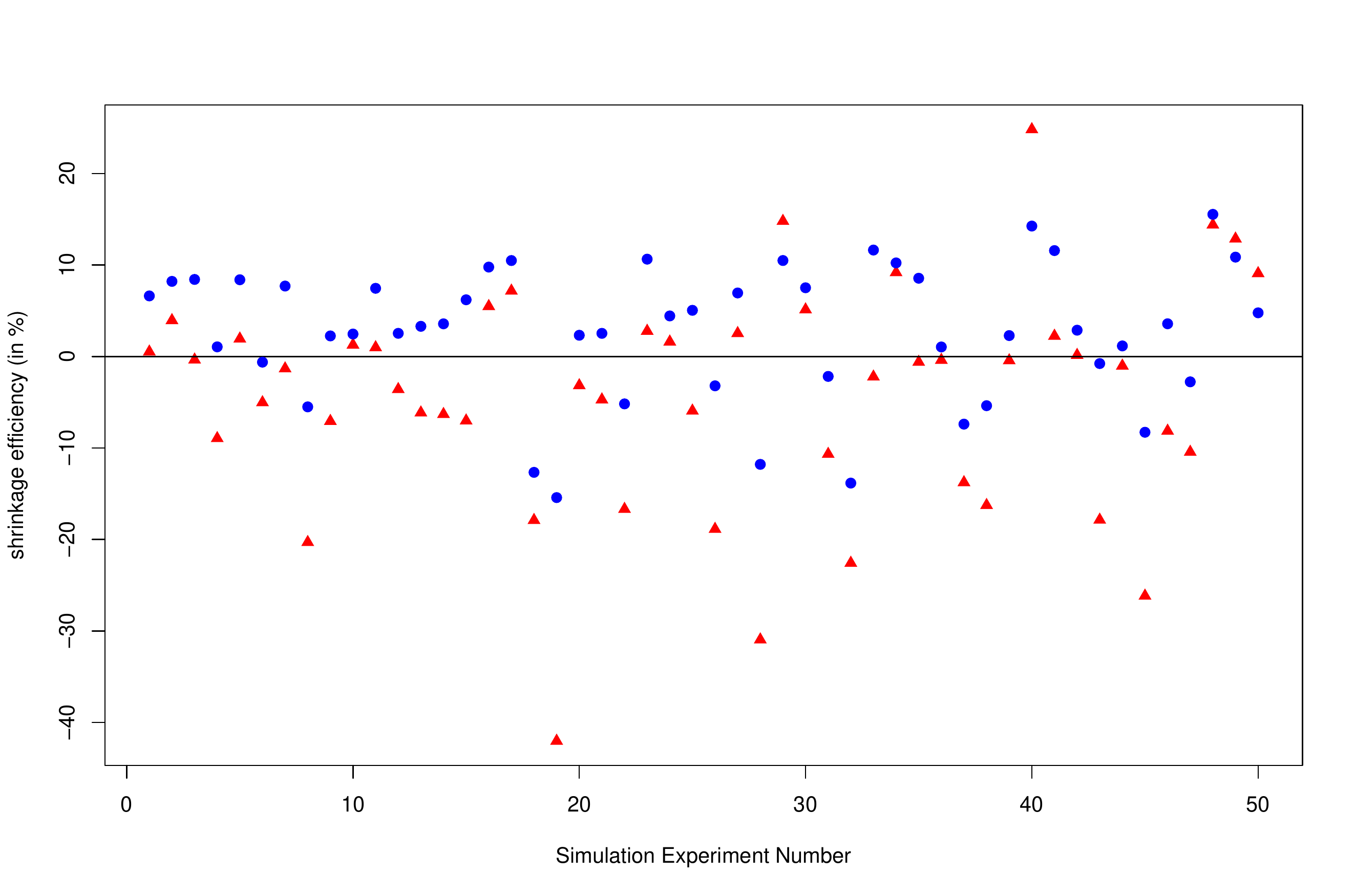}
	\caption{In red and blue we have respectively the efficiency of the EBMLE and our proposed ARE based method relative to the unshrunken linear estimator.}\label{fig-real-data}
\end{figure}

%% file: misc-lemmas.tex
\smallskip
The following lemma provides an upper bound on Hermite polynomial. 

\begin{lem}\label{lem:hermite-bound}
	There is an absolute constant $c$ such that for all $k \geq 1$ and $x \in \Real$,
	$$
	\abs{ H_k(x) } \leq 
	c\, e^{x^2/4} \frac{  k!   }{ k^{1/3} \, {\left( {k / e} \right)}^{k/2} }.  
	$$
\end{lem}

\textit{Proof of Lemma~\ref{lem:hermite-bound}}\label{append:hermits-bound}

\citeasnoun{krasikov2004} shows that for $k \geq 6$,
$$
(2 k )^{1/6} 2^k   \max_{x \in \Real}  (H_k( x) )^2 e^{-x^2/2} ~\leq ~  \frac{2}{3} C_k  \exp\left(  {15 \over 8} \left( 1 + \frac{12}{4(2k)^{1/3} - 9} \right)\right)~,
$$
where 
$$
C_k = \left\{
\begin{array}{ll}
\frac{ 2 k \sqrt{4k - 2}\,\, (k!)^2}{\sqrt{8k^2 - 8k + 3}\; ((k/2)!)^2}, \quad & \text{ if } \quad k \text{ is even},\\
& \\
\frac{ \sqrt{16k^2 - 16k +6} \,\, k! (k-1)!}{\sqrt{2k-1} \;((k-1)/2)!^2}, \quad & \text{ if } \quad k \text{ is odd},\\
\end{array}
\right .
$$
Note that for $k \geq 6$,
$
\frac{2}{3}   e^{  {15 \over 8} \left( 1 + \frac{12}{4(2k)^{1/3} - 9} \right)}
$
is decreasing in $k$ and 
\bex
C_k 
&\leq& \left\{
\begin{array}{ll}
	2 \sqrt{k} \times \frac{(k!)^2}{((k/2)!)^2}, \quad & \text{ if } \quad k \text{ is even},\\
	& \\
	4 \sqrt{k} \times\frac{ k! (k-1)!}{((k-1)/2)!^2}, \quad & \text{ if } \quad k \text{ is odd},\\
\end{array}
\right .\\
&\leq& \left\{
\begin{array}{ll}
	2 \sqrt{k} \times \frac{(k!)^2}{\left(\sqrt{2 \pi} {\left( k \over 2 \right)}^{(k+1)/2} e^{-k /2} \right)^2}, \quad & \text{ if } \quad k \text{ is even},\\
	& \\
	4 \sqrt{k} \times \frac{ (k!)^2}{k \left( \sqrt{2 \pi}  {\left( {k-1 \over 2} \right)}^{k/2}  e^{-(k-1) /2} \right)^2}, \quad & \text{ if } \quad k \text{ is odd},\\
\end{array}
\right .\\
&=& \left\{
\begin{array}{ll}
	2 \sqrt{k} \times \frac{(k!)^2}{{k \over 2}  (2\pi) \left( {\left( k \over 2 \right)}^{k/2} e^{-k /2} \right)^2}, \quad & \text{ if } \quad k \text{ is even},\\
	& \\
	4 \sqrt{k} \times \frac{ (k!)^2}{k  e (2 \pi) \left(  {\left( {k \over 2} \right)}^{k/2}  \times \left(1 - {1 \over k}\right)^{k/2} \times e^{-k/2} \right)^2}, \quad & \text{ if } \quad k \text{ is odd},\\
\end{array}
\right . \\
&=& \left\{
\begin{array}{ll}
	4 \sqrt{k} \times \frac{(k!)^2}{ k (2\pi) \left( {\left( k \over 2 \right)}^{k/2} e^{-k /2} \right)^2}, \quad & \text{ if } \quad k \text{ is even},\\
	& \\
	4 \sqrt{k} \times \frac{ (k!)^2}{k  e (2 \pi)  \left(1 - {1 \over k}\right)^{k}    \left(  {\left( {k \over 2} \right)}^{k/2}  e^{-k/2} \right)^2}, \quad & \text{ if } \quad k \text{ is odd},\\
\end{array}
\right .
\eex
where the second equality follows from Sterling's bound.  It is easy to verify that $(1-{1 \over k})^k$ is increasing in $k$ and approaches $1 \over e$ as $k$ approaches infinity.
Putting everything together, we conclude that 
there is an absolute constant $a_1$ such that for all $k \geq 6$ and $x \in \Real$,
$$
\abs{ H_k(x) } \leq a_1 e^{x^2/4} \frac{  k!   }{ k^{1/3} \, 2^{k/2} \, {\left( {k \over 2 e} \right)}^{k/2} }  = a_1 e^{x^2/4} \frac{  k!   }{ k^{1/3} \, {\left( {k \over  e} \right)}^{k/2} }  
$$
Since the results hold for $k \geq 6$, it is easy to verify that it also holds for all $k \geq 1$, by choosing appropriately large constant $a_1$.

\begin{lem}\label{lem:Cai2}
If $X \sim N(\theta,1)$ then
$$\ex H_k^2(X) \leq k^k (1+ \theta^2/k)^k.$$
\end{lem}

\text{Proof.} See Lemma~3 of \citeasnoun{cai2011}.

\begin{lem}\label{lem:var.bound.1}
If $Y$ and $I_A$ are independent random variables then:
\begin{itemize}
\item $\var(YI_A)=\var(Y) P(A) + (\ex[Y])^2 P(A) P(A^c)$
\item $\var (YI_A) \leq \ex [Y^2]\, P(A)$
\item $\var(YI_A) \leq \var(Y) + (\ex[Y])^2P(A^c)$
\end{itemize}
\end{lem}
\text{Proof.}
Using the independence between $Y$ and $I_A$ we have $\var(YI_A)$ equals 
$\ex [Y^2] P(A) - (\ex[Y] P(A))^2 = \var(Y) P(A) + (\ex[Y])^2 P(A)  - (\ex[Y] P(A))^2$
from which we have the identity in the lemma. The inequalities immediately follow from it.  

\begin{lem}\label{lem:mills.ratio}
Mills Ratio and Gaussian Tails: For any $a > 0$ we have:
$$ -a^{-3} \phi(a) \leq \Phit(a) - a^{-1} \phi(a) \leq 0.$$
\end{lem}

\text{Proof.} See Exercise 8.1, Chapter 8 in \citet{Johnstone-book}.

\begin{lem}\label{lem:G.bound}
	For any $w \in \Real$ and $b \in (0,1)$,  let $G(w,b)$ be defined as in Equation~\eqref{eq:loss-function}. Then,
	$G(w,b) \leq \phi(0)+ \max\{1-b\,,\,b\} |w|$.
\end{lem}
\begin{proof}
By definition $G(w,b)=\phi(w) + w \Phi(w) - b w$.  Since $\phi(w) \leq \phi(0)$ for all $w$, the result follows.
\end{proof}

\begin{lem}\label{lem:ui}
	Extra Integrability condition. If family $\{X_t: t \in T\}$ is such that $\sup_{t \in T} \ex |X_t|^{1+\delta} < \infty$ for some $\delta > 0$ then $\{X_t: t \in T\}$ is uniformly integrable. 
\end{lem}

\text{Proof.} See \citet{Billingsley-book}.

\begin{lem}\label{lem:inequality}
	For any fixed $m > 0$ we have
	$$ \bigg(1+\frac{m}{k}\bigg)^k \leq e^m \text{ for all } k \geq 1$$
\end{lem}

\text{Proof.} We know that for any $x > 0$, $\log (1+x) \leq x$ and taking logarithm and dividing both sides by $m$ the statement in the lemma reduces to 
$\log(1+m/k)\leq m/k$.\\[2ex]
The following well known random variable lemmas have been used in our proofs.

\begin{lem}\label{lem:Cai}
For random variables $W_1, \ldots, W_n$ we have:
$$
	\E \left[ \left( \sum_{i=1}^n W_i \right)^2 \right]  \leq \left(  \sum_{i=1}^n \sqrt{ \E (W_i^2) }  \right)^2
$$
\end{lem}

\begin{lem}\label{lem:modified.markov} For any random variable $X$ and $\lambda >0$, we have
	$$\ex \{X I\{X \geq \lambda\} \} \leq |\lambda|^{-1} \ex X^2.$$  
\end{lem}


\begin{lem}\label{lem:var.bound.2}
	For any finite $l \geq 1$ we have
	$$\var\bigg(\sum_{i=1}^l X_i\bigg) \leq 2^{l-1} \sum_{i=1}^l \var(X_i).$$ 
\end{lem}

%% file: glossary.tex

In Table~\ref{table-glossary}, we briefly list the notations that have been \pr{used} repeatedly in the current \pr{paper. As a convention}, multivariate vectors, expressions and estimates are represented in bold. \\[-3.5ex]
\begin{table*}[h]
\caption{List of important notations used in the current paper.}\label{table-glossary}
\begin{tabular}{|l|l|}
\hline
\textsf{Notation} & \textsf{Description} \\
\hline
$n$ & dimension of the problem\\
$i$ &  coordinate index\\
$a_n$ & $\log \log n$ \\
$\theta_i$ & {\em unknown} mean of coordinate $i$\\
$\sigma_{p,i}$ & variance of the past data of coordinate $i$\\
$\sigma_{f,i}$ & variance of the future data of coordinate $i$\\
$X_i$ & past  data for coordinate $i$ with $X_i  \sim N(\theta_i, \sigma_{p,i})$ \\
$Y_i$ & future data of coordinate $i$ with $Y_i  \sim N(\theta_i, \sigma_{f,i})$ \\
$b_i$ & per-unit underestimation cost associated with coordinate $i$\\
$h_i$ & per-unit overestimation cost associated with coordinate $i$\\
$\net$ & Grid over which ARE criterion is optimized\\
$\bt_i$ &  the critical ratio $b_i / (b_i+h_i)$ of the weights of the check loss\\
$l_i(\theta_i, \qhat_i(\x))$ & loss associated with coordinate $i$ under the policy $\qhat_i$ when $\X = \x$ is observed \\
$L_n(\thetab, \qhatb)$ & average loss over $n$ coordinates under the prediction policy $\qhat$ \\
$r_i(\theta_i, \qhat_i)$ & risk associated with coordinate $i$ under the prediction policy $\qhat_i$\\
$R_n(\thetab, \qhatb)$ &  average risk of $n$ coordinates under the prediction policy $\qhat$\\
$(\eta,\tau)$ & hyperparameters for the prior distribution of $\theta_i$, with $\theta_i \sim N(\eta, \tau)$\\
$\qhat_i^{{\sf Bayes}}(\eta,\tau)$ & bayes estimate for $N(\eta,\tau)$ prior; see \eqref{eq:linear.est} for definition\\
$\alpha_i(\tau)$ & shrinkage factor in our estimates, with $\alpha_i(\tau) = \tau / (\tau + \sigma_{p,i})$ for all $i$\\
$\mathcal{S}^0$ & class of shrinkage estimators  $\qhatb(\tau)$ based on origin-centric priors\\ 
$\mathcal{S}^G$ & class of shrinkage estimators $\qhatb^G(\tau)$  based on grand-mean centric priors \\
$\mathcal{S}$ & class of data driven shrinkage estimators $\qhatb(\eta,\tau)$\\
$\Ahat$ & our proposed estimate of the risk function $R_n(\thetab, \qhatb(\tau))$ of estimators in $\mathcal{S}^0$\\
$\AGhat$ & our proposed estimate of the risk function $R_n(\thetab, \qhatb^G(\tau))$ of estimators in $\mathcal{S}^G$\\
$\AD$ & our proposed estimate of the risk function $R_n(\thetab, \qhatb(\eta,\tau))$ of estimators in $\mathcal{S}$\\
$\tau_{n}^{{\sf OR}}$,  $\tau_{n}^{{\sf GOR}}$ & Oracle estimates of the hyperparameter $\tau$ for $\mathcal{S}^0$ and 
$\mathcal{S}^G$, respectively\\
$\tauh^{\A}_n,\,\tauh^{\AG}_n $ & ARE-based estimate of $\tau$ for  $\mathcal{S}^0$ and $\mathcal{S}^G$, respectively \\
$(\eta_{n}^{{\sf DOR}}, \tau_{n}^{{\sf DOR}})$ & oracle estimates of the hyperparameter in $\mathcal{S}$; ARE estimates are $\etah^{\ADD}_n, \tauh^{\ADD}_n$ \\
$\That_i(X_i,\tau)$ & coordinate-wise estimator used in our risk estimation method; see \eqref{def.A}\\
$\thr(i)$ & threshold parameter used in our risk estimation method; see \eqref{thr.definition}\\
$K_n(i)$ & truncation parameter used in our risk estimation method; see \eqref{eq:trun.parameter}\\
$G(\omega,\beta)$ &  describes the predictive loss function; see \eqref{eq:loss-function}\\
$\bigo{\cdot}$,\;$\smallo{\cdot}$ & denote the Big O and the little-o mathematical notations, respectively\\
\hline
\end{tabular}
\end{table*}